\documentclass[english,a4paper,11pt,oneside,reqno]{amsart}

\usepackage{graphicx}
\usepackage[utf8]{inputenc}
\usepackage{amscd,amsfonts,amsmath,amssymb,amsthm}
\usepackage{latexsym}
\usepackage{bbm}
\usepackage[english]{babel}
\usepackage{marvosym}
\usepackage{color}
\usepackage{varwidth}
\usepackage{tikz}
\usetikzlibrary{decorations.pathreplacing,decorations.markings}
\usetikzlibrary{arrows.meta}
\usepackage{enumitem}

\usepackage{float}
\usepackage{caption}
\usepackage{subcaption}
\usepackage{pgfplots}
\usepackage{graphics}
\usepackage{grffile}
\pgfplotsset{compat=newest}

\usepackage{hyperref}
\usepackage[margin=1.2in]{geometry}

\newcommand{\hhookrightarrow}{\lhook\mkern-3mu\relbar\mkern-12mu\hookrightarrow}

\newtheorem{theorem}{Theorem}
\theoremstyle{definition} 
 
\theoremstyle{definition} 
\newtheorem{proposition}[theorem]{Proposition} 
\theoremstyle{definition} 
 
\theoremstyle{definition} 
\newtheorem{remark}[theorem]{Remark}
\theoremstyle{definition} 
\newtheorem{corollary}[theorem]{Corollary} 
\theoremstyle{definition}
 
\theoremstyle{definition} 
 
\theoremstyle{definition} 

\newcommand{\coleq}{\mathrel{\mathop:}=}
\newcommand{\lm}[1]{#1_0}
\newcommand{\st}[1]{#1_{\mathrm{st}}}
\newcommand{\sto}[1]{#1_{\mathrm{st},1}}
\newcommand{\stt}[1]{#1_{\mathrm{st},2}}
\newcommand{\sti}[1]{#1_{\mathrm{st},i}}
\newcommand{\de}[1]{#1_{\mathrm{de}}}

\DeclareMathOperator*{\argmin}{\mathrm{argmin}}

\title[On hysteresis-reaction-diffusion systems]{On hysteresis-reaction-diffusion systems: \\
Singular fast-reaction limit derivation and nonlinear hysteresis feedback.}
\author{Klemens Fellner, Christian M\"unch}
\address{Klemens Fellner \hfill\break
Institute of Mathematics and Scientific Computing, University of Graz, Heinrichstra{\ss}e 36, 8010 Graz, Austria}
\email{klemens.fellner@uni-graz.at}
\address{Christian M\"unch \hfill\break
	Technical University of Munich, Boltzmannstra{\ss}e 3, 85748 Garching bei M\"unchen, Germany}
\email{christian.muench@ma.tum.de}

\begin{document}

\begin{abstract}
This paper concerns a general class of PDE-ODE reaction-diffusion systems, 
which features a singular fast-reaction limit towards a reaction-diffusion equation coupled to a scalar hysteresis operator. 

As prototypical application, we present a PDE model for the growth of a population 
according to a given food supply coupled to an ODE for the turnover of a food stock.
Under realistic conditions the stock turnover is much faster 
than the population growth yielding an intrinsic scaling parameter. 
We present two models of consume rate functions 
such that the resulting food stock dynamics converges to a generalised play operator in the associated fast-reaction-limit. 
We emphasise that the structural assumptions on the considered PDE-ODE models are quite general and 
that 
analogue systems might describe e.g. cell-biological buffer mechanisms, where proteins are stored and used in parallel. 

Finally, we discuss an explicit example showing
that nonlinear coupling with a scalar generalised play operator can lead to spatially inhomogeneous large-time behaviour in a kind of hysteresis-diffusion driven instability.
\end{abstract}

\maketitle


\section{Introduction} 
We consider the evolution of a population density $u(t,x)$ depending on time $t\ge0$ and position $x\in \Omega$ within a sufficiently smooth domain $\Omega\subset \mathbb{R}^d$, $d\in \{1,2,3\}$.
The total number of individuals at any time $t\ge0$ is therefore given by
\[
N(t)\coleq \int_\Omega u(t,x) dx.
\]
Note the notational convention that small letters like $u(t,x)$ shall be used for spatial densities and individual rates, 
while capital letters like $N(t)$ shall denote total amounts.

The considered population is assumed to have equal access to an external 
food supply described by a given non-negative function $F(t)\ge0$.
For an arbitrary time $T>0$, the dynamics of the population is modelled by the following (nonlinear) PDE:
\begin{align}
\partial_t u -D\Delta u &= \lambda(N,F,S)\, u && \text{ in } [0,T]\times\Omega,\label{pop_dnymics_pop_evol1}\\
\partial_\nu u &= 0 &&\ \text{on } [0,T]\times\partial\Omega,\label{pop_dnymics_pop_evol2}\\
u(0) &= {u_{in}}&&\ \text{in } \Omega,\label{pop_dnymics_pop_evol3}
\end{align}
where $D$ is a diffusion coefficient and $\nu$ the outer unit normal. 
In particular, 
$\lambda(N,F,S)$ is a nonlinear, time-dependent 
population growth rate to be detailed later and which depends besides $N$ and $F$ also on 
the amount of food $S=S(t)$ stored in stock at time $t$. 
Note also that the evolution of total population size $N(t)$ is entailed from \eqref{pop_dnymics_pop_evol1}--\eqref{pop_dnymics_pop_evol3}:
\begin{equation}\label{pop_dnymics_pop_evol4}
\begin{aligned}
\dot{N}(t) &= \lambda N(t) &&\ \text{in } [0,T],   \\
N(0) &= {N_{in} \coleq \int_\Omega u_{in}(x) dx}.
\end{aligned}
\end{equation}

The evolution of the food stock $S(t)$ assumes (idealistically) that the population stores the entire unconsumed food into the stock at all times.
More precisely, 
the time change of the stock shall be the difference between food supply and the total food consumption of the population, which is proportional to a consume rate function $c=c(N,F,S)$ to be detailed later. Additionally, we suppose 
that the stock evolution can equally be written as a sum of a gain and a loss term 
(according to $\dot{S}>0$ or $\dot{S}<0$) which entails an additional constraint on the structure of admissible consume rate functions $c=c(N,F,S)$. Hence, we consider
\begin{align}
\varepsilon \dot{S} &= F-Nc =  {G(c,N,F) - l(c,N,F)S} &&\ \text{ in } [0,T], \label{pop_dnymics_stock_evol1}\\
S(0)&=S_{in}\ge0,\label{pop_dnymics_stock_evol2}
\end{align}
where $\varepsilon\ll 1$ represents the fast time scale of the stock turnover (in comparison 
to the population dynamics). 

The gain term on the right hand side of \eqref{pop_dnymics_stock_evol1} corresponds to 
situations where stock levels are increasing, i.e. $\dot{S}\ge0$ and we are interested in consume rate functions where$(F-Nc)_+$ equals a rate function $G$ with no dependence on current stock level $S$, 
i.e. 
\begin{equation}\label{G}
(F-Nc)_+ = G(c,N,F).
\end{equation}
This fits the assumption that all excess food is stocked independently of the stock level.  

On the other hand, 
the stock loss $(F-Nc)_-$ seems more realistically modelled proportional to the current stock size $S(t)$, i.e. 
\begin{equation}\label{lS}
(F-Nc)_-  = l(c,N,F)\,S,
\end{equation}
where $l$ denotes an averaged usage rate.
Note that assumption \eqref{lS} 
implies that the population will use 
an emtpying stock more carefully and also ensures that no food can be taken from an empty stock. 
Hence mathematically, the stock loss rate \eqref{lS} implies that the stock dynamics 
\eqref{pop_dnymics_stock_evol1} 
satisfies the quasi-positivity property, which 
yields $S(t)\ge0$ for all $t\ge0$  provided that $S_{in}\ge0$. 
Finally, note that at any time $t$, either $G(c(t),N(t),F(t))=0$, $l(c(t),N(t),F(t))=0$ or both functions are zero.
\medskip

The first result of the paper provides global existence of 
strong solutions for general consume and growth rate functions $c$ and $\lambda$ satisfying natural Lipschitz conditions. 


\begin{theorem}\label{Thm:pop_dnymics_existence}
For $T>0$, let the food supply function $F(t)$ be nonnegative in $\mathrm{W}^{1,\infty}(0,T)$ with
\[
F_{\max}:=\max_{t\in [0,T]}\{F(t)\}.
\]
Assume locally Lipschitz continuous consume and growth rate functions $c(S,N,F):\mathbb{R}^+ \times \mathbb{R}^+ \times \mathbb{R}^+ \rightarrow \mathbb{R}^+$ and $\lambda(S,N,F): \mathbb{R}^+ \times \mathbb{R}^+ \times \mathbb{R}^+ \rightarrow \mathbb{R}$, 
which can be extended to continuous and locally Lipschitz continuous functions on $\mathbb{R} \times \mathbb{R} \times \mathbb{R}^+$.
Consider nonnegative initial data $0\le S_{in}$ and $0\le u_{in}
\in \lbrace v\in \mathrm{H}^2(\Omega): \partial_\nu v = 0 \text{ on } \partial\Omega \rbrace$ 
with positive initial population size
\[
N_{in}=\int_{\Omega} u_{in}\, dx>0.
\] 
	
	
Then, for any $\varepsilon>0$ fixed, the PDE-ODE system \eqref{pop_dnymics_pop_evol1}-\eqref{pop_dnymics_stock_evol2} has a unique solution $(u,S)$ with  
\begin{align*}
S&\in \mathrm{C}^1([0,T]), \\
u&\in \mathrm{C}([0,T];\mathrm{H}^2(\Omega))\cap\mathrm{C}((0,T];\mathrm{C}^{2,\beta}(\overline{\Omega}))\cap \mathrm{C}^{1}([0,T];\mathrm{L}^2(\Omega)) \cap \mathrm{C}^1((0,T];\mathrm{C}^{0,\beta}(\overline{\Omega})),
\end{align*}
where the H\"older exponent $0< \beta <1$ is determined by the Sobolev embedding of $\mathrm{H}^{s}(\Omega)$ into $\mathrm{C}^{0,\beta}(\overline{\Omega})$ for arbitrary $s\in (0,2)$.
Moreover, the solutions $(u,S)$ of \eqref{pop_dnymics_pop_evol1}-\eqref{pop_dnymics_stock_evol2} are nonnegative, i.e. 
\[
(u(t,x),S(t))\geq 0, \qquad 	\text{for all } t\in [0,T], \ x\in \overline{\Omega}.
\]
\end{theorem}

\medskip

The main aim of this paper is to rigorously perform the singular fast-reaction limit $\varepsilon\to 0$. 
Indeed, populations reproduce on a time-scale of months yet  stock turnover happens on a time scale of several hours, which implies $\varepsilon =  O(10^{-3})$. 

%
%

The following Theorem \ref{Thm:singular_limit} proves in the limit $\varepsilon\to 0$ that solutions $(u_\varepsilon,S_\varepsilon)$ to system 
\eqref{pop_dnymics_pop_evol1}-\eqref{pop_dnymics_stock_evol2} converge under 
natural choices and assumptions on growth- and consume rate functions
$\lambda$ and $c$ to solutions  $(u_0,S_0)$
to a PDE for $u_0$ coupled to a hysteresis (i.e. generalised play) operator for $S_0$.
\medskip

The key structural assumptions of Theorem \ref{Thm:singular_limit} which are responsible 
for the limiting hysteresis operator concern the 
consume rate functions $c$, in particular the dependency on the stock level $S$
and the individual food supply 
\begin{equation*}
f:=\frac{F}{N}.
\end{equation*} 
More precisely, 
we consider consume rate functions, which implement a 
partition of the $f$-$S$-phase space in terms of upper and lower threshold functions
$U(f)$ and $L(f)$ enclosing the area where $\dot{S}=0$, see Fig. \ref{pop_pic_limit_phase_diagram_1} for a prototypical example. 

The precise definitions of $U$ and $L$ in terms of the modelling of the consume rate function $c$
will be stated in Section \ref{consume}. 
At this point, we only remark that both types of consume rate functions $c_1$ and $c_2$ detailed in Section \ref{consume}
are defined as a concatenation of three areas in the $f$-$S$-phase space corresponding to 
$\dot{S}<0$, $\dot{S}=0$ and $\dot{S}>0$:
The first case details the use of the stock by the population, when the individual food supply $f=\frac{F}{N}$ is small, 
i.e. when the population is large in comparison to the food supply rate $F$. The third case considers the opposite 
case when $f$ is large and the population stores into the stock. Finally, the second case 
refers to medium values of $f$, which leave the stock levels unchanged. 

Note that Theorem \ref{Thm:singular_limit} holds for both types of consume rate functions $c_1$ and $c_2$, 
which are detailed in Section \ref{consume}. However, for consume rate function $c_1$, which allows 
in some situations for unbounded stock levels, it requires 
additional assumptions which ensure that $S(t)$ 
is uniformly bounded by a maximal value $S_{\max}$, see Proposition \ref{c1Smax}.

\begin{theorem}[Singular limit to PDE-hysteresis system]\label{Thm:singular_limit}\hfill\\
Suppose the assumptions from Theorem~\ref{Thm:pop_dnymics_existence}.
Let $c_1$ and $c_2$ be the consume rate functions as defined in Section \ref{consume} and consider 
the corresponding growth rate functions 
$$
\lambda_i(S,N,F)= \left(\frac{c_i(S,N,F)}{c_{\min}}-1\right), \qquad i\in \{1,2\}.
$$
For $c_1$ suppose furthermore the assumptions of Proposition \ref{c1Smax}.
Let $U(f)$ and $L(f)$ be as sketched in Fig. \ref{pop_pic_limit_phase_diagram_1}
and defined in detail in Section \ref{consume}, where the maximal stock level $S_{\max}$ 
is given either by the definition of $c_2$ or in 
Proposition \ref{c1Smax} in the case of $c_1$. 
	
Let $(\lm{u}, \lm{S}, \lm{N})$ and $\lm{f}:=F/\lm{N}$ be the unique solution of
\begin{alignat}{2}
	\partial_t \lm{u} -D\Delta \lm{u} &= \left(\frac{c_i(\lm{S},\lm{N},F)}{c_{\min}}-1\right)\lm{u} &&\qquad \text{a.e. in } (0,T)\times \Omega,\label{pop_limit_pop1}\\ 
	\partial_\nu \lm{u} &= 0 ,&&\qquad \text{a.e. in } (0,T)\times\partial\Omega,\label{pop_limit_pop2}\\ 
	\lm{u}(0) &= u_{in}&&\qquad \text{a.e. in } \Omega\label{pop_limit_pop3}\\
	N_{in} &= \int_\Omega \lm{u}(0,x) dx,\\
	\lm{S}(0)&= \min \{\max \{L(\lm{f}(0)) , S_{in}\} ,{U}	(\lm{f}(0))\},\label{pop_limit_stock1}\\
	\dot{\lm{S}}(t)(\lm{S}(t)-z) &\leq 0,\qquad \text{ for all } z\in [{L}(\lm{f}(t)),{U}(\lm{f}(t))]&& \qquad \text{a.e. in }[0,T], \label{pop_limit_stock2}\\
	\lm{S}(t)&\in [{L}(\lm{f}(t)),{U}(\lm{f}(t))] &&\qquad \text{in }[0,T].\label{pop_limit_stock3}
\end{alignat}
	
Then,  for $2\leq q <\infty$ arbitrary, $\lm{u}$ has the same regularity as $u_\varepsilon$ in Theorem \ref{Thm:pop_dnymics_existence} and
$\lm{S}$ is in $\mathrm{W}^{1,\infty}(0,T)$.
Moreover, in the limit $\varepsilon\to 0$ 
\begin{align*}
	u_{\varepsilon} &\rightarrow \lm{u}\qquad \text{in } \mathrm{W}^{1,q}(0,T;\mathrm{L}^2(\Omega))\cap\mathrm{L}^{q}(0,T;\mathrm{H}^2(\Omega))\text{ and}\\
	S_{\varepsilon} &\rightarrow \lm{S} \qquad \text{in } \mathrm{L}^{q}(0,T) 
\end{align*}  
and $S_{\varepsilon}(t)$ and $\lm{S}(t)$ are uniformly bounded in $t$ and $\varepsilon$ by 
\begin{equation*}
	S_\varepsilon(t), \lm{S}(t) \leq S_{\max},\qquad\forall t\in[0,T],\ \forall\varepsilon>0.
\end{equation*}
	
Note that the limit $\lm{S}$ is a generalised play operator with input $\lm{f}$ for the curves ${U}(f)$ and
${L}(f)$.
In fact, Figure~\ref{pop_pic_limit_phase_diagram_1} depicts the behaviour in case of consume rate function $c_2$. 
\end{theorem}

%
%

\begin{figure}
	\begin{tikzpicture}[] 
	\draw[->] (0,0) -- (10,0) coordinate (x axis);
	\draw (-1pt,5 ) node[anchor=east,fill=white] {${S}$};
	\draw (10 cm,-1pt) node[anchor=north] {${f}$};
	\draw[->] (0,0) -- (0,5) coordinate (y axis);
	\draw (1 cm,1pt) -- (1 cm,-1pt) node[anchor=north] {$c_{\min}$};
	\draw (1pt,3 ) -- (-1pt,3 ) node[anchor=east,fill=white] {$S_{\max}$};
	\draw [blue, thick, domain=0:1] plot (\x, {0});
	\draw [blue, thick, domain=1:2.5] plot (\x, {(\x-1)/(3-\x)});
	\draw [blue, thick, domain=2.5:8] plot (\x, {(2.5-1)/(3-2.5)});
	\draw [blue] (2 cm,1cm) node[anchor=east] {${U}$};
	\draw (1.7 cm,1.5cm) node[anchor=east] {$\dot{S}<0$};
	\draw (4.9 cm,1.5cm) node[anchor=east] {$\dot{S}=0$};
	\draw (8.8 cm,1.5cm) node[anchor=east] {$\dot{S}>0$};
	\draw [<-] [black, domain=2.1:2.3] plot (\x, {(\x-1)/(3-\x)});
	\draw [red, thick, domain=5:((2.5-1)/(3-2.5)+5)] plot (\x, {(\x-5});
	\draw [red] (6cm +3pt,1cm+1pt) node[anchor=west] {${L}$};
	\draw [red, thick, domain=((2.5-1)/(3-2.5)+5):9.5] plot (\x, {(2.5-1)/(3-2.5)});
	\draw [red, thick, domain= 1:5] plot (\x, {0});
	\draw [->] [black, domain=6:6.5] plot (\x, {(\x-5});
	\draw [<->] [black, domain=(8.6):9.0] plot (\x, {(2.5-1)/(3-2.5)});
	\draw [<->] [black] (4.5,3) -- (6.5,3);
	\draw [<->] [black] (0.3,0) -- (0.7,0);
	\draw [<->] [black] (3.5,2) -- (5.5,2);
	\draw [<->] [black] (2.5,1) -- (4.5,1);
	\draw [<->] [black] (2,0) -- (4,0);
	\end{tikzpicture}
	\caption{${f}$-${S}$-phase diagram: upper and lower thresholds $U$ and $L$}
	\label{pop_pic_limit_phase_diagram_1}
\end{figure}
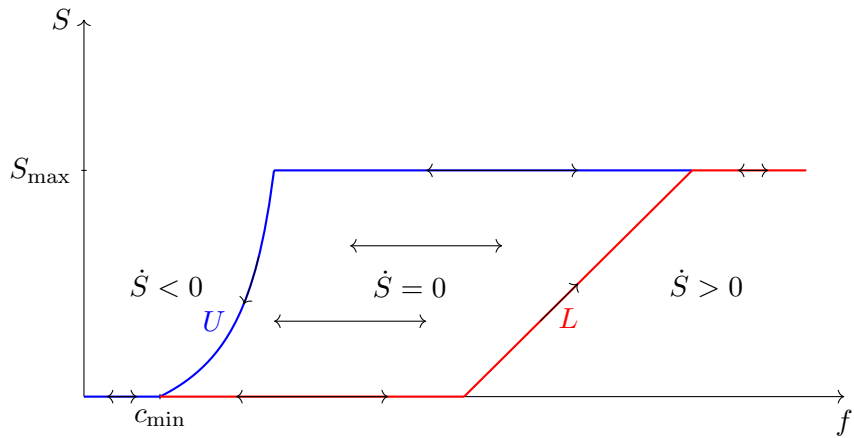

\begin{remark}[Remarks on the existence and uniqueness of limiting system]\label{Rem:Exist/UniqueLimit}\hfill\\
The existence and uniqueness of solutions for \eqref{pop_limit_pop1}--\eqref{pop_limit_stock3} are part of the results of Theorem~\ref{Thm:singular_limit}. Existence follows by showing that a limit $(\lm{u},\lm{S}, \lm{N})$ of $(u_\varepsilon,S_\varepsilon,N_\varepsilon)$ solves \eqref{pop_limit_pop1}--\eqref{pop_limit_stock3}.
The limit $\lm{S}$ is a generalised play operator for the Lipschitz continuous curves ${U}$ and $L$ with input $\lm{f}=F/\lm{N}$.
By \cite[III.2. Theorem 2.2]{visintin2013differential}, this generalised play is a Lipschitz continuous hysteresis operator from $\mathrm{C}([0,T])\times \mathbb{R}$ to $\mathrm{C}([0,T])$. 
An alternative, direct approach to show existence uses a fixed point argument similar to \cite[Thm.~3.1]{muench1}.
	Crucial for the proof of uniqueness is an estimate of the form
	\begin{equation*}
	\lm{N}(t)>\delta(T)>0,
	\end{equation*}
	uniformly in $[0,T]$.
	The latter follows from the at most exponential decay of $\lm{N}$ and boundedness of the interval $[0,T]$.
	Therefore, $\lm{f}=F/\lm{N}$ is a Lipschitz continuous function and
	uniqueness of $(\lm{u},\lm{S},\lm{N})$ follows with a Gronwall argument.
\end{remark}

\begin{remark}[Generalisation to reaction-diffusion-system-ODE/hysteresis models]\label{Rem:Generalization}\hfill\\
The semi-group based existence theory of Theorem \ref{Thm:pop_dnymics_existence} 
can be extended to suitable nonlinear reaction-diffusion systems (see e.g.  \cite[Chp.~6]{pazy}),
where the reaction terms describe, for instance, nonlinear interactions between different species. 
One example could be competition for food. 	
For multi-species models, the total population size $N(t)$ at time $t$ for $m$ species is given by
\[
N(t) = \sum_{i=1}^m \int_\Omega u_i(t,x) dx.
\]
Using this definition as total population size, we expect that analogue results to Theorem \ref{Thm:singular_limit}
can be proven via similar arguments. Finally, growth rate functions of the form 
$\lambda(N,S,F) \sim \lambda(c-c_{min})$ can be treated analog as long as  
the dependence of $\lambda$ on $N$, $S$ and $F$ only occurs indirectly via $c(N,S,F)$.
\end{remark}

\begin{remark}[Generalisation to spatially heterogeneous models]\hfill\\
A generalisation of Theorem~\ref{Thm:pop_dnymics_existence} and Theorem~\ref{Thm:singular_limit} to $x$-dependent consume (and growth) functions describing e.g. a heterogeneous domain $\Omega$ can be interesting for many applications. 
In general, such generalisations to $x$-dependent functions are straight forward. 
Note that the modelling has to be adapted if the space-dependency involves an $x$-dependent consume rate function $c$.
More precisely, all terms $F-Nc$ which appear in the evolution of the scalar stock level $S$ have to be replaced by some functional in $c$, for example by $F-\int_\Omega uc\, dx$.

Then, similar proofs as stated in this paper apply also to spatially heterogeneous models as long as all 
functions are sufficiently regularity in $x$ and satisfy the correct boundary conditions. 
For example, if $x\mapsto\lambda(N,F,S,x)\in C_0^\infty(\Omega)$ holds for all fixed $N,F,S$, then $\lambda(N,F,S,\cdot)u(\cdot) \in \lbrace v\in \mathrm{H}^2(\Omega): \partial_\nu v = 0 \text{ on } \partial\Omega \rbrace$ for all $u\in \lbrace v\in \mathrm{H}^2(\Omega): \partial_\nu v = 0 \text{ on } \partial\Omega \rbrace$, which is essential for the proof of Theorem~\ref{Thm:pop_dnymics_existence}.
In this case, all other required regularity conditions during the proofs can be proven without much effort.
Finally, we remark that also the evolution equation for $N$ needs changing if $\lambda$ is $x$-dependent:
\begin{equation*}
	\dot{N}(t) = \int_{\Omega}\lambda(N(t),F(t),S(t),x)u(x,t)\,dx.
\end{equation*}
	The proof for positivity of $u_\varepsilon$ does not require a uniform positive lower bound of $N_\varepsilon$.
Hence, 
	\begin{align*}
	\dot{N}(t) 
	\geq -\|\lambda(N(t),F(t),S(t),\cdot)\|_{\mathrm{L}^\infty(\Omega)}\int_{\Omega}u(x,t)\, dx
	=-\|\lambda(N(t),F(t),S(t),\cdot)\|_{\mathrm{L}^\infty(\Omega)}N(t),
	\end{align*}
which implies uniform positive lower bounds for $N_\varepsilon$ and $\lm{N}$ on $[0,T]$.
The latter is crucial in the proof of Theorem~\ref{Thm:singular_limit}, cf. Remark \ref{Rem:Exist/UniqueLimit}.
\end{remark}

\noindent\underline{Outline:}
The proof of existence and regularity in Theorem~\ref{Thm:pop_dnymics_existence} 
for general consume and growth rate functions $c$ and $\lambda$ and for fixed $\varepsilon>0$
adapts well-known methods of semigroup theory (see e.g. \cite{henry,pazy}) and is therefore postponed 
to the final Section~\ref{existence}.
More specific reference on hysteresis-reaction-diffusion models and fast-reaction hysteresis limits are e.g. 
\cite{GR15b,GR15,PKKCPR12,TAD14}
and the references therein. 


In the following Section~\ref{consume}, we present two heuristic models for consume rate functions $c_1(S,N,F)$ and $c_2(S,N,F)$, which satisfy the general regularity assumptions of  Theorem~\ref{Thm:pop_dnymics_existence}.
The first model of a consume rate function implements an upper bound for a maximal consume rate. 
As a consequence the stock can build up arbitrarily large. However, under additional assumption on the 
food supply $F(t)$, we will show that the stock remains bounded along the corresponding solutions. 
As alternative model, we also discuss a consume rate function which enforces
an upper maximal bound for the stock level $S(t)$ by forcing the consume rate to become arbitrarily large as the stock 
level gets close to its maximum level. 

In Section~\ref{limit}, we show our main Theorem~\ref{Thm:singular_limit}, i.e. we
prove rigorously the limit  $\varepsilon\to 0$. 
The proof generalises a related previous result of ODE-hysteresis systems \cite{kuehn2017generalized}.

In Section~\ref{examples}, we present and discuss  numerical examples.
A first example illustrates Theorem~\ref{Thm:singular_limit} and plots a typical temporal evolution of $F(t)$, $N(t)$ and $S(t)$ in the hysteresis-reaction-diffusion model \eqref{pop_limit_pop1}--\eqref{pop_limit_stock3} and also provides a simulation video (supplementary material).
A second example is constructed in terms of a simplest nonlinear hysteresis-reaction-diffusion model 
and demonstrates the remarkable observation that the shape of a scalar generalised play operator 
can decide between spatial homogenisation or unbounded growth of spatially inhomogeneous Fourier modes, 
a phenomenon which could be called hysteresis-diffusion driven instability.

\section{Two models of consume rate functions}\label{consume}


In this section, we derive two general heuristic models for consumption rate functions, which each leads to a generalised play operator in the singular limit $\varepsilon\to 0$.

The first model assumes a maximal consume rate $c_{\max}>0$, which entails, however,
a possible pile up of an infinitely large stock in case of unbounded food supplies.
The second model of a consume rate function ensures that the stock is bounded by some maximum value $S_{\max}>0$ 
by assuming that the population consumes all excess food, which would increase the stock beyond $S_{\max}$.
One justifying interpretation of such an assumption could be 
the tendency to waste resources when the stock reaches its maximum level $S_{\max}$.

\subsection{A bounded consume rate function without stock limitation}\label{cbounded}\hfill\\
In the following, we derive an autonomous consumption rate $c_1(S,N,F)$
as a function, which only depends on time via its arguments $S(t),N(t),F(t)$. 
Hence, we will suppress the time-dependency in the notation of $c_1(S,N,F)$. 

The function $c_1$ shall be a prototypical model for a uniformly bounded consumption rate function, i.e. we postulate a maximal possible individual consumption rate $c_{\max}>c_{\min}$ (recall that $c_{\min}>0$ denotes the minimal consumption rate required for the growth rate function $\lambda$ to be nonnegative, i.e. the population declines for $c_1<c_{\min}$) such that 
\[
0\leq c_1(S,N,F) \leq c_{\max}.
\]
As a consequence, at times when the individual food supply $f(t) = F(t)/N(t)> c_{\max}$ it is certain that not all food can be consumed
and that the unconsumed food will increase the stock level $S(t)$, which is thus potentially unbounded. 

\medskip

In the following, we model $c_1(S,N,F)$ as the concatenation of three regimes in the $f$-$S$-phase plane, 
see Fig. \ref{Fig:f-Sdiagram_UnboundedStock}.
Note that while these three regions are only characterised by the two variable $(S,f)$, the actual values of $c_1$ will also depend on $N$. 
\begin{description}[topsep=5pt, leftmargin=6mm]
\item[Depleting regime]
The depleting regime refers to the left-sided area of the $f$-$S$-phase plane, where the individual food supply $f$ is  close to $c_{\min}$ (or below) and the population 
resorts to supplies from the stock, 
yet depending on the available stock level $S$:
We fix first an intermediate supply level $\de{c}^\infty\in(c_{\min},c_{\max})$
and consider situations where the individual food supply 
$f$ satisfies 
\[
c_{\min}\le f \le \de{c}^\infty< c_{\max},
\]
which means that the total population $N$ is quite large in comparison to the available food supply $F$.

Hence, we postulate an upper threshold of the depleting regime (i.e. the value of $f$ under which the population uses the stock) as a monotone increasing function $\de{c}(S) \in [c_{\min},\de{c}^\infty]$, see the left side of Figure~\ref{Fig:f-Sdiagram_UnboundedStock}.
If the stock is empty, it is natural to set $\de{c}(0)=c_{\min}$, which is the critical rate of consumption below which the population will decline. 
A prototypical choice for $\de{c}(S)$ is the function
\begin{equation}\label{cde}
	\de{c}(S)=\frac{\de{c}^\infty S + c_{\min}}{S+ 1},
\end{equation}
which saturates at $\de{c}^\infty $ in the limit $S\to\infty$,
yet any strictly monotone $C^1$ function connecting 
$(c_{\min},0)$ and $(\de{c}^\infty,\infty)$ will yield equivalent results. 
Note that $\de{c}^\infty $ is the 
asymptotically largest consumption rate below which the population enters the depleting regime and can be interpreted as a measure for how careful
the population deals with the stock. 

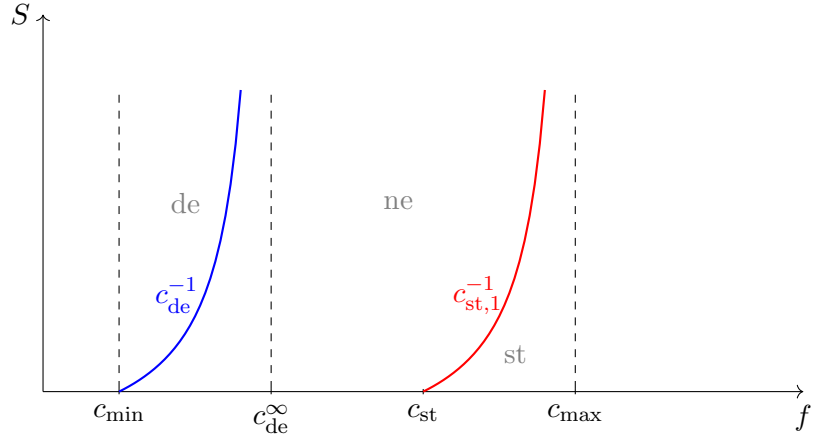
\begin{figure}[htb]
	\begin{tikzpicture}[] 
	\draw[->] (0,0) -- (10,0) coordinate (x axis);
	\draw (-1pt,5 ) node[anchor=east,fill=white] {$S$};
	\draw (10 cm,-1pt) node[anchor=north] {$f$};
	\draw[->] (0,0) -- (0,5) coordinate (y axis);
	\draw (1 cm,1pt) -- (1 cm,-1pt) node[anchor=north] {$c_{\min}$};
	\draw (5 cm,1pt) -- (5 cm,-1pt) node[anchor=north] {$\st{c}$};
	\draw (3 cm,1pt) -- (3 cm,-1pt) node[anchor=north] {$\de{c}^\infty $};
	\draw[dashed] (1 cm,1pt) -- (1 cm,4) ;
	\draw[dashed] (3 cm,1pt) -- (3 cm,4) ;
	\draw (7 cm,1pt) -- (7 cm,-1pt) node[anchor=north] {$c_{\max}$};
	\draw [blue, thick, domain=1:2.6] plot (\x, {(\x-1)/(3-\x)});
	\draw [gray] (2.2 cm,2.5cm) node[anchor=east] {$\mathrm{de}$};
	\draw [gray] (5 cm,2.5cm) node[anchor=east] {$\mathrm{ne}$};
	\draw [blue] (2cm+5pt,1cm+8pt) node[anchor=east] {$\de{c}^{-1}$};
	\draw [red, thick, domain=5:6.6] plot (\x, {(\x-5)/(7-\x)});
	\draw [red] (6cm +5pt,1cm+8pt) node[anchor=east] {$\sto{c}^{-1}$};
	\draw [gray] (6.5cm,0.5cm) node[anchor=east] {$\mathrm{st}$};
	\draw[dashed] (7 cm,1pt) -- (7,4) ;
	\end{tikzpicture}
\caption{$f$-$S$-phase diagram for a consume rate function without stock limitation. The depleting regime ($\mathrm{de}$) is the area left to $S=\de{c}^{-1}$. The feasting regime ($\mathrm{st}$) is located to the right of $S=\sto{c}^{-1}(f)$.
The neutral consumption regime ($\mathrm{ne}$) corresponds to the area between $\de{c}^{-1}$ and $\sto{c}^{-1}$.}
	\label{Fig:f-Sdiagram_UnboundedStock}
\end{figure}

In the depleting regime, 
the stock consumption yields (recall \eqref{pop_dnymics_stock_evol1}): 
\begin{align}
	\varepsilon\dot{S}=- l(S,N,F)\, S = \left[f - c_1(S,N,F)\right]N \le 0\label{pop_stock_take}
\end{align}
for all $S\geq 0$. 
We aim to define a bounded consume rate function $c_1$ in such a way that the rate function $l$ is also bounded by a constant (which we will assume normalised w.l.o.g. due to a possible time rescaling), i.e. 
\begin{equation}\label{lbound}
0\leq l(S,N,F)\leq 1.
\end{equation}
Note that because of \eqref{pop_stock_take} and \eqref{lbound}, we have to choose $c_1$ such that for $S\geq 0$, $F\geq 0$ and $N>0$ satisfying $f< \de{c}(S)$ holds
\[
	- S \leq - l S = \left[f - c_1(S,N,F)\right] N \leq 0.
\]
These two inequalities lead to two constraints in the choice of $c_1$: 
\begin{align*}
	f  \leq c_1(S,N,F)\leq  f +  \frac{S}{N}, 
\end{align*}
which requires in particular in the limit $S\rightarrow 0$
\begin{equation*}
 c_1(S,N,F) - f = O\Bigl(\frac{S}{N}\Bigr) \qquad \text{and}\qquad
 c_1(0,N,F)=f.
\end{equation*}
Note that the latter condition ensures $\dot S=0$ whenever $S=0$ in \eqref{pop_stock_take}. 
It further entails 
\[
	c_1(0,N,F)=c_{\min}=\de{c}(0) \qquad \text{if}\qquad f= c_{\min}.
\]
Finally, we wish for $c_1(S,N,F)$ to be continuous in all variables and increasing in $S$.

	
	
All these requirements are satisfied by the following  prototypical model for a consumption function $c_1(S,N,F)$ in the depleting regime $f\leq \de{c}(S)$: \begin{equation*}
	c_1(S,N,F)=f + \frac{S}{N}\left(1-e^{-N \left(1-{f}/{\de{c}(S)}\right)}\right)
\qquad \text{for} \qquad f\leq \de{c}(S).
\end{equation*}
This implies 
\[
	l(S,N,F)  = 1-e^{-N \left(1-{f}/{\de{c}(S)}\right)}\le1\qquad \text{for} \qquad f\leq \de{c}(S).
\]
	
\item[Storing regime]
As second regime, we consider the reverse situation when the food supply $F$ is large compared to the total population $N$. To define this regime, we introduce an upper individual consumption level $\st{c}$, such that 
\[
	\de{c}^\infty <\st{c}<c_{\max}
\]
and consider the situation when $f \geq \st{c}$. Recall that $f\geq c_{\max}$
entails that not all food can be consumed. 
However, depending on the current stock level $S$, we postulate that the population decides to store food whenever the individual food supply $f$ surpasses a lower threshold value $\sto{c}(S)< c_{\max}$, cf. Figure~\ref{Fig:f-Sdiagram_UnboundedStock}.
We call this range the \textbf{storing regime}.
The storing threshold $\sto{c}(S)$ is modelled (similar to $\de{c}(S)$)  as a monotone increasing function of $S$ such that
\[
\st{c} =  \sto{c}(0) \leq  \sto{c}(S) < c_{\max}. 
\]
One interpretation of $\st{c}$, which equals the storing threshold at empty stock $\sto{c}(0)$, 
can be as a measure of how optimistic the population feels for future times.
Moreover, since not all supplied food can be consumed for $f\ge c_{\max}$, it is natural to set
\[
	\lim_{S\to\infty} \sto{c}(S)= c_{\max}.
\]
Hence, a suitable heuristic choice for $\sto{c}(S)$ is
\begin{align}\label{st1}
	\sto{c}(S)=\frac{\st{c} + c_{\max}S}{1+S},
\end{align}
yet again any such strictly monotone increasing $C^1$ function will yield equivalent results. 
Next, a prototypical (and simplest) model for $c_1$ in the storing regime is
\begin{align*}
	c_1(S,N,F)=\sto{c}(S)=\frac{\st{c} + c_{\max}S}{1+S}\qquad \text{for}\qquad f \geq \sto{c}(S),
\end{align*}
which means that $c_1$ saturates at the level of $\sto{c}(S)$ for all food supply rates $f \geq \sto{c}(S)$.
	%
This readily determines
\[
	G(S,N,F)= \left(f -\sto{c}(S)\right)N\qquad \text{for} \qquad f \geq \sto{c}(S).
\]

\item[Neutral consumption regime]
For a medium individual food supply, i.e.
\[
	\de{c}(S) <f< \sto{c}(S),
\]
we assume that the population decides neither to
add nor to use food from the stock, i.e. $G(S,N,F)=l(S,N,F)=0$, cf. Figure~\ref{Fig:f-Sdiagram_UnboundedStock}.
Because of 
\begin{align*}
	\dot{S}&=0 = G(S,N,F)- l(S,N,F)\, S= \left(f - c_1(S,N,F)\right) N
\end{align*}
the only choice for $c_1$ is $c_1=f$ in this regime.
\end{description}
\medskip

Combining those three regimes, our model of a bounded consume rate function $c_1$ reads as
\begin{equation}\label{c_1}
c_1(S,N,F)=
\begin{cases}
f + \frac{S}{N}\left(1-e^{-N \left(1-{f}/{c_{\scriptscriptstyle \mathrm{de}}(S)}\right)}\right) &\qquad \text{if } f \leq \de{c}(S),\\
f &\qquad \text{if } \de{c}(S) <f< \sto{c}(S),\\
\sto{c}(S) &\qquad \text{if } f \geq \sto{c}(S).
\end{cases}
\end{equation}
The functions $G(S,N,F)$ and $l(S,N,F)$ are accordingly given by
\begin{equation}\label{G_1}
G(S,N,F)=
\begin{cases}
0 &\qquad \text{if } f< \sto{c}(S),\\
\left(f -\sto{c}(S)\right)N &\qquad \text{if } f \geq \sto{c}(S),
\end{cases}
\end{equation}
and

\begin{equation}\label{l_1}
l(S,N,F)=
\begin{cases}
1-e^{-N \left(1-{f}/{c_{\scriptscriptstyle \mathrm{de}}(S)}\right)} &\qquad \text{if } f \leq \de{c}(S),\\
0 &\qquad \text{if } f> \de{c}(S).
\end{cases}
\end{equation}


\subsection{An unbounded consume rate function ensuring limited stock}\label{cunbounded}\hfill\\
As alternative model to the previous bounded consume rate function without stock limitation, we introduce here 
a second consume rate $c_2(S,N,F)$, 
which ensures a bounded stock with maximal value $S_{\max}>0$ as a consequence of a potentially unbounded consume rate $c_2$. In this context, unbounded consumption can be interpreted as a (seemingly) unlimited tendency to 
waste resources with saturating stock level, despite the fact that any realistic stock is limited.

As in the previous subsection, we distinguish three cases:

\begin{description}[topsep=5pt, leftmargin=6mm]
\item[Depleting regime]
If the population is large in comparison to the food supply rate, we chose the same 
depleting regime as in the previous Subsection \ref{cbounded}, i.e. for a $\de{c}^\infty  \in(c_{\min} , c_{\max})$, 
we consider $\de{c}(S)$ as defined in \eqref{cde} and call the depleting regime 
all states such that $f<\de{c}(S)$.
In particular, this implies
\begin{align*}
c_2(S,N,F) = f+ \frac{S}{N}\left(1-e^{-N \left(1-{f}/{c_{\scriptscriptstyle \mathrm{de}}(S)}\right)}\right)
\qquad \text{for} \quad f<\de{c}(S).
\end{align*}
	
\item[Storing regime] Considering a small population in comparison to the food supply rate
in the sense that $f> \st{c}$ for a $\st{c}\in(\de{c}^\infty ,c_{\max})$, 
we postulate (similar to Subsection \ref{cbounded}) a lower threshold $f=\stt{c}(S)$, above which the population decides to store food in the stock. 
Again it is natural to model $\stt{c}(S)$ monotone increasing in $S$.  
If the current stock level $S$ reaches the maximum value $S_{\max}$, 
the individuals consume (or waste) everything.
A simplest choice for $\stt{c}(S)$ which satisfies all our conditions is
\begin{equation}\label{st2}
\stt{c}(S)= S + c_{\mathrm{st}}\quad\Leftrightarrow\quad \stt{c}^{-1}(f) = f - c_{\mathrm{st}}.
\end{equation}
A model of a prototypical continuous consume rate function for $f\geq \stt{c}(S)$ is given by	
\begin{equation*}
c_2(S,N,F)=
\begin{cases}
	fe^{S-S_{\max}}+ \stt{c}(S)\left(1- e^{S-S_{\max}}\right)  & \text{if } f > \stt{c}(S) \text{ and } S< S_{\max},\\
	f & \text{if } f > \stt{c}(S) \text{ and } S\geq S_{\max}.
\end{cases}
\end{equation*}
	
\item[Neutral consumption regime] 
As in Subsection \ref{cbounded}, for medium individual food supply $\de{c}(S) <f< \stt{c}(S)$, 
the only possible choice is 
\[
	c_2(S,N,F)=f.
\]	
\end{description}
\medskip

Combining the three regimes, a prototypical model of a consume rate function with limited stock is 
\begin{equation}\label{c_2}
c_2(S,N,F)=
\begin{cases}
f + \frac{S}{N}\left(1-e^{-N \left(1-{f}/{c_{\scriptscriptstyle \mathrm{de}}(S)}\right)}\right) & \text{if } f \leq \de{c}(S),\\
f & \text{if } \de{c}(S) < f\leq \stt{c}(S),\\
fe^{S-S_{\max}}+ \stt{c}(S)\left(1- e^{S-S_{\max}}\right) & \text{if } f > \stt{c}(S) \text{ and } S< S_{\max},\\
f & \text{if } f > \stt{c}(S) \text{ and } S\geq S_{\max}.
\end{cases}
\end{equation}

By the definition of $G(S,N,F)= \left(f-c_2(S,N,F)N\right)_{+}$, we obtain
\begin{equation*}
G(S,N,F)=
\begin{cases}
N (f-\stt{c}(S))\left(1- e^{S-S_{\max}}\right) &\qquad \text{if } f > \stt{c}(S) \text{ and } S< S_{\max},\\
0 &\qquad \text{else, }
\end{cases}
\end{equation*}
and for $l(S,N,F)S= \left(f-c_2(S,N,F)N\right)_{-}$, we have
\begin{equation*}
l(S,N,F)=
\begin{cases}
1-e^{-N \left(1-{f}/{c_{\scriptscriptstyle \mathrm{de}}(S)}\right)}&\qquad \text{if } f \leq \de{c}(S),\\
0 &\qquad \text{if } f> \de{c}(S).
\end{cases}
\end{equation*}

\subsection{Properties of the consume rate functions $c_1$ and $c_2$}\hfill\\
The following Proposition~\ref{c1Smax} and Corollary~\ref{c12} ensure that both consume rate models satisfy the requirements of  Theorems~\ref{Thm:pop_dnymics_existence} and 
\ref{Thm:singular_limit} under suitable assumptions. 

\begin{proposition}[A-priori bounds for $S$ for consume rate function $c_1$]\label{c1Smax}\hfill\\
For fixed $T>0$, assume either 
\begin{equation}\label{bounda}
\|F\|_{C([0,T])}\leq \kappa c_{\max}N_{in}\exp(-T)\qquad 
\text{for some} \quad  \kappa\in (0,1)
\end{equation}
or 
\begin{equation}\label{boundb}
\begin{gathered}
0< F_{\min}\le F(t) \le F_{\max},\qquad \forall t\in [0,T],\qquad
\text{with}\quad \gamma:=\frac{F_{\max}}{F_{\min}}\frac{c_{\min}}{c_{\max}}<1,\\
\text{and initial data, which satisfy}\\
N_{in}\ge\frac{F_{\min}}{c_{\min}}.
\end{gathered}
\end{equation}

Then, there holds
$$
0\leq S_\varepsilon \leq {S_{\max}}, \qquad \text{on} \quad  [0,T], 
$$
where 
$$
S_{\max}:=\max\left\{S_{in}, \frac{\kappa c_{\max} - \st{c}}{c_{\max}(1-\kappa)}\right\}
 \qquad \text{or} \qquad S_{\max}:=\frac{\gamma-\frac{\st{c}}{c_{\max}}}{1-\gamma}.
$$
\end{proposition}
\begin{remark}
We remark that the assumptions \eqref{boundb} allow to identify an invariant region for $(N_{\varepsilon},S_{\varepsilon})$ of the form $N_{\varepsilon}\ge \frac{F_{\min}}{c_{\min}}$ and $S_{\varepsilon}\in[0,\frac{\kappa-{\st{c}}/{c_{\max}}}{1-\kappa}]$.
Note that without lower bounds $F_{\min}$ 
and $N_{in}$ the population $N_{\varepsilon}$ can be arbitrarily small
and that without an upper bound $F_{\max}$ the stock 
can grow arbitrarily large arbitrarily fast as $\varepsilon\to 0$.
\end{remark}

\begin{proof}
Remember the definitions $N_\varepsilon:=N(u_\varepsilon)$ and $f_\varepsilon={F}/{N_\varepsilon}$.

Integrating \eqref{pop_dnymics_pop_evol4} with 
$\lambda_1= \Bigl(\frac{c_1(S,N,F)}{c_{\min}}-1\Bigr)\ge -1$ yields independently of $\varepsilon$
\begin{equation}\label{Nbounda}
	N_\varepsilon(t) \geq N_{in}\exp(-t)\geq N_{in}\exp(-T)=: \delta(T),\qquad\forall t\in[0,T],\ \forall\varepsilon>0.
\end{equation}
By assumption \eqref{bounda}, $F$ is uniformly bounded by $\kappa c_{\max} \delta(T)$.
Consequently, $f_\varepsilon=F/N_\varepsilon$ is uniformly bounded by 
\begin{equation*}
	f_\varepsilon \leq \kappa c_{\max} ,\qquad\forall t\in[0,T],\ \forall\varepsilon>0.
\end{equation*}
In case $\kappa c_{\max}\leq \st{c}$, then $\dot{S}_\varepsilon\leq 0$ in $[0,T]$ by definition of 
$G=\max\{0,N(f-c_1(S,N,fN))\}$ in \eqref{G_1} and since $c_1\geq f$ holds for $f\leq \st{c}$.
Hence, then $S_\varepsilon \leq S_{in}$ for all $t\ge 0$ and $\varepsilon>0$.

In case that $\kappa c_{\max}\in (\st{c},c_{\max})$, we have $\dot{S}_\varepsilon \leq 0$ if $f_\varepsilon \leq \sto{c}(S_\varepsilon)$ while $\dot{S}_{\varepsilon}\ge0$ provided $f_\varepsilon \geq \sto{c}(S_\varepsilon)= \frac{\st{c} + c_{\max}S_\varepsilon}{1+S_\varepsilon}$. In the second case, we estimate 
\begin{equation*}
0 \le \dot{S}_\varepsilon = \frac{N_\varepsilon}{\varepsilon}\left(f_\varepsilon - \frac{\st{c} + c_{\max}S_\varepsilon}{1+S_\varepsilon}\right) \leq \frac{N_\varepsilon}{\varepsilon}\left(\kappa c_{\max} - \frac{\st{c} + c_{\max}S_\varepsilon}{1+S_\varepsilon}\right),
\end{equation*}
which implies that $S_\varepsilon$ is uniformly bounded by
\begin{equation*}
	S_\varepsilon(t) \leq \max\left\{S_{in}, \frac{\kappa c_{\max} - \st{c}}{c_{\max}(1-\kappa)}\right\},\qquad\forall t\in[0,T],\ \forall\varepsilon>0.
\end{equation*}
\medskip
		
Alternatively, assume assumption \eqref{boundb}.
Since $N_\varepsilon$ in \eqref{pop_dnymics_pop_evol4}
decays only if $c_1\le c_{\min}$, which by definition of $c_1$ only happens if also $c_1\ge f_{\varepsilon}$ holds, we estimate in such situations that
\begin{equation*}
\dot{N}_\varepsilon = \lambda N_\varepsilon \ge \left(\frac{f_{\varepsilon}}{c_{\min}}-1\right) N_\varepsilon \ge 
\frac{F_{\min}}{c_{\min}}-N_\varepsilon,
\end{equation*}
while otherwise $N_\varepsilon$ is non-decreasing. 
Hence, we obtain independently of $\varepsilon$ that 
\begin{equation}\label{Nboundb}
N_\varepsilon \ge \min\left\{N_{in},\frac{F_{\min}}{c_{\min}}\right\}=\frac{F_{\min}}{c_{\min}} 
,\qquad\forall t\in[0,T],\ \forall\varepsilon>0,
\end{equation}
by the assumption on $N_{in}$ in \eqref{boundb}.
Together with the definition of $\gamma<1$ in \eqref{boundb}, this implies the estimate
$f_\varepsilon \le \frac{F_{\max}}{N_{\min}}\le \gamma c_{\max}$, which is a sufficient condition to avoid 
that $S_{\varepsilon}$ grows unboundedly 
in the limit $\varepsilon\to 0$, see Fig.~\ref{Fig:f-Sdiagram_UnboundedStock}.
Moreover, we have $\dot{S}_\varepsilon\ge0$ 
only if $f_\varepsilon \geq \st{c}(S_\varepsilon)= \frac{\st{c} + c_{\max}S_\varepsilon}{1+S_\varepsilon}$, where we estimate
\begin{equation*}
\dot{S}_\varepsilon = \frac{N_\varepsilon}{\varepsilon}\left(f_\varepsilon - \sto{c}(S_\varepsilon)\right) \leq \frac{1}{\varepsilon}\left(F_{\max} - \frac{F_{\min}}{c_{\min}}\frac{\st{c} + c_{\max}S_\varepsilon}{1+S_\varepsilon}\right)=\frac{F_{\min}c_{\max}}{\varepsilon \,c_{\min}}
\left(\gamma - \frac{\frac{\st{c}}{c_{\max}}+S_\varepsilon}{1+S_\varepsilon}\right),
\end{equation*}
which implies independently of $\varepsilon$ that
\begin{equation*}
S_\varepsilon(t) \leq \max\left\{S_{in}, \frac{\gamma-\frac{\st{c}}{c_{\max}}}{1-\gamma}\right\}=\max\left\{S_{in}, \frac{\gamma c_{\max} - \st{c}}{c_{\max}(1-\gamma)}\right\},
\qquad\forall t\in[0,T],\ \forall\varepsilon>0.
\end{equation*}
Note that in case $\gamma < \frac{\st{c}}{c_{\max}}$,
we always have $\dot{S}_{\varepsilon}\le 
{\varepsilon}^{-1}\left(F_{\max} - \frac{F_{\min}}{c_{\min}}\st{c}\right)\le 0$.
\end{proof}

%

\begin{corollary}[Admissibility of the consume rate functions $c_1$ and $c_2$]\label{c12}\hfill\\
	Consider the consume rate function $c_1(S,N,F)$ 
	as defined in \eqref{c_1} and with $S_{\max}$ 
	as given by Proposition~\ref{c1Smax}
	or  $c_2(S,N,F)$ as given in \eqref{c_2} with $S_{\max}$ as in the definition. 
	
	Then, $c_1$ and $c_2$ 
	satisfy the assumptions of 
	Theorems~\ref{Thm:pop_dnymics_existence} and 
	\ref{Thm:singular_limit}. 
	Moreover, for $0<c_{\min}<c^{\infty}_{\mathrm{de}}<\st{c}<c_{\max}$, 
	we have as upper and lower thresholds functions 
	\begin{align*}
	{U}(f)&=\min\{S_{\max},u(f)\}\qquad \text{with}\quad 
	u(f):=\max\{0,c_{\mathrm{de}}^{-1}(f)\}\quad \text{and}\quad c_{\mathrm{de}}^{-1}(f)=\frac{f-c_{\min}}{\de{c}^{\infty} -f} \\
	&=
	\begin{cases}
	0 & \qquad f\in[0,c_{\min})\\[2mm]
	\in(0,S_{\max})\quad \text{ strictly monotone increasing} 
	&\qquad f\in\left(c_{\min},\tilde{c}\right)\\[2mm]
	S_{\max} & \qquad f\ge \tilde{c},
	\end{cases}
	\end{align*}
	where
	\[
	\tilde{c}:=\frac{S_{\max}\de{c}^{\infty} + c_{\min}}{S_{\max} +1} \quad\Leftrightarrow \quad u(\tilde{c})=S_{\max},
	\] 
and 
	\begin{align*}
	L(f)&=\max\{0,l(f)\}\qquad \text{with}\quad l(f):=\min\{S_{\max},\sti{c}^{-1}(f)\}\\
	&=
	\begin{cases}
	0 & \qquad f\in[0,\st{c})\\[2mm]
	\in(0,S_{\max})\quad \text{ strictly monotone increasing}  
	&\qquad f\in\left(\st{c},\sti{c}(S_{\max})\right)\\[2mm]
	S_{\max} & \qquad f\ge \sti{c}(S_{\max}),
	\end{cases}
	\end{align*}
	where $\sti{c}^{-1}(f) =\eqref{st1}$ or $\eqref{st2}$ for 
	$i=1,2$ and $l(\st{c})=0$.
\end{corollary}

\begin{proof}
	In order to verify the assumptions of Theorems~\ref{Thm:pop_dnymics_existence} and 
	\ref{Thm:singular_limit}, we observe first that 
	both consume rate functions $c_1(S,N,F)$ and $c_2(S,N,F)$ 
	are piecewise $C^1$-functions as long as $N>0$
	and absolutely continuous in the points where the functions are glued together. Hence, $c_1$ and $c_2$
	are locally Lipschitz continuous in $N$ for all $S,F\geq 0$ and $N>0$.
	This lower bound for $N_\varepsilon$ is given by \eqref{Nbounda} or respectively \eqref{Nboundb} in the case of consumption rate $c_1$, and by \eqref{Nbounda} in the case of consumption rate $c_2$, where $\lambda_i(S,N,F)= \frac{c_i(S,N,F)}{c_{\min}}-1$, $i=1,2$. 
	Equally, local Lipschitz continuity holds for the mappings $S\mapsto c_i(S,N,F)$ and $F\mapsto c_i(S,N,F)$, $i=1,2$ and for every $N>0$. 
	
	Finally, the above characterisations of $U$ and $L$ 
	follow directly from the definitions. 
	%
	%
\end{proof}

\section{Limit $\varepsilon \rightarrow 0$}\label{limit}

This section is devoted to the proof of Theorem~\ref{Thm:singular_limit} for 
both consume functions $c_1$ or $c_2$, where the case $c_1$ requires the a-priori estimates stated in Proposition~\ref{c1Smax}.
Moreover, we recall the form of the growth rate function $\lambda = \frac{c}{c_{\min}} - 1$ for $c_{\min} >0$.


In performing the limit $\varepsilon\to 0$, we denote $N_\varepsilon:=N(u_\varepsilon)$ and $f_\varepsilon:= F/N_\varepsilon$.
A main difficulty in the limit $\varepsilon\rightarrow 0$ are 
lacking bounds of the derivative of $S_\varepsilon$.
One main element of our proof is to introduce a projection operator $p_{\varepsilon}$ onto the ($\varepsilon$ independent) area $\dot{S}_\varepsilon=0$ 
in the $f_\varepsilon$-$S_\varepsilon$ phase space diagram, cf. \cite{kuehn2017generalized}. In order to appropriately define this projection $p_{\varepsilon}$,
we require the uniform in $\varepsilon$ bounds on the stock level $S_\varepsilon$.

The main advantage of the projection operator $p_{\varepsilon}$ is to bypass the unbounded 
derivative of $S_\varepsilon$ in the limit $\varepsilon\rightarrow 0$. More precisely, since the upper and lower threshold functions $U(f)$ and $L(f)$ (which envelop the area $\dot{S}_{\varepsilon}=0$) have finite slope, 
the derivatives of $p_{\varepsilon}$ are bounded in norm independently of $\varepsilon$
in terms of the regularity of its driving input variable 
$f_{\varepsilon}(t)=F/N_\varepsilon$, i.e. by 
$F,N_\varepsilon\in\mathrm{W}^{1,\infty}(0,T)$ since 
$N_\varepsilon$ is bounded below independently of $\varepsilon>0$.

Finally, before we prove Theorem~\ref{Thm:singular_limit}, let us illustrate that the 
limit $\lm{S}$ evolves as a generalised play operator between the curves ${U}$ and $L$ with input $\lm{f}=F/\lm{N}$,
see e.g. Figure~\ref{pop_pic_limit_phase_diagram_1} in the case of $c_1$ or compare Figure~\ref{pop_pic_projection} in the case $c_2$. 
Let $\lm{f}=F/\lm{N}$ be a piecewise monotone function
on some intervals $[t_1,t_2],[t_2,t_3],\ldots\subset [0,T]$,
then $\lm{S}$ behaves according to the following cases: 
\begin{align*}
\text{If }\lm{f}(t_1) \in[0,\tilde{c})\ &\text{and}\  
\lm{S}(t_1) = u(\lm{f}(t_1))\\ 
&\text{then for }t>t_1
\begin{cases}
\lm{S}(t)= u(\lm{f}(t)) & \text{as long as}\quad \lm{f}(t) \subset [0,\lm{f}(t_1))\searrow \\
\lm{S}(t)\equiv  \lm{S}(t_1) & \text{as long as}\quad \lm{f}(t)  \subset (\lm{f}(t_1),\lm{f}(t_2)] \nearrow.
\end{cases}
\end{align*}
In the second case, if $\lm{f}(t)$ increases 
until $\lm{S}(t_1)=\lm{S}(t_2)=l(\lm{f}(t_2))$ at some time $t_2$, 
then, for $\lm{f}$ monotone on $[t_2,t_3]\subset [0,T]$
we have $\lm{S}(t_2) = l(\lm{f}(t_2))$ and $\lm{f}(t_2) \ge \st{c}$
\begin{align*}
\qquad
\text{and for }t>t_2 
\begin{cases}
\lm{S}(t)\equiv  \lm{S}(t_2)& \text{as long as}\quad \lm{f}(t) \subset [c_{\mathrm{de}}(\lm{S}(t_2)),\lm{f}(t_2))\searrow \\
\lm{S}(t) =  l(\lm{f}(t)) & \text{as long as}\quad \lm{f}(t)  \subset (\lm{f}(t_2),\lm{f}(t_3)] \nearrow
\end{cases}
\end{align*}
where in the first case the maximal possible decrease of $\lm{f}$ down to $\lm{S}(t_3)=c_{\mathrm{de}}(\lm{S}(t_2))=u(\lm{f}(t_3))$ 
leads back to the beginning of the first case above with $\lm{f}(t_3) \in[c_{\min},\tilde{c}]$.

%
%

\begin{proof}[Proof of Theorem~\ref{Thm:singular_limit}]
	
We consider a sequence $\varepsilon>0$ which tends to zero and recall that Proposition \ref{c1Smax} implies 
uniform bounds of $S_\varepsilon(t)$ in $\varepsilon$ and $t$  for the consume rate functions $c_1$ under suitable assumptions on external food supply $F(t)$
while such uniform bounds are true for $c_2$ by construction. Note that we shall denote both $c_1$ or $c_2$ simply by $c$. 
Hence, for $c$ being either $c_1$ or $c_2$, there exists a constant $S_{\max}>0$ such that
$$
S_\varepsilon(t)\le S_{\max}>0, \qquad \text{for all } \varepsilon >0, t\ge0.
$$

	Moreover, we already know that $S_\varepsilon$ is nonnegative independently of $\varepsilon$.
	During the whole proof, we suppress the dependence of $l$ and $L$ on $i\in \{1,2\}$.
	
	Let $p_{\varepsilon}(S_\varepsilon,N_\varepsilon,F)$ be the following projection operator in the $f_\varepsilon$-$S_\varepsilon$ phase space, cf. Figure~\ref{pop_pic_projection}:
	\begin{equation}\label{p}
	p_{\varepsilon}(S_\varepsilon,f_\varepsilon):= \begin{cases}
	l(f_\varepsilon) & \text{if } S_\varepsilon\leq \sti{c}^{-1}(f_\varepsilon),\quad i=1,2,\\
	u(f_\varepsilon) & \text{if } S_\varepsilon\geq u(f_\varepsilon) \text{ and } f_\varepsilon \leq \tilde{c},\\
	S_\varepsilon & \text{else.}
	\end{cases}
	\end{equation}
	Note that $0\leq S_\varepsilon \leq {S}_{\max}$ and also $0\leq p_{\varepsilon} \leq {S}_{\max}$.
	In particular, $p_{\varepsilon}(S_\varepsilon,f_\varepsilon)=0\neq S_\varepsilon$ follows always due to a projection downwards, whereas $p_{\varepsilon}(S_\varepsilon,f_\varepsilon)={S}_{\max}\neq S_\varepsilon$ follows always due to a projection upwards.
	
	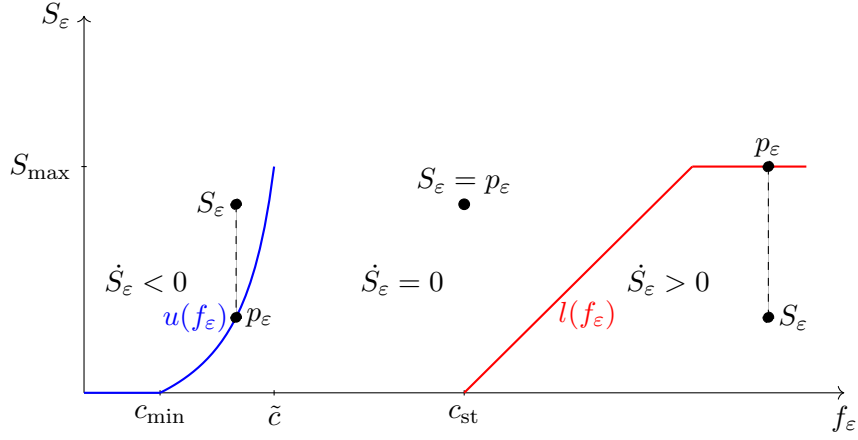
\begin{figure}
		\center
		\begin{tikzpicture}[]
		\draw[->] (0,0) -- (10,0) coordinate (x axis);
		\draw (-1pt,5 ) node[anchor=east,fill=white] {$S_\varepsilon$};
		\draw (10 cm,-1pt) node[anchor=north] {$f_\varepsilon$};
		\draw[->] (0,0) -- (0,5) coordinate (y axis);
		\draw (1 cm,1pt) -- (1 cm,-1pt) node[anchor=north] {$c_{\min}$};
		\draw (5 cm,1pt) -- (5 cm,-1pt) node[anchor=north] {$\st{c}$};
		\draw (2.5 cm,1pt) -- (2.5 cm,-1pt) node[anchor=north] {$\tilde{c}$};
		\draw (1pt,3 ) -- (-1pt,3 ) node[anchor=east,fill=white] {$S_{\max}$};
		\draw [blue, thick, domain=0:1] plot (\x, {0});
		\draw [blue, thick, domain=1:2.5] plot (\x, {(\x-1)/(3-\x)});
		\draw [blue] (2 cm +1pt,1cm) node[anchor=east] {$u(f_{\varepsilon})$};
		\draw [red, thick, domain=5:((2.5-1)/(3-2.5)+5)] plot (\x, {(\x-5});
		\draw [red] (6cm +3pt,1cm+1pt) node[anchor=west] {$l(f_{\varepsilon})$};
		\draw [red, thick, domain=((2.5-1)/(3-2.5)+5):9.5] plot (\x, {(2.5-1)/(3-2.5)});
		\draw  [black] (3.5,1.5) node[anchor=west] {$\dot{S_{\varepsilon}}=0$};
		\draw  [black] (1.5,1.5) node[anchor=east] {$\dot{S_{\varepsilon}}<0$};
		\draw  [black] (7,1.5) node[anchor=west] {$\dot{S_{\varepsilon}}>0$};
		\filldraw [dashed]
		(2,2.5) circle (2pt) node[align=left, left] {$S_{\varepsilon}$}--
		(2,1) circle (2pt) node[align=left, right] {$p_{\varepsilon}$};
		\filldraw [dashed]
		(9,3) circle (2pt) node[align=left, above] {$p_{\varepsilon}$}--
		(9,1) circle (2pt) node[align=left, right] {$S_{\varepsilon}$};
		\filldraw (5,2.5) circle (2pt) node[align=left, above] {$S_{\varepsilon}=p_{\varepsilon}$};
		\end{tikzpicture}
		\caption{Sign of the gradient of $S_\varepsilon$ and projection to $p_{\varepsilon}$ in the $f_\varepsilon$-$S_\varepsilon$-phase diagram in the case of $c_2$.}
		\label{pop_pic_projection}
	\end{figure}
	
\bigskip	
	
We divide the proof of the theorem into six steps.
\medskip
	
\noindent\underline{Step 1: Boundedness of $N_\varepsilon$ and $p_{\varepsilon}$ in $\mathrm{W}^{1,\infty}(0,T)$ independently of $\varepsilon>0$.}
\smallskip

The proof of Theorem \ref{Thm:pop_dnymics_existence} yielded $N_\varepsilon(t)\ge \delta(T)>0$  for $t\in[0,T]$, which implies together with $0\leq S_\varepsilon \leq {S}_{\max}$ and $F\in \mathrm{W}^{1,\infty}(0,T)$ that 
$\lambda={c}/{c_{\min}} - 1$ is bounded independently of $\varepsilon$.		
Hence, a Gronwall argument applied to \eqref{pop_dnymics_pop_evol4} yields boundedness of $N_\varepsilon\in C([0,T])\cap C^1(0,T)$ independently of $\varepsilon$.
Moreover, $f_\varepsilon=F/N_\varepsilon$ is equally Lipschitz continuous on $[0,T]$ with a modulus $L$ that is independent of $\varepsilon$. Since $u$ and $l$ are Lipschitz continuous on $[0,\tilde{c}]$ and $[\st{c},\infty)$ respectively, we conclude that $u(f_\varepsilon)$ is Lipschitz continuous for $f_\varepsilon\in[0,\tilde{c}]$ and $l(f_\varepsilon)$ is Lipschitz continuous for $f_\varepsilon\in [\st{c},\infty)$.
		
Next, we choose $\delta>0$ sufficiently small such that for all $t_1,t_2 \in [0,T]$ with $|t_1-t_2|<\delta$, there holds
\begin{equation}\label{asdf}
	|f_\varepsilon(t_1)-f_\varepsilon(t_2)| \leq L |t_1-t_2| < \frac{\st{c} - \tilde{c}}{2}.
\end{equation}
Let $t_0\in [0,T]$ be given. The definition of $p_{\varepsilon}$ implies that 
$p_{\varepsilon}(t_0)=S_\varepsilon(t_0)$ and  
$\dot{p}_\varepsilon(t_0) = \dot{S_\varepsilon}(t_0)=0$ almost surely whenever 
$f_\varepsilon\in [0,\tilde{c}]$ and $S_\varepsilon(t_0) < u(f_\varepsilon(t_0))$ or $\tilde{c}\leq f_\varepsilon(t_0) \leq \st{c}$ or $f_\varepsilon>\st{c}$ and $S_\varepsilon(t_0) > l(f_\varepsilon(t_0))$, see also Figure~\ref{pop_pic_projection}.

Moreover, if $f_{\varepsilon}(t_0)<\tilde{c}$, then \eqref{asdf} yields $f_{\varepsilon}(t)<\st{c}$ on $[t_0,t_0+\delta]$. Hence, since $S_\varepsilon\geq 0$, in this case
$p_{\varepsilon}(t)= \min\{u(f_\varepsilon(t)), S_\varepsilon(t)\}$ in $[t_0,t_0+\delta]$, and for a.e. $t\in [t_0,t_0+\delta]$ holds
\[
f_{\varepsilon}(t_0)<\tilde{c} \quad\Longrightarrow\quad
\dot{p}_{\varepsilon}(t)=
\begin{cases}
\frac{d}{dt} u(f_\varepsilon)(t)
&\ \text{ if } \quad
u(f_\varepsilon(t))\leq S_\varepsilon(t),\\
0 &\ \text{ if } \quad 
u(f_\varepsilon(t))> S_\varepsilon(t),
\end{cases}
\quad\text{on}\quad  [t_0,t_0+\delta].
\]
As a consequence $\dot{p}_{\varepsilon}$ is bounded a.e. on $[t_0,t_0+\delta]$ independently of $\varepsilon$.		

Finally, the case when $f_{\varepsilon}(t_0)>\st{c}$ is treated analogously.		
Consequently, $p_{\varepsilon}$ is bounded in $\mathrm{W}^{1,\infty}(0,T)$ independently of $\varepsilon>0$.
\medskip
		
\noindent\underline{Step 2: Convergence of $S_\varepsilon-p_{\varepsilon}$ to zero in $\mathrm{L}^q(0,T)$ for arbitrary $q \in (1,\infty)$.}\smallskip
		
		
		
For arbitrary $\varepsilon>0$ and $t\in[0,T]$, we have 
\begin{multline*}
|S_\varepsilon(t)-p_{\varepsilon}(t)| = |S_\varepsilon(0)-p_{\varepsilon}(0)| + \int_0^t (\dot{S_\varepsilon}(\tau) - \dot{p}_{\varepsilon}(\tau)) \frac{S_\varepsilon(\tau)-p_{\varepsilon}(\tau)}{|S_\varepsilon(\tau)-p_{\varepsilon}(\tau)|} d\tau\\
= |S_\varepsilon(0)-p_{\varepsilon}(0)| + \int_0^t \underbrace{\dot{S_\varepsilon}(\tau) \frac{S_\varepsilon(\tau)-p_{\varepsilon}(\tau)}{|S_\varepsilon(\tau)-p_{\varepsilon}(\tau)|}}_{\le 0} d\tau - \int_0^t \dot{p}_{\varepsilon}(\tau) \frac{S_\varepsilon(\tau)-p_{\varepsilon}(\tau)}{|S_\varepsilon(\tau)-p_{\varepsilon}(\tau)|} d\tau,
\end{multline*}
where the above inequality follows from the 
definition of $p_{\varepsilon}$ 
for all $\tau$, see Figure~\ref{pop_pic_projection}.
Hence,
\begin{equation}\label{pop_dnymics_stock_estimate}
|S_\varepsilon(t)-p_{\varepsilon}(t)| + \int_0^t \left|\dot{S_\varepsilon}(\tau) \frac{S_\varepsilon(\tau)-p_{\varepsilon}(\tau)}{|S_\varepsilon(\tau)-p_{\varepsilon}(\tau)|} \right|d\tau 
\leq |S_\varepsilon(0)-p_{\varepsilon}(0)| + \int_0^t |\dot{p}_{\varepsilon}(\tau)|d\tau,
\end{equation}
and the bounds of Step 1 imply that the right side of \eqref{pop_dnymics_stock_estimate} and thus $|S_\varepsilon(t)-p_{\varepsilon}(t)|$ is bounded independently of $\varepsilon$.
		
Similar, we calculate 
\begin{align*}
(S_\varepsilon(t)-p_{\varepsilon}(t))^2 &= (S_\varepsilon(0)-p_{\varepsilon}(0))^2 + 2\int_0^t (\dot{S_\varepsilon}(\tau) - \dot{p}_{\varepsilon}(\tau)) (S_\varepsilon(\tau)-p_{\varepsilon}(\tau)) d\tau,
\end{align*}
and
\begin{multline*}
(S_\varepsilon(t)-p_{\varepsilon}(t))^2 - 2\int_0^t \dot{S_\varepsilon}(\tau) (S_\varepsilon(\tau)-p_{\varepsilon}(\tau)) d\tau\\ =(S_\varepsilon(0)-p_{\varepsilon}(0))^2 - 2\int_0^t \dot{p}_{\varepsilon}(\tau) (S_\varepsilon(\tau)-p_{\varepsilon}(\tau)) d\tau.
\end{multline*}
For $t=T$, the above boundedness of $|S_\varepsilon-p_{\varepsilon}|$ and $\dot{p}_{\varepsilon}$ together with
$\dot{S_\varepsilon}(\tau) (S_\varepsilon(\tau)-p_{\varepsilon}(\tau)) \leq 0$
yields that $\dot{S_\varepsilon}(S_{\varepsilon}-p_{\varepsilon})$ is bounded in $L^1(0,T)$ independent of $\varepsilon$. Hence, 
$$
\varepsilon \|\dot{S_\varepsilon} (S_\varepsilon-p_{\varepsilon})\|_{L^1(0,T)} \xrightarrow{\varepsilon\to 0} 0
\quad\Longrightarrow\quad
 \varepsilon \dot{S_\varepsilon}(\tau)(S_{\varepsilon}(\tau)-p_{\varepsilon}(\tau)) \xrightarrow{\varepsilon\to 0} 0
\quad \text{for a.a. } \ \tau \in(0,T).
$$
By the definition of $\varepsilon	\dot{S_\varepsilon}=F(\tau) - c N_\varepsilon$, it follows 
\begin{align*}
(F(\tau) - c(S_\varepsilon(\tau),N_\varepsilon(\tau),F(\tau))N_\varepsilon(\tau))\,(S_{\varepsilon}(\tau)-p_{\varepsilon}(\tau))\xrightarrow{\varepsilon\to 0} 0, \qquad \text{a.e. } 
\tau\in (0,T),
\end{align*}
which implies 
$S_{\varepsilon}(\tau) - p_{\varepsilon}(\tau)  \xrightarrow{\varepsilon\to 0} 0$ whenever 
 $(F(\tau) - c(S_\varepsilon(\tau),N_\varepsilon(\tau),F(\tau))N_\varepsilon(\tau))$ 
should be bounded away from zero.  
		
%
		
Alternatively, if  
\[
F(\tau) - c(S_\varepsilon(\tau),N_\varepsilon(\tau),F(\tau))N_\varepsilon(\tau)
=(f_\varepsilon(\tau) - c(S_\varepsilon(\tau),N_\varepsilon(\tau),F(\tau)))N_\varepsilon(\tau) \xrightarrow{\varepsilon\to 0} 0
\]
it follows from the definition of $p_{\varepsilon}$ and $c$ for either $c_1$ or $c_2$, and because $N_\varepsilon(\tau)>\delta(T)$, that
\[
S_{\varepsilon}(\tau) - p_{\varepsilon}(\tau) \xrightarrow{\varepsilon\to 0} 0.
\]
		
Finally, since $|S_{\varepsilon}(\tau) - p_{\varepsilon}(\tau)| \leq 2S_{\max}$, Lebesgue's dominated convergence theorem yields that
$S_{\varepsilon} - p_{\varepsilon}$
converges to zero in $\mathrm{L}^q(0,T)$ for arbitrary $q\in(1,\infty)$.
		
\noindent\underline{Step 3: Weak convergence of a subsequence $\varepsilon\to 0$ and compact embeddings.}\smallskip
		
		
We recall from Step 1 that due to $0\leq S_\varepsilon \leq {S}_{\max}$, $F\in \mathrm{W}^{1,\infty}(0,T)$ and $N_{\varepsilon}>\delta(T)$
holds independently of $\varepsilon$, the growth rate function
$\lambda(S_{\varepsilon},N_{\varepsilon},F)=\bigl(\frac{c}{c_{\min}}-1\bigr)$
is bounded independently of $\varepsilon$, which implied that $N_\varepsilon$ is bounded 
in $C([0,T])$ independently of $\varepsilon$.
Moreover, the norms of $f_\varepsilon\in \mathrm{W}^{1,\infty}(0,T)$ are bounded independently of $\varepsilon$.
$0\leq S_\varepsilon \leq {S}_{\max}$ also implies that $S_\varepsilon$ is bounded with respect to the sup norm in $C([0,T])$,
independently of $\varepsilon$.
Moreover, with $\lambda$ bounded independently of $\varepsilon$, a Gronwall argument applied to the mild representation of $u_\varepsilon$ (see also the proof of Theorem~\ref{Thm:pop_dnymics_existence}) yields 
that the following norms of $u_\varepsilon$ are bounded independently of $\varepsilon$:
\begin{align*} \mathrm{C}([0,T];\mathrm{H}^2(\Omega))\cap\mathrm{C}((0,T];\mathrm{C}^{2,\beta}(\overline{\Omega}))\cap \mathrm{C}^{1}([0,T];\mathrm{L}^2(\Omega)) \cap \mathrm{C}^1((0,T];\mathrm{C}^{0,\beta}(\overline{\Omega}))\\
\hookrightarrow \mathrm{W}^{1,\infty}(0,T;\mathrm{L}^2(\Omega))\cap\mathrm{L}^{\infty}(0,T;\mathrm{H}^2(\Omega)).
\end{align*}
Finally, by Step 1, $p_{\varepsilon}$ is bounded in $\mathrm{W}^{1,\infty}(0,T)$ independently of $\varepsilon$.
		
Consequently, we can extract a subsequence $\{\varepsilon_k\}_k$ and find some 
\[
\lm{S}\in \mathrm{L}^q(0,T)\quad \text{ and }\quad
\lm{u}\in \mathrm{W}^{1,q}(0,T;\mathrm{L}^2(\Omega))\cap\mathrm{L}^{q}(0,T;\mathrm{H}^2(\Omega)),
\] 
such that $(p_{\varepsilon_k},u_{\varepsilon_k})$ converges to $(\lm{S},\lm{u})$ weakly in those spaces.
		
Next, we use the following embeddings, where $(\cdot,\cdot)_{\eta,1}$ denotes real interpolation:
\begin{align*}
\mathrm{W}^{1,q}((0,T);\mathrm{L}^{2}(\Omega))\cap \mathrm{L}^{q}((0,T);\mathrm{H}^2(\Omega)) &\hhookrightarrow \mathrm{C}^\beta((0,T); (\mathrm{L}^{2}(\Omega), \mathrm{H}^{2}(\Omega))_{\eta,1})\\
&\hookrightarrow \mathrm{C}^\beta([0,T]; \mathrm{L}^{2}(\Omega)),
\end{align*}
\begin{align*}
\mathrm{W}^{1,q}((0,T);\mathrm{L}^{2}(\Omega))\cap \mathrm{L}^{q}((0,T);\mathrm{H}^2(\Omega))
&\hhookrightarrow \mathrm{C}([0,T]; (\mathrm{L}^{2}(\Omega), \mathrm{H}^{2}(\Omega))_{\eta,q})\\
&\hookrightarrow \mathrm{C}([0,T]; \mathrm{L}^{2}(\Omega))
\end{align*}
for every $0 < \eta < 1-1/q$ and $0\leq \beta < 1/q'-\eta$,
see e.g. \cite[Theorem 3]{amann1995linear}. 
Moreover, $\mathrm{W}^{1,q}(0,T)$ is compactly embedded into $\mathrm{C}([0,T])$.
		
Hence, $u_{\varepsilon_k}$ converges strongly to $\lm{u}$ in $\mathrm{C}([0,T], \mathrm{L}^{2}(\Omega))$ and thus $N_{\varepsilon_k}$ converges strongly to $\lm{N}=N(\lm{u})$ in $\mathrm{C}([0,T])$ while $p_{\varepsilon}$ converges strongly to $\lm{S}$ in $\mathrm{C}([0,T])$.
Since by Step 2, 
$S_{\varepsilon} - p_{\varepsilon}$
converges to zero in $\mathrm{L}^q(0,T)$ as $\varepsilon \rightarrow 0$, we obtain that
\[
S_{\varepsilon_k} \xrightarrow{k\to \infty} \lm{S}\qquad \text{in}\quad \mathrm{L}^q(0,T).
\]		

\noindent\underline{Step 4: $(\lm{S},\lm{u})$ solves the limiting 
system \eqref{pop_limit_pop1}-\eqref{pop_limit_stock3}
and strong convergence.}
\smallskip
		
		
We will first show by Step 3 and dominated convergence that 
\begin{equation}\label{step4a}
\left(\frac{c(S_{\varepsilon_k},N_{\varepsilon_k},F)}{c_{\min}}-1\right)u_{\varepsilon_k}
\xrightarrow{\varepsilon_k \to 0}
\left(\frac{c	(\lm{S},\lm{N},F)}{c_{\min}}-1\right)\lm{u}
\quad \text{in}\quad \mathrm{L}^{q}((0,T),\mathrm{L}^{2}(\Omega)).
\end{equation}
By Step 3, the subsequence $u_{\varepsilon_k}$ converges to $\lm{u}$ in $\mathrm{C}([0,T],\mathrm{L}^{2}(\Omega))$.
Moreover, for a further subsequence $\varepsilon_k$ (again denoted by $\varepsilon_k$), it follows 
from $S_{\varepsilon_k}\rightarrow \lm{S}$ in $\mathrm{L}^q(0,T)$
that
\[
\left(\frac{c(S_{\varepsilon_k}(t),N_{\varepsilon_k}(t),F(t))}{c_{\min}}-1\right)
\xrightarrow{\varepsilon_k \to 0} \left(\frac{c(\lm{S}(t),\lm{N}(t),F(t))}{c_{\min}}-1\right) \qquad \text{a.e. in }  [0,T],
\]
where we have used that $c$ is locally Lipschitz continuous due to $N_\varepsilon(t),\lm{N}(t)>\delta(T)>0$ and $F(t) \geq 0$ (recall also   
the conditions of Proposition~\ref{c1Smax} in the case $c=c_1$). Together with $u_{\varepsilon_k}(x,t)\rightarrow\lm{u}(x,t)$ for all $t\in [0,T]$ and a.a. $x\in \Omega$ (w.l.o.g. for the same subsequence $\varepsilon_k$), we obtain 
\[
\left(\frac{c(S_{\varepsilon_k}(t),N_{\varepsilon_k}(t),F(t))}{c_{\min}}-1\right) u_{\varepsilon_k}(x,t)
\xrightarrow{\varepsilon_k \to 0}
\left(\frac{c(\lm{S}(t),\lm{N}(t),F(t))}{c_{\min}}-1\right)\lm{u}(x,t)
\]
for a.a. $t\in [0,T]$ and a.a. $x\in \Omega$.
Moreover, we have derived in Step 1 that
\[
\left|\left(\frac{c(S_{\varepsilon_k},N_{\varepsilon_k},F)}{c_{\min}}-1\right)\right|\leq C
\]
uniformly in $[0,T]$ for some $C>0$ and this estimate holds also when $S_{\varepsilon_k}$ is replaced by any $S$ with $0\leq S \leq S_{\max}$.
Hence, also
\[
\left|\left(\frac{c(\lm{S},\lm{N},F)}{c_{\min}}-1\right)\right|\leq C
\]
is uniformly bounded in $[0,T]$. Therefore, since $u_{\varepsilon_k}(.,t)$ and $\lm{u}(.,t)$ have a common upper bound in $\mathrm{L}^2(\Omega)$ by Step~3, Lebesgue's dominated convergence theorem yields for a.e. $t\in [0,T]$,
\[
\left\|\biggl(\frac{c(S_{\varepsilon_k}(t),N_{\varepsilon_k}(t),F(t))}{c_{\min}}-1\biggr) u_{\varepsilon_k}(\cdot,t)
	-
\biggl(\frac{c(\lm{S}(t),\lm{N}(t),F(t))}{c_{\min}}-1\biggr)\lm{u}(\cdot,t)\right\|_{\mathrm{L}^2(\Omega)}\!\!\!
\xrightarrow{\varepsilon_k \to 0} 0.
\]
Finally, since this sequence is also bounded uniformly in $t\in [0,T]$, using Lebesgue's dominated convergence theorem again implies convergence in $\mathrm{L}^q(0,T)$ and thus \eqref{step4a}.
\smallskip
%
		
		Next, we observe that the Neumann realisation of $-D \Delta$ satisfies maximal parabolic Sobolev regularity on $\mathrm{L}^2(\Omega)$, see e.g. \cite[Theorem 2.9]{disser2015h} as a recent reference with a state-of-the-art context.
As a consequence, 
\begin{align*}
u_{\varepsilon_k} &= (\partial_t -D\Delta)^{-1} \biggl(\frac{c(S_{\varepsilon_k},N_{\varepsilon_k},F)}{c_{\min}}-1\biggr)u_{\varepsilon_k}
\xrightarrow{\varepsilon_k \to 0}
 (\partial_t -D\Delta)^{-1} \biggl(\frac{c(\lm{S},\lm{N},F)}{c_{\min}}-1\biggr)\lm{u}
\end{align*}
in $\mathrm{W}^{1,q}((0,T),\mathrm{L}^{2}(\Omega))\cap \mathrm{L}^{q}((t_0,T),\mathrm{H}^2(\Omega))$.
Since $u_{\varepsilon_k}$ converges to $\lm{u}$ in $\mathrm{L}^{q}((0,T),\mathrm{L}^{2}(\Omega))$, this shows also
\begin{align*}
\lm{u} &= (\partial_t -D\Delta)^{-1} \left(\frac{c(\lm{S},\lm{N},F)}{c_{\min}}-1\right)\lm{u}
\end{align*}
and that the weak convergence of $u_{\varepsilon_k}$ is actually strong. As a consequence, $\lm{u}$ solves the limiting evolution equation 
\begin{alignat*}{2}
\partial_t \lm{u} -D\Delta \lm{u} &= \left(\frac{c(\lm{S},\lm{N},F)}{c_{\min}}-1\right)\lm{u} &&\qquad \text{in } (0,T)\times\Omega,\\
\partial_\nu \lm{u} &= 0 &&\qquad \text{on } (0,T)\times\partial\Omega,\\
\lm{u}(0) &= u_{in}&&\qquad \text{on } \Omega.
\end{alignat*}
\medskip
		
		
		\begin{figure}
			\begin{tikzpicture}[] 
			\draw[->] (0,0) -- (10,0) coordinate (x axis);
			\draw (-1pt,5 ) node[anchor=east,fill=white] {$\lm{S}$};
			\draw (10 cm,-1pt) node[anchor=north] {$\lm{f}$};
			\draw[->] (0,0) -- (0,5) coordinate (y axis);
			\draw (1 cm,1pt) -- (1 cm,-1pt) node[anchor=north] {$c_{\min}$};
			\draw (5 cm,1pt) -- (5 cm,-1pt) node[anchor=north] {$\st{c}$};
			\draw (2.5 cm,1pt) -- (2.5 cm,-1pt) node[anchor=north] {$\tilde{c}$};
			\filldraw [dashed]
			(4.5, 2) circle (2pt) node[anchor=south] {\tiny $(\lm{f}(t_0),\lm{S}(t_0))$}--
			(5.75,0.75) circle (2pt) node[anchor=north west] {\tiny $(x_l,l(x_l))$}
			node[anchor=north, pos=0.5]{\tiny$\delta_{l}$};
			\filldraw [dashed]
			(4.5, 2) node[anchor=south, above] {}--
			(2.4,1.4/0.6) circle (2pt) node[align=right, left] {\tiny $(x_u,u(x_u))$}
			node[anchor=north, pos=0.5]{\tiny $\delta_{u}$};
			\draw (1pt,3 ) -- (-1pt,3 ) node[anchor=east,fill=white] {$S_{\max}$};
			\draw [blue, thick, domain=0:1] plot (\x, {0});
			\draw [blue, thick, domain=1:2.5] plot (\x, {(\x-1)/(3-\x)});
			\draw [blue] (2 cm +2pt,1cm+1pt) node[anchor=east] {$u$};

			\draw [red, thick, domain=5:((2.5-1)/(3-2.5)+5)] plot (\x, {(\x-5});
			\draw [red] (6cm +3pt,1cm+1pt) node[anchor=west] {$l$};
			\draw [red, thick, domain=((2.5-1)/(3-2.5)+5):9.5] plot (\x, {(2.5-1)/(3-2.5)});
			\end{tikzpicture}
\caption{$\lm{f}$-$\lm{S}$-phase diagram: Example for $(\lm{f}(t_0),\lm{S}(t_0))$ and $x_u$ and $x_l$}
	\label{pop_pic_limit_phase_diagram}
\end{figure}
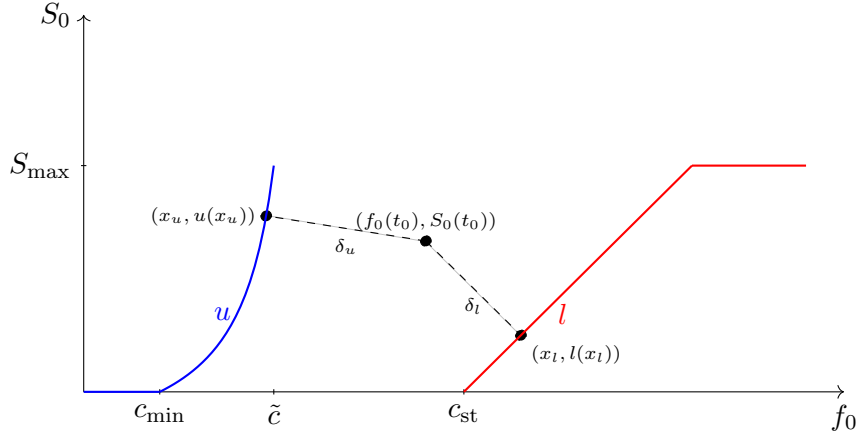
		
We now analyse the limiting behaviour of $\lm{S}$ in dependence on
\[
\lm{f}=F/\lm{N}.
\]
Assume first that at a time $t_0\in [0,T]$,
the point $(\lm{f}(t_0),\lm{S}(t_0))$ is located in the $\lm{f}$-$\lm{S}$-phase space diagram between the graphs $(x,l(x))$ for $x\geq \st{c}$ and $(x,u(x))$ with $x\leq \tilde{c}$, cf. Figure~\ref{pop_pic_limit_phase_diagram} (in the case of $c=c_2$)
and recall that $\lm{S}\le S_{\max}$.

We will show that this assumption implies $\dot{\lm{S}}=0$ a.e. on a sufficiently small interval $J$ with $t_0\in J$:
For the point $(\lm{f}(t_0),\lm{S}(t_0))$, we introduce {any} nearest point $(x_u,u(x_u))$ and the associated distance $\delta_u$, i.e. 
\[
{\delta}_{u}=\|(x_u,u(x_u)))-(\lm{f}(t_0),\lm{S}(t_0))\| 
\quad\text{with}\quad x_u{\in} \!\!\!
 \argmin\limits_{(x,u(x)),\, x\leq \tilde{c}}\|(x,u(x))-(\lm{f}(t_0),\lm{S}(t_0))\|.
\]
{Note that $(x_u,u(x_u))$ is unique due to the convexity of $u$. However, the following argument holds also for general increasing functions $u$.}
Analog, we defined {any} nearest point $(x_l,l(x_l))$ and its distance ${\delta}_l$.
{Again, $(x_l,l(x_l))$ is unique due to the concavity of $l$, but the argument holds for general increasing functions $l$.}
		
Due to the uniform convergence of $f_{\varepsilon_k}$ to $\lm{f}$ and of $p_{\varepsilon_k}$ to $\lm{S}$, as shown in Step~3, we can choose $\varepsilon_0$ sufficiently small such that for all $\varepsilon_k<\varepsilon_0$ holds
\[
	|(f_{\varepsilon_k},p_{\varepsilon_k})-(\lm{f},\lm{S})|<\frac{1}{4}\min \{\delta_u,{\delta}_l\}
\]
uniformly in $t\in [0,T]$.  Hence, for all $\varepsilon_k<\varepsilon_0$,
the point $(f_{\varepsilon_k}(t_0),p_{\varepsilon_k}(t_0))$ has a distance larger than 
$\frac{3}{4}\min \{\delta_u,{\delta}_l\}$
to the graphs $(x,l(x))$ for $x\geq \st{c}$ and $(x,u(x))$ with $x\leq \tilde{c}$.

As a consequence, there exists an open interval $I\ni t_0$ such that  for all $t\in I$ the distance between $(f_{\varepsilon_k}(t),p_{\varepsilon_k}(t))$ and the graphs $(x,l(x))$ and $(x,u(x))$ is greater than
$\frac{1}{2}\min \{\delta_u,{\delta}_l\}$.
From Steps 2 and 3, we know that $(f_{\varepsilon_k}(t),S_{\varepsilon_k}(t))$ converges to $(f_{\varepsilon_k}(t),p_{\varepsilon_k}(t))$  a.e. in $I$.
Moreover, $f_{\varepsilon_k}$ converges  uniformly to $\lm{f}$ and is Lipschitz continuous with a modulus independent of $\varepsilon_k$.
Therefore, there is some $t_1\in I$, $t_1<t_0$, with
\[
|(f_{\varepsilon_k}(t_1),S_{\varepsilon_k}(t_1))-(f_{\varepsilon_k}(t_1),p_{\varepsilon_k}(t_1))|< \frac{1}{4}\min \{\delta_u,{\delta}_l\}
\]
for all $\varepsilon_k$  sufficiently small (i.e. by eventually choosing $\varepsilon_0$ smaller).
Moreover, for each $\varepsilon_k$ and for some $t\in [0,T]$, if $(f_{\varepsilon_k}(t),S_{\varepsilon_k}(t))$ is located between the graphs $(x,l(x))$ and $(x,u(x))$, then $S_{\varepsilon_k}$ remains constant until the first time $\tilde{t}>t$ when 
\[
	S_{\varepsilon_k}(\tilde{t})=u\bigl(f_{\varepsilon_k}(\tilde{t})\bigr)\qquad \text{or} \qquad
	S_{\varepsilon_k}(\tilde{t})=l\bigl(f_{\varepsilon_k}(\tilde{t})\bigr).
\]
With the Lipschitz modulus of $f_{\varepsilon_k}$ being independent of  $\varepsilon_k$, 
the trajectory $(S_{\varepsilon_k},f_{\varepsilon_k})$ keeps for all $\varepsilon_k<\varepsilon_0$ a positive distance to the graphs $(x,l(x))$ and $(x,u(x))$ 	
on an interval $J:=I\cap [t_1,T]\ni t_0$ if $I$ is chosen sufficiently small.
Furthermore on $J$, it follows by definition that 
\[
	S_{\varepsilon_k}=p_{\varepsilon_k}\qquad
	\text{and}\qquad 
	\dot{p}_{\varepsilon_k} = 0 
	\qquad\text{a.e. in}\quad J\quad \text{and for all}\quad \varepsilon_k < \varepsilon_0.
\]
Hence, for $t\in J$, 
\[
		\lm{S}(t)=\lim_{k\to\infty} p_{\varepsilon_k}(t) = \lim_{k\to\infty} p_{\varepsilon_k}(t_1) = \lm{S}(t_1)
\]
and $\dot{\lm{S}}=0$ a.e. in $J$ as claimed.
\medskip
		
We will now consider the case when $\lm{S}(t_0) = u(\lm{f}(t_0))$:
For $t\in [t_0-\delta,t_0+\delta]$ with $\delta$ chosen in \eqref{asdf} (which excludes $(\lm{f},\lm{S})$ to reach the graph $(x,l(x))$ in $[t_0-\delta,t_0+\delta]$), we know from the uniform convergence of $f_{\varepsilon_k}$ to $\lm{f}$, that $\dot{S}_{\varepsilon_k}(t)\leq 0$
 for $\varepsilon_k$ sufficiently small, cf. Figure~\ref{pop_pic_projection}.
We claim that therefore 
\[
	\dot{p}_{\varepsilon_k}\leq 0, \qquad \text{a.e. in}\quad [t_0-\delta,t_0+\delta].
\]
The proof assumes in contradiction that for some $t_1<t_2\in [t_0-\delta,t_0+\delta]$ and for some 
$0<\varepsilon_k<\varepsilon_1$ for $\varepsilon_1$ chosen sufficiently small
\[
	p_{\varepsilon_k}(t_2)>p_{\varepsilon_k}(t_1).
\]		
The ${\varepsilon_k}$-independent Lipschitz continuity of $p_{\varepsilon_k}$ (see Step 1), implies the existence of a time $t_3\in (t_1,t_2)$ and a constant $\delta_1>0$ such that
\[
	p_{\varepsilon_k}(t_2)-\delta_1 >p_{\varepsilon_k}(t)\qquad\text{for all}\quad t\in [t_1,t_3]\quad\text{and all}\quad 
	\varepsilon_k < \varepsilon_1,
\]
with $0<\varepsilon_1$ sufficiently small.
Due to the a.e. pointwise convergence of $S_{\varepsilon_k}$ to $p_{\varepsilon_k}$, 
there exists a time $t_4\in [t_1,t_3]$ such that 
\[
 S_{\varepsilon_k}(t_4) < p_{\varepsilon_k}(t_2)-\frac{\delta_1}{2}\qquad\text{for all}\quad  \varepsilon_k< \varepsilon_1,
\]
where $\varepsilon_1$ might again be chosen smaller.  
Now, because $\dot{S}_{\varepsilon_k}\leq 0$ on $[t_0-\delta,t_0+\delta]$, we have
\[
	S_{\varepsilon_k}(t) < p_{\varepsilon_k}(t_2)-\frac{\delta_1}{2}\qquad \text{for all}\quad t\in [t_4,t_2] \quad
	\text{and all} \quad \varepsilon_k< \varepsilon_1.
\]
In return, by using again that $S_{\varepsilon_k}$ converges a.e. to $p_{\varepsilon_k}$,		
\[
	p_{\varepsilon_k}(t) < p_{\varepsilon_k}(t_2)-\frac{\delta_1}{4}\qquad \text{for a.a.}\quad t\in [t_4,t_2]\quad
	\text{and all}\quad \varepsilon_k< \varepsilon_1,
\]
for $\varepsilon_1$ chosen sufficiently small. 
By the continuity of $p_{\varepsilon_k}$, this estimate holds for all $t\in [t_4,t_2]$, which gives the contradiction 
\[
		p_{\varepsilon_k}(t_2) < p_{\varepsilon_k}(t_2)-\frac{\delta_1}{4}.
\]
Hence, we have proven 
\[
\dot{p}_{\varepsilon_k}\leq 0\qquad \text{a.e. in}\quad [t_0-\delta,t_0+\delta]
\]
and for $\varepsilon_k$ sufficiently small.
Since $p_{\varepsilon_k}$ converges to $\lm{S}$ uniformly and weakly in $\mathrm{W}^{1,\infty}(0,T)$, this also yields
\[
\dot{\lm{S}}\leq 0\qquad\text{a.e. in}\quad  [t_0-\delta,t_0+\delta].
\]
		
Altogether, by combining the situations with $\dot{\lm{S}}=0$ a.e. and $\dot{\lm{S}}\le 0$ a.e., we have proven that
$\dot{\lm{S}}<0$ 
in a set of positive measure is only possible if a.e. in this set
$
\lm{S}=u(\lm{f}).
$ 
In the analog case when $\lm{S}(t_0) = l(\lm{f}(t_0))$, it follows similarly that
\[
	\dot{\lm{S}}\geq 0
\]
and that $\dot{\lm{S}}>0$
in a set of positive measure is only possible if a.e. in this set
$\lm{S} = l(\lm{f})$.

Summarising, we have shown that $\lm{S}$ satisfies for a.e. $t\in [0,T]$
\begin{align*}
&\dot{\lm{S}}(t)(\lm{S}(t)-z) \leq 0 \qquad 
\begin{cases}
\text{for all}\quad z\in [0,u(\lm{f}(t))] & \text{if}\quad \lm{f}(t)\leq \tilde{c},\\[1.7mm]
\text{for all}\quad z\in [l(\lm{f}(t)),S_{\max}] & \text{if}\quad \lm{f}(t)\geq \st{c},
\end{cases}\\
&\dot{\lm{S}}(t)=0,\hspace{7.5cm} \text{ if }\quad \tilde{c}< \lm{f}(t)< \st{c}.
\end{align*}
Because $0\leq \lm{S} \leq S_{\max}$, this shows that $\lm{u}$ and $\lm{S}$ solve \eqref{pop_limit_pop1}-\eqref{pop_limit_stock3}.
\medskip
		
\noindent\underline{Step 5: Uniqueness of solutions $(\lm{u},\lm{S})$ to the limiting system \eqref{pop_limit_pop1}-\eqref{pop_limit_stock3}.}\smallskip

The limit $\lm{S}$ is a generalised play for the Lipschitz continuous curves $U$ and $L$ with input $\lm{f}=F/\lm{N}$.
By \cite[III.2. Theorem 2.2]{visintin2013differential} this generalised play is a Lipschitz continuous hysteresis operator from $\mathrm{C}([0,T])\times \mathbb{R}$ to $\mathrm{C}([0,T])$. Since 
$\lm{N}(t)>\delta(T)>0$ uniformly in $[0,T]$, also $\lm{f}=F/\lm{N}$ is Lipschitz continuous.
The uniqueness of $(\lm{u},\lm{S})$ then follows by using a Gronwall argument, see e.g. \cite{visintin2013differential}.
		
The regularity properties of $\lm{u}$ follow essentially as in the proof of Theorem \ref{Thm:pop_dnymics_existence}. Since $\lm{f}\in \mathrm{W}^{1,\infty}(0,T)$, also $\lm{S}\in \mathrm{W}^{1,\infty}(0,T)$ by \cite[III.2. Theorem 2.3]{visintin2013differential}.
\medskip
		
\noindent\underline{Step 6: Convergence of the whole sequence $\varepsilon\to 0$.}
\smallskip
		
		Because every sequence $\{\varepsilon\}$ with $\varepsilon \rightarrow 0$ has a subsequence $\{\varepsilon_k\}_k$ such that  \[u_{\varepsilon_k} \rightarrow \lm{u}\] in $\mathrm{W}^{1,q}(0,T;\mathrm{L}^2(\Omega))\cap\mathrm{L}^{q}(0,T;\mathrm{H}^2(\Omega))$ and
		\[S_{\varepsilon_k} \rightarrow \lm{S}\]
		in $\mathrm{L}^{q}(0,T)$, uniqueness of the limit implies convergence of the whole sequence $(u_{\varepsilon},S_{\varepsilon})$.
\end{proof}

\section{Numerical examples and discussion 
}\label{examples}
In this section, we present selected examples of the behaviour of hysteresis-reaction-diffusion systems
on the one-dimensional domain $\Omega= (0,1)$.  

The first example depicts the behaviour of the population-hysteresis system \eqref{pop_limit_pop1}--\eqref{pop_limit_stock3} subject to a time-periodic food supply $F(t)$. 

The subsequent examples detail the interplay between the geometric properties of the two defining boundary curves $\mathcal{U}$ and $\mathcal{L}$ (see below) of a generalised play operator and 
a reaction-diffusion model. For the sake of a clear  discussion, those examples consider a 
simplest hysteresis-reaction-diffusion model and compare the numerical simulation with a 
Fourier analysis of the analytic solution. Generalisations of our observations to systems of hysteresis-reaction-diffusion equations are possible.

\subsubsection*{Numerical method} 
All numerical examples are implemented and simulated in Matlab using a uniform mesh with $99$ elements for the domain $\Omega = (0,1)$.
The discretised Neumann realisation of the Laplace operator $-\Delta$ is computed by a standard 3-point stencil finite difference scheme and the corresponding discretised eigenfunctions $\phi_k$, $k\geq 0$, together with their eigenvalues $\lambda_k$ can be computed by explicit formulas. The eigenfunctions are only used for the specification of initial data.

The hysteresis operator is approximated by an ODE regularisation of the variational inequality \eqref{Eq:Hyst_op}, similar as in \cite{brokate2013optimal} or \cite{kuehn2017generalized}.
As regularisation parameter we use $\varepsilon = 10^{-4}$.
The resulting approximative ODE-reaction-diffusion-equation is solved by a semi-implicit time-stepping procedure 
with implicit diffusion and explicit reaction with a step size $dt = 7.5*10^{-5}$.


\begin{figure}[b]
	\input{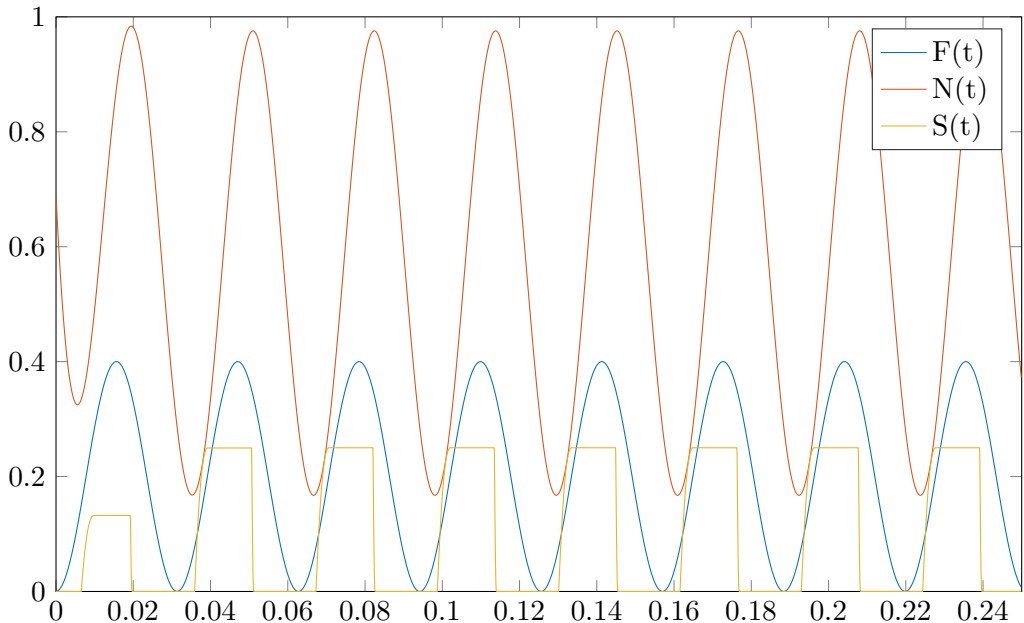}
\caption{Evolution of the population-hysteresis-diffusion system \eqref{pop_limit_pop1}--\eqref{pop_limit_stock3}:
Time-periodic food supply $F(t)=0.2(1-\cos(t))$ (blue) and the resulting population size $N$ (red) and stock levels $S$ (yellow).
	}\label{Fig:PlotNFS}
\end{figure}

\subsection{Simulation of the hysteresis-reaction-diffusion system \eqref{pop_limit_pop1}--\eqref{pop_limit_stock3}}\hfill\\
The first example illustrates the 
population dynamical model \eqref{pop_limit_pop1}--\eqref{pop_limit_stock3} in case of consume rate function $c_2$ with parameters $c_{\min}= 0.35$, $\de{c}^\infty=0.4$, $\st{c}=0.45$ and $S_{\max}= 0.3$. In particular, the food supply is given as $F(t)=0.2(1-\cos(t))$.
Moreover, the initial data where set to
$N_{in}= 7$,
$u_{in} = N_{in}\phi_1 + 2.5\phi_2 + 2.5\phi_3$ and
$S_{in} = 0$.

Figure~\ref{Fig:PlotNFS} depicts the evolution of the scalar quantities $F$, $N$ and $S$. More details can be observed in a simulation video (see supplementary material). The video shows the time evolution of $u,S,F$ and $F/N(u)$.
The upper plot in the video depicts the current location of $(F/N(u),S)$ (red dot) in phase space in relation to the upper (magenta) and lower (cyan) boundary curves of 
the limiting generalised play operator
and it can be verified that $S$ indeed approximates this generalised play operator.
The legend also shows the current value of $S$ (see also Figure~\ref{Fig:PlotNFS}) as well as the maximal value of $S$ during the previous cycle. The value is updated every time when $S$ starts to decrease. 
The lower plot shows the evolution of the population density $u$ (blue).

\subsubsection*{Discussion}
We observe that the hysteresis cycles gain initially amplitude before saturating. Qualitatively spoken, the system seems to behaves like a nonlinear oscillator, which adapts to the periodic external forcing within a transition phase. 

The simulation video shows in more detail the interplay between the amplitude of the hysteresis cycles, the current stock level $S$ (red), the individual food supply $f=F/N(u)$ (green) and the total food supply $F$ (yellow).
The legend also shows the current value of $N(|u-N(u)|)$, as well as its maximal value during the last cycle. Note that since $N(u-N(u))=0$, we can interpret  $N(|u-N(u)|)$ as a measure of the spatial inhomogeneity
of the population density $u$. 
As done for $S$, the video updates the values whenever $N(|u-N(u)|)$ starts to decrease. 

We observe that although diffusion is clearly dominant since the maximal value of $N(|u-N(u)|)$ decreases, the nonlinear coupling hysteresis-reaction leads nevertheless to large oscillations of $N(|u-N(u)|)$.
In fact, the following example will show that 
these nonlinear effects can be so strong as to prevent 
spatial homogenisation and lead to 
the growth of spatial inhomogeneities.

\subsection{Interplay between scalar hysteresis and a reaction-diffusion equation}

\subsubsection{A simplest hysteresis-reaction-diffusion model}\label{Subsec:Introduction of the model}\hfill\\
For $\Omega=(0,1)$ and $D>0$, we consider the hysteresis-reaction-diffusion equation
\begin{equation}\label{Eq:Reac_Diff_Syst}
\begin{aligned}
&\partial_t y -D\Delta y = R y 	&&\qquad \text{on } \quad \Omega\times (0,\infty),\\ 
&\partial_\nu y = 0				&&\qquad \text{on } \quad \partial\Omega\times (0,\infty),\\ 
&y(0) = y_0&&\qquad \text{on }\quad\Omega,
\end{aligned}
\end{equation}
where $R=R(\mathrm{T}y,R_0)$ (or $R(\mathrm{T}y)$ for short)
is a generalised scalar play operator defined according to Lipschitz continuous and strictly monotone increasing boundary curve functions $\mathcal{U}>\mathcal{L}:\mathbb{R}\rightarrow \mathbb{R}$, see e.g. \cite[Chapter III.2]{visintin2013differential}, i.e. 
\begin{equation}\label{Eq:Hyst_op}
\begin{aligned}
	&R(0)= \min \{\max \{\mathcal{L}(\mathrm{T}y(0)) , R_0\} , \mathcal{U}(\mathrm{T}y(0)) \} &&\qquad R_0\in \mathbb{R}\\
	&\dot{R}(t)(R(t)-z) \leq 0 \quad\text{ for all }\quad z\in [\mathcal{L}(\mathrm{T}y(t)), \mathcal{U}(\mathrm{T}y(t))] &&\qquad \text{for a.e. }t>0,\\
	&R(t)\in [\mathcal{L}(\mathrm{T}y(t)), \mathcal{U}(\mathrm{T}y(t))] 						&&\qquad \text{for }t\geq 0.
\end{aligned}
\end{equation}
In \eqref{Eq:Hyst_op}, $T$ is a linear and continuous functional  on $\mathrm{L}^2(\Omega)$ which is independent of time.
In particular, if defined on $C([0,\infty); \mathrm{L}^2(\Omega))$, we find that $(\mathrm{T}y)(t) = (\mathrm{T}y(\cdot,t))$ for $t>0$ and $y\in C([0,\infty); \mathrm{L}^2(\Omega))$ serves as input to a scalar hysteresis operator.  
Specifically, we consider $T$ to be a linear combination of Fourier coefficients of $y$
in terms of the eigenfunctions $\{\phi_k\}_{k\geq 0}$ of the Neumann realisation of the Laplacian $-\Delta$ (see e.g.  \cite{amann1995linear} and recall that the eigenfunctions $\{\phi_k\}_{k\geq 0}$ form an orthonormal basis of $\mathrm{L}^2(\Omega)$ while the eigenvalues satisfy $\lambda_0=0<\lambda_1 < \lambda_2 <\ldots$ and $\lambda_k\rightarrow \infty$ as $k\rightarrow \infty$).
Recalling that $\phi_0$ is a positive constant and
in view of the goal to study the interplay between hysteresis and reaction respectively diffusion, we  
consider $T$ 
as a linear combination of Fourier coefficients $\langle u ,\phi_k \rangle =\langle u ,\phi_k \rangle_{\mathrm{L}^2((0,1))}$ of spatially inhomogeneous eigenfunctions $\{\phi_k(x)\}_{k\geq 1}$:
\begin{equation}\label{Eq:W}
	\mathrm{T}y := \underbrace{k_m}_{>0} \langle y , \phi_m \rangle + \sum_{i=m+1}^{M-1} \underbrace{k_i}_{\ge 0} \langle y , \phi_i \rangle
	+ \underbrace{k_M}_{>0} \langle y , \phi_M \rangle, \qquad 1\le m <M, \quad m,M \in \mathbb{N}.
\end{equation}
Another generic example for $\mathrm{T}$ could be the mean value functional $\mathrm{T}y = |\Omega|^{-1} \int_\Omega y(x) \,dx$.

The system \eqref{Eq:Reac_Diff_Syst}--\eqref{Eq:W}
will be considered subject to non-trivial initial data
$0 \neq y_0\in \lbrace v\in \mathrm{H}^2(\Omega): \partial_\nu v = 0 \text{ on } \partial\Omega \rbrace$.
We choose $y_0$ to be a linear combination of eigenfunctions $\phi_k$.
In particular, we assume 
\begin{equation}\label{y0}
\begin{split}
&0<y_{0}^{(m)} = \langle y_0,  \phi_{m} \rangle, \qquad 0<y_{0}^{(M)} = \langle y_0, \phi_{M} \rangle\quad
\text{and}\\
&0\le y_{0}^{(k)} = \langle y_0 , \phi_{k} \rangle \quad \text{for}\quad  m<  k <  M,\quad\text{while}\quad 
0=y_{0}^{(k)}  \quad\text{for} \quad0\leq k < m.
\end{split}
\end{equation}

\begin{remark}
The assumption $y_{0}^{(k)} =0$ for $0\leq k < m$ is made because those eigenmodes would have no effect on $R(\mathrm{T}y)$. However, the eigenmodes to small eigenfunctions have a much stronger influence on the large-time behaviour of $y$ and would significantly complicate the interpretation of the results in this section. 

Note moreover that inverting the signs $y_{0}^{(m)}$, $y_{0}^{(M)} <0$ and $y_{0}^{(k)} \leq 0$ for $m< k < M$
as well as $k_m,k_M <0$ and $k_{m+1}, \ldots, k_{M-1} \leq 0$ will lead to the same qualitative behaviour 
except that $y$ is replaced by $-y$. 
\end{remark}

\subsubsection{Spatial homogenisation versus grow-up due to hysteresis}\label{dichoto}\hfill\\
In this section, we show that geometric properties of 
generalised play operators \eqref{Eq:Hyst_op} such as 
convexity/concavity or the slope of  
$\mathcal{U}$ resp. $\mathcal{L}$ can have a decisive 
influence on the evolution of the model 
\eqref{Eq:Reac_Diff_Syst}--\eqref{y0}. 
In particular, the scalar hysteresis operator \eqref{Eq:Hyst_op} can decide between spatial homogenisation or unbounded 
growth of spatially inhomogeneous Fourier modes.  
This dichotomy is illustrated in Figures \ref{Fig:phase_diag} and \ref{Fig:solution}.

\begin{figure}[h]
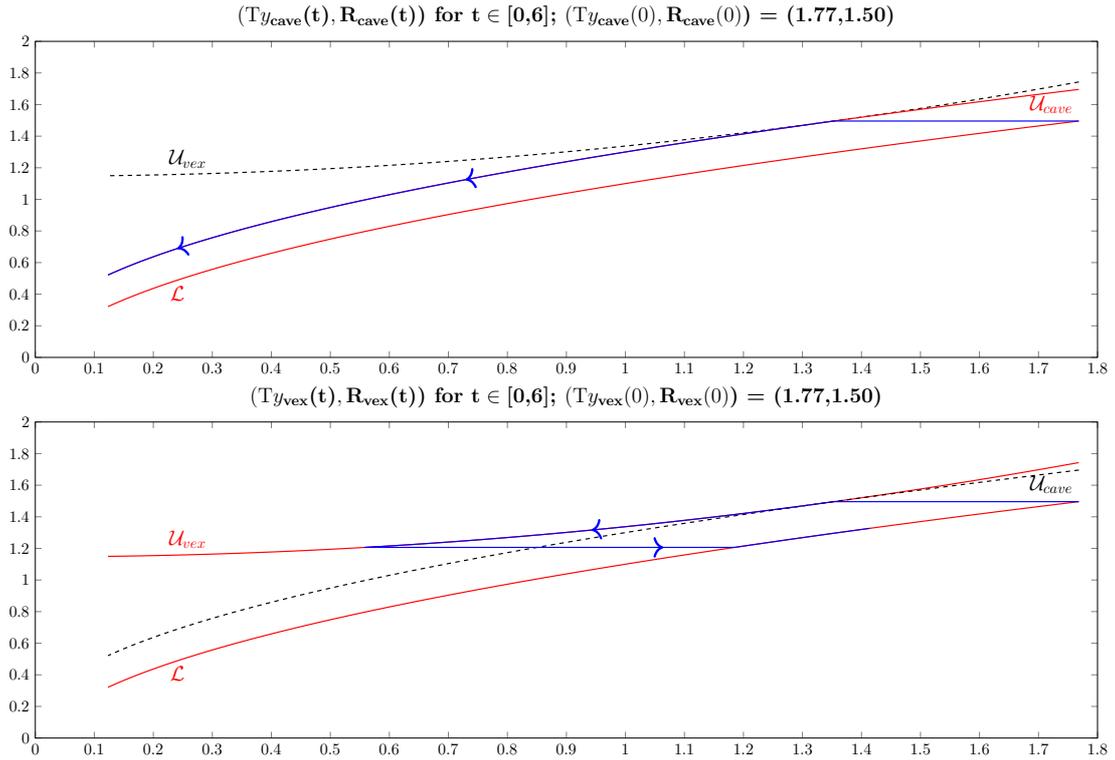

	\centering	
\include{phase_diagram}
\caption{
The top image shows the phase-space evolution of $(\mathrm{T}y_{cave},R_{cave})$ (blue line) subject to $\mathcal{U}_{cave}$ (upper red line) and $\mathcal{L}$ (lower red line). The black dashed line plots $\mathcal{U}_{vex}$ for the sake of comparison.
The bottom image shows the evolution of $(\mathrm{T}y_{vex},R_{vex})$ (blue line) subject to $\mathcal{U}_{vex}$ (upper red line) and $\mathcal{L}$ (lower red line). The black dashed line plots $\mathcal{U}_{cave}$ for the sake of comparison.
Both solutions start at the upper right end of the blue graphs at $(1.77,1.5)$ and continue identically until hitting $\mathcal{U}_{cave}$ resp. $\mathcal{U}_{vex}$ at a point of identical slope. While the decay of $(\mathrm{T}y_{cave},R_{cave})$ to zero yields spatial homogenisation, the turning of $(\mathrm{T}y_{vex},R_{vex})$ leads to 
unbounded growth. 
}\label{Fig:phase_diag}
\end{figure}

\subsubsection*{Description}
Figures~\ref{Fig:phase_diag}  and \ref{Fig:solution} 
depict a numerical simulation of system \eqref{Eq:Reac_Diff_Syst}--\eqref{y0} 
subject to initial data 
$$
y_0 = \frac{\phi_1 + \phi_2}{\|\phi_1 + \phi_2\|_{C([0,1])}}\qquad\text{ and }\qquad R_0=0. 
$$
Moreover, we have set the diffusion coefficient and the parameters of the functional $T$  in \eqref{Eq:W} to
$$
D = \frac{1}{\lambda_1}\qquad\text{and}\qquad 
m=1, M=2,\qquad  k_1 = 0.1, k_2= 0.4.
$$

As boundary curves of the generalised play operator \eqref{Eq:Hyst_op}, we
consider either a concave or a convex upper curve
$\mathcal{U}$ and a concave lower curve $\mathcal{L}$:
\begin{eqnarray*}
&\mathcal{U}_{cave}(z) = 1.2 |z|^{0.5}\mathrm{sign}(z) + 0.1,\qquad\text{resp.}\qquad
\mathcal{U}_{vex}(z)= 0.1097 |z|^{2}\mathrm{sign}(z) + 1.1468,\\
&\mathcal{L}(z) = 1.2 |z|^{0.5}\mathrm{sign}(z) - 0.1.
\end{eqnarray*}
Note that the lacking Lipschitz continuity of 
$\mathcal{U}_{cave}$ and $\mathcal{L}$ at zero is irrelevant since $\mathrm{T}y$ remains positive during the entire simulation. 
%

In the following, we denote by $(y_{vex},R_{vex})$ the solution of \eqref{Eq:Reac_Diff_Syst} subject to $\mathcal{U}_{vex}$ and by $(y_{cave},R_{cave})$ the solution for $\mathcal{U}_{cave}$.
Figure~\ref{Fig:phase_diag} compares the $\mathrm{T}y$-$R$ phase-space diagrams of two numerical solutions $(\mathrm{T}y_{vex},R_{vex})$ and $(\mathrm{T}y_{cave},R_{cave})$
starting both at the initial point $(1.77,1.5)$.
Hence, both solution trajectories move initially identically to the left at constant $R$-level 
and hit the upper boundary at the same time $t_+=0.15$ at a point 
where $\mathcal{U}_{cave}$ and $\mathcal{U}_{vex}$ share the same slope. 
With $\mathrm{T}y$ continuing to 
decrease, both solutions slide along their respective 
upper boundaries. 
While the top image in Figure~\ref{Fig:phase_diag} depicts the evolution of  $(\mathrm{T}y_{cave},R_{cave})$
according to the concave shape of $\mathcal{U}_{cave}$
the lower images shows $(\mathrm{T}y_{vex},R_{vex})$
following  $\mathcal{U}_{vex}$. 
\medskip

The key difference shown by Figure~\ref{Fig:phase_diag}  is that the solution 
$(\mathrm{T}y_{cave},R_{cave})$ continues to slide along $\mathcal{U}_{cave}$
and thus converges to zero (see discussion below), while $(\mathrm{T}y_{vex},R_{vex})$ features a turning point at time 
$t_0= 1.34$ when $\mathrm{T}y_{vex}$ starts to increase.  
In fact, $(\mathrm{T}y_{vex},R_{vex})$ remains increasing, first at constant $R$-level, later 
sliding along the lower curve $\mathcal{L}$ and will thus become unbounded (see discussion below).


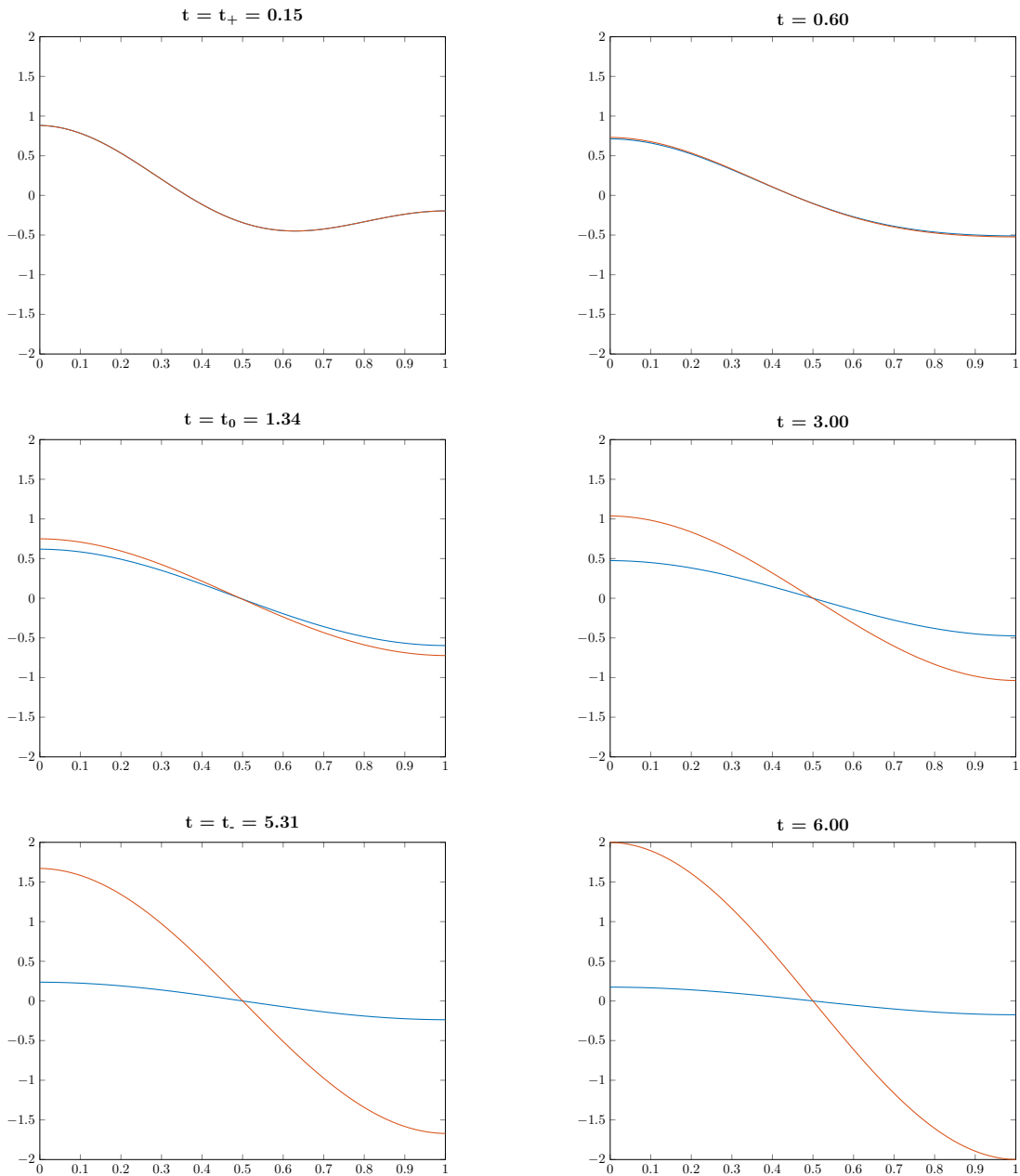
\begin{figure}[h]
	\centering
	
	\begin{subfigure}[b]{0.43\textwidth}
%
%
\definecolor{mycolor1}{rgb}{0.00000,0.44700,0.74100}%
\definecolor{mycolor2}{rgb}{0.85000,0.32500,0.09800}%
\begin{tikzpicture}[scale=0.50]

\begin{axis}[%
width=4.521in,
height=3.566in,
at={(0.758in,0.481in)},
scale only axis,
xmin=0,
xmax=1,
ymin=-2,
ymax=2,
axis background/.style={fill=white},
title style={font=\bfseries},
title={\Large $\text{t = t}_\text{+}\text{ = 0.15}$}
]
\addplot [color=mycolor1, forget plot, thick]
  table[row sep=crcr]{%
0	0.882132015183305\\
0.0101010101010101	0.880244994308845\\
0.0202020202020202	0.876476826337094\\
0.0303030303030303	0.870839237194056\\
0.0404040404040404	0.863349761781741\\
0.0505050505050505	0.854031679428158\\
0.0606060606060606	0.84291392821027\\
0.0707070707070707	0.830030998481953\\
0.0808080808080808	0.815422806019499\\
0.0909090909090909	0.799134545276168\\
0.101010101010101	0.781216523314291\\
0.111111111111111	0.761723975058173\\
0.121212121212121	0.740716860583301\\
0.131313131313131	0.718259645226794\\
0.141414141414141	0.694421063370369\\
0.151515151515152	0.669273866810097\\
0.161616161616162	0.642894558686663\\
0.171717171717172	0.615363114005384\\
0.181818181818182	0.586762687826848\\
0.191919191919192	0.557179312256287\\
0.202020202020202	0.526701583402701\\
0.212121212121212	0.495420339517026\\
0.222222222222222	0.463428331552159\\
0.232323232323232	0.430819887416314\\
0.242424242424242	0.397690571214853\\
0.252525252525253	0.3641368387943\\
0.262626262626263	0.330255690915697\\
0.272727272727273	0.296144325392665\\
0.282828282828283	0.261899789532524\\
0.292929292929293	0.227618634216572\\
0.303030303030303	0.193396570948101\\
0.313131313131313	0.159328133184056\\
0.323232323232323	0.125506343248335\\
0.333333333333333	0.0920223861018151\\
0.343434343434343	0.0589652912162314\\
0.353535353535354	0.0264216237662179\\
0.363636363636364	-0.00552481368376153\\
0.373737373737374	-0.0367932678630365\\
0.383838383838384	-0.0673063087609667\\
0.393939393939394	-0.0969900898387175\\
0.404040404040404	-0.125774593504409\\
0.414141414141414	-0.153593860928221\\
0.424242424242424	-0.180386205336385\\
0.434343434343434	-0.206094407988679\\
0.444444444444444	-0.230665896112851\\
0.454545454545455	-0.254052902140987\\
0.464646464646465	-0.276212603667019\\
0.474747474747475	-0.297107243620921\\
0.484848484848485	-0.316704230233539\\
0.494949494949495	-0.334976216445905\\
0.505050505050505	-0.351901158498247\\
0.515151515151515	-0.367462353516131\\
0.525252525252525	-0.381648455994188\\
0.535353535353535	-0.394453473161131\\
0.545454545454545	-0.405876739293147\\
0.555555555555556	-0.415922869125703\\
0.565656565656566	-0.424601690596206\\
0.575757575757576	-0.431928157231357\\
0.585858585858586	-0.437922240573145\\
0.595959595959596	-0.442608803115956\\
0.606060606060606	-0.44601745230386\\
0.616161616161616	-0.448182376211536\\
0.626262626262626	-0.449142161604183\\
0.636363636363636	-0.448939595140848\\
0.646464646464646	-0.447621448551626\\
0.656565656565657	-0.445238248681905\\
0.666666666666667	-0.441844033355966\\
0.676767676767677	-0.43749609406756\\
0.686868686868687	-0.43225470655643\\
0.696969696969697	-0.426182850376814\\
0.707070707070707	-0.419345918606654\\
0.717171717171717	-0.411811418884353\\
0.727272727272727	-0.403648666993297\\
0.737373737373737	-0.394928474242896\\
0.747474747474747	-0.385722829918458\\
0.757575757575758	-0.376104580090705\\
0.767676767676768	-0.366147104089125\\
0.777777777777778	-0.355923989951538\\
0.787878787878788	-0.345508710165208\\
0.797979797979798	-0.334974299012629\\
0.808080808080808	-0.324393032827585\\
0.818181818181818	-0.313836114454487\\
0.828282828282828	-0.303373363186173\\
0.838383838383838	-0.293072911432527\\
0.848484848484849	-0.28300090934444\\
0.858585858585859	-0.273221238584978\\
0.868686868686869	-0.263795236402203\\
0.878787878787879	-0.254781431116102\\
0.888888888888889	-0.246235290085693\\
0.898989898989899	-0.238208981171738\\
0.909090909090909	-0.230751148655829\\
0.919191919191919	-0.223906704518162\\
0.929292929292929	-0.217716635914242\\
0.939393939393939	-0.212217829625417\\
0.94949494949495	-0.207442914189638\\
0.95959595959596	-0.203420120347643\\
0.96969696969697	-0.200173160365944\\
0.97979797979798	-0.197721126722025\\
0.98989898989899	-0.196078410559214\\
1	-0.195254640239191\\
};
\addplot [color=mycolor2, forget plot, thick]
  table[row sep=crcr]{%
0	0.882132015183305\\
0.0101010101010101	0.880244994308845\\
0.0202020202020202	0.876476826337094\\
0.0303030303030303	0.870839237194056\\
0.0404040404040404	0.863349761781741\\
0.0505050505050505	0.854031679428158\\
0.0606060606060606	0.84291392821027\\
0.0707070707070707	0.830030998481953\\
0.0808080808080808	0.815422806019499\\
0.0909090909090909	0.799134545276168\\
0.101010101010101	0.781216523314291\\
0.111111111111111	0.761723975058173\\
0.121212121212121	0.740716860583301\\
0.131313131313131	0.718259645226794\\
0.141414141414141	0.694421063370369\\
0.151515151515152	0.669273866810097\\
0.161616161616162	0.642894558686663\\
0.171717171717172	0.615363114005384\\
0.181818181818182	0.586762687826848\\
0.191919191919192	0.557179312256287\\
0.202020202020202	0.526701583402701\\
0.212121212121212	0.495420339517026\\
0.222222222222222	0.463428331552159\\
0.232323232323232	0.430819887416314\\
0.242424242424242	0.397690571214853\\
0.252525252525253	0.3641368387943\\
0.262626262626263	0.330255690915697\\
0.272727272727273	0.296144325392665\\
0.282828282828283	0.261899789532524\\
0.292929292929293	0.227618634216572\\
0.303030303030303	0.193396570948101\\
0.313131313131313	0.159328133184056\\
0.323232323232323	0.125506343248335\\
0.333333333333333	0.0920223861018151\\
0.343434343434343	0.0589652912162314\\
0.353535353535354	0.0264216237662179\\
0.363636363636364	-0.00552481368376153\\
0.373737373737374	-0.0367932678630365\\
0.383838383838384	-0.0673063087609667\\
0.393939393939394	-0.0969900898387175\\
0.404040404040404	-0.125774593504409\\
0.414141414141414	-0.153593860928221\\
0.424242424242424	-0.180386205336385\\
0.434343434343434	-0.206094407988679\\
0.444444444444444	-0.230665896112851\\
0.454545454545455	-0.254052902140987\\
0.464646464646465	-0.276212603667019\\
0.474747474747475	-0.297107243620921\\
0.484848484848485	-0.316704230233539\\
0.494949494949495	-0.334976216445905\\
0.505050505050505	-0.351901158498247\\
0.515151515151515	-0.367462353516131\\
0.525252525252525	-0.381648455994188\\
0.535353535353535	-0.394453473161131\\
0.545454545454545	-0.405876739293147\\
0.555555555555556	-0.415922869125703\\
0.565656565656566	-0.424601690596206\\
0.575757575757576	-0.431928157231357\\
0.585858585858586	-0.437922240573145\\
0.595959595959596	-0.442608803115956\\
0.606060606060606	-0.44601745230386\\
0.616161616161616	-0.448182376211536\\
0.626262626262626	-0.449142161604183\\
0.636363636363636	-0.448939595140848\\
0.646464646464646	-0.447621448551626\\
0.656565656565657	-0.445238248681905\\
0.666666666666667	-0.441844033355966\\
0.676767676767677	-0.43749609406756\\
0.686868686868687	-0.43225470655643\\
0.696969696969697	-0.426182850376814\\
0.707070707070707	-0.419345918606654\\
0.717171717171717	-0.411811418884353\\
0.727272727272727	-0.403648666993297\\
0.737373737373737	-0.394928474242896\\
0.747474747474747	-0.385722829918458\\
0.757575757575758	-0.376104580090705\\
0.767676767676768	-0.366147104089125\\
0.777777777777778	-0.355923989951538\\
0.787878787878788	-0.345508710165208\\
0.797979797979798	-0.334974299012629\\
0.808080808080808	-0.324393032827585\\
0.818181818181818	-0.313836114454487\\
0.828282828282828	-0.303373363186173\\
0.838383838383838	-0.293072911432527\\
0.848484848484849	-0.28300090934444\\
0.858585858585859	-0.273221238584978\\
0.868686868686869	-0.263795236402203\\
0.878787878787879	-0.254781431116102\\
0.888888888888889	-0.246235290085693\\
0.898989898989899	-0.238208981171738\\
0.909090909090909	-0.230751148655829\\
0.919191919191919	-0.223906704518162\\
0.929292929292929	-0.217716635914242\\
0.939393939393939	-0.212217829625417\\
0.94949494949495	-0.207442914189638\\
0.95959595959596	-0.203420120347643\\
0.96969696969697	-0.200173160365944\\
0.97979797979798	-0.197721126722025\\
0.98989898989899	-0.196078410559214\\
1	-0.195254640239191\\
};
\end{axis}
\end{tikzpicture}%
	\end{subfigure}
	\qquad \qquad
	\begin{subfigure}[b]{0.43\textwidth}
%
%
\definecolor{mycolor1}{rgb}{0.00000,0.44700,0.74100}%
\definecolor{mycolor2}{rgb}{0.85000,0.32500,0.09800}%
\begin{tikzpicture}[scale=0.50]

\begin{axis}[%
width=4.521in,
height=3.566in,
at={(0.758in,0.481in)},
scale only axis,
xmin=0,
xmax=1,
ymin=-2,
ymax=2,
axis background/.style={fill=white},
title style={font=\bfseries},
title={\Large t = 0.60}
]
\addplot [color=mycolor1, forget plot, thick]
  table[row sep=crcr]{%
0	0.713178501898282\\
0.0101010101010101	0.712175122330279\\
0.0202020202020202	0.710170535629272\\
0.0303030303030303	0.707169079855643\\
0.0404040404040404	0.703177245101055\\
0.0505050505050505	0.698203653160051\\
0.0606060606060606	0.692259030542341\\
0.0707070707070707	0.685356174925352\\
0.0808080808080808	0.677509915170769\\
0.0909090909090909	0.668737065052474\\
0.101010101010101	0.659056370866408\\
0.111111111111111	0.648488453115332\\
0.121212121212121	0.637055742483176\\
0.131313131313131	0.624782410334519\\
0.141414141414141	0.61169429399472\\
0.151515151515152	0.59781881708518\\
0.161616161616162	0.583184905206129\\
0.171717171717172	0.567822897276092\\
0.181818181818182	0.551764452852791\\
0.191919191919192	0.535042455774532\\
0.202020202020202	0.517690914474148\\
0.212121212121212	0.499744859329232\\
0.222222222222222	0.481240237422601\\
0.232323232323232	0.462213805095762\\
0.242424242424242	0.442703018685438\\
0.252525252525253	0.422745923839061\\
0.262626262626263	0.402381043809382\\
0.272727272727273	0.38164726713112\\
0.282828282828283	0.360583735083742\\
0.292929292929293	0.339229729344123\\
0.303030303030303	0.317624560230871\\
0.313131313131313	0.295807455938681\\
0.323232323232323	0.273817453156071\\
0.333333333333333	0.251693289453318\\
0.343434343434343	0.229473297819487\\
0.353535353535354	0.207195303717945\\
0.363636363636364	0.184896525018939\\
0.373737373737374	0.162613475155605\\
0.383838383838384	0.140381869836196\\
0.393939393939394	0.118236537630582\\
0.404040404040404	0.0962113347330015\\
0.414141414141414	0.0743390641859435\\
0.424242424242424	0.0526513998317997\\
0.434343434343434	0.0311788152396984\\
0.444444444444444	0.00995051783479892\\
0.454545454545455	-0.0110056115636766\\
0.464646464646465	-0.0316630736111436\\
0.474747474747475	-0.0519967995509872\\
0.484848484848485	-0.0719831944334379\\
0.494949494949495	-0.0916001743244102\\
0.505050505050505	-0.110827197412635\\
0.515151515151515	-0.129645288947631\\
0.525252525252525	-0.148037059965492\\
0.535353535353535	-0.165986719784012\\
0.545454545454545	-0.183480082273226\\
0.555555555555556	-0.200504565931937\\
0.565656565656566	-0.217049187825096\\
0.575757575757576	-0.233104551460959\\
0.585858585858586	-0.248662828710587\\
0.595959595959596	-0.263717735895503\\
0.606060606060606	-0.278264504191949\\
0.616161616161616	-0.292299844522234\\
0.626262626262626	-0.30582190712495\\
0.636363636363636	-0.318830236016317\\
0.646464646464646	-0.331325718574541\\
0.656565656565657	-0.343310530497655\\
0.666666666666667	-0.354788076402962\\
0.676767676767677	-0.36576292635263\\
0.686868686868687	-0.376240748605339\\
0.696969696969697	-0.386228238907948\\
0.707070707070707	-0.395733046653922\\
0.717171717171717	-0.40476369824675\\
0.727272727272727	-0.413329518016639\\
0.737373737373737	-0.421440547047411\\
0.747474747474747	-0.429107460277778\\
0.757575757575758	-0.436341482246829\\
0.767676767676768	-0.443154301857841\\
0.777777777777778	-0.449557986537191\\
0.787878787878788	-0.455564896166333\\
0.797979797979798	-0.461187597164459\\
0.808080808080808	-0.466438777097565\\
0.818181818181818	-0.471331160186245\\
0.828282828282828	-0.475877424079645\\
0.838383838383838	-0.480090118256597\\
0.848484848484849	-0.483981584407127\\
0.858585858585859	-0.48756387913824\\
0.868686868686869	-0.490848699337259\\
0.878787878787879	-0.493847310513951\\
0.888888888888889	-0.496570478429415\\
0.898989898989899	-0.499028404305165\\
0.909090909090909	-0.501230663890099\\
0.919191919191919	-0.503186150646254\\
0.929292929292929	-0.504903023296323\\
0.939393939393939	-0.506388657957076\\
0.94949494949495	-0.507649605063049\\
0.95959595959596	-0.50869155126429\\
0.96969696969697	-0.509519286460627\\
0.97979797979798	-0.510136676112935\\
0.98989898989899	-0.510546638949352\\
1	-0.510751130161379\\
};
\addplot [color=mycolor2, forget plot, thick]
  table[row sep=crcr]{%
0	0.731021330205663\\
0.0101010101010101	0.729992847344008\\
0.0202020202020202	0.727938108407288\\
0.0303030303030303	0.724861559988691\\
0.0404040404040404	0.720769854553806\\
0.0505050505050505	0.715671829603622\\
0.0606060606060606	0.709578480011616\\
0.0707070707070707	0.702502923636994\\
0.0808080808080808	0.694460360340895\\
0.0909090909090909	0.685468024556655\\
0.101010101010101	0.675545131588939\\
0.111111111111111	0.664712817839529\\
0.121212121212121	0.652994075179831\\
0.131313131313131	0.640413679711535\\
0.141414141414141	0.626998115177355\\
0.151515151515152	0.612775491303176\\
0.161616161616162	0.597775457371331\\
0.171717171717172	0.582029111341895\\
0.181818181818182	0.565568904854867\\
0.191919191919192	0.54842854446079\\
0.202020202020202	0.530642889440675\\
0.212121212121212	0.512247846588063\\
0.222222222222222	0.49328026233653\\
0.232323232323232	0.473777812624974\\
0.242424242424242	0.453778890900511\\
0.252525252525253	0.433322494664776\\
0.262626262626263	0.412448110973809\\
0.272727272727273	0.391195601304508\\
0.282828282828283	0.369605086201881\\
0.292929292929293	0.347716830120907\\
0.303030303030303	0.325571126874876\\
0.313131313131313	0.303208186098533\\
0.323232323232323	0.280668021129199\\
0.333333333333333	0.257990338702393\\
0.343434343434343	0.23521443085031\\
0.353535353535354	0.212379069381792\\
0.363636363636364	0.189522403311348\\
0.373737373737374	0.166681859592243\\
0.383838383838384	0.143894047494789\\
0.393939393939394	0.12119466695583\\
0.404040404040404	0.0986184212089615\\
0.414141414141414	0.0761989339875016\\
0.424242424242424	0.0539686715735044\\
0.434343434343434	0.0319588699464354\\
0.444444444444444	0.0101994672644622\\
0.454545454545455	-0.0112809581101946\\
0.464646464646465	-0.0324552438527086\\
0.474747474747475	-0.0532976940177331\\
0.484848484848485	-0.0737841233395589\\
0.494949494949495	-0.0938918953718708\\
0.505050505050505	-0.113599954373141\\
0.515151515151515	-0.132888850868519\\
0.525252525252525	-0.151740760844112\\
0.535353535353535	-0.170139498554724\\
0.545454545454545	-0.188070522951277\\
0.555555555555556	-0.205520937759253\\
0.565656565656566	-0.222479485264393\\
0.575757575757576	-0.238936533886559\\
0.585858585858586	-0.254884059646883\\
0.595959595959596	-0.27031562165716\\
0.606060606060606	-0.285226331783667\\
0.616161616161616	-0.299612818660128\\
0.626262626262626	-0.313473186246434\\
0.636363636363636	-0.326806967150666\\
0.646464646464646	-0.339615070952115\\
0.656565656565657	-0.351899727782039\\
0.666666666666667	-0.36366442743698\\
0.676767676767677	-0.374913854316308\\
0.686868686868687	-0.385653818491401\\
0.696969696969697	-0.395891183228273\\
0.707070707070707	-0.405633789298586\\
0.717171717171717	-0.41489037642571\\
0.727272727272727	-0.423670502222854\\
0.737373737373737	-0.431984458989119\\
0.747474747474747	-0.439843188736743\\
0.757575757575758	-0.447258196828654\\
0.767676767676768	-0.454241464609774\\
0.777777777777778	-0.460805361418294\\
0.787878787878788	-0.466962556364345\\
0.797979797979798	-0.472725930263108\\
0.808080808080808	-0.478108488107506\\
0.818181818181818	-0.483123272462114\\
0.828282828282828	-0.487783278154901\\
0.838383838383838	-0.492101368636861\\
0.848484848484849	-0.496090194371565\\
0.858585858585859	-0.499762113607152\\
0.868686868686869	-0.503129115872338\\
0.878787878787879	-0.506202748525762\\
0.888888888888889	-0.508994046674314\\
0.898989898989899	-0.511513466761228\\
0.909090909090909	-0.513770824108585\\
0.919191919191919	-0.515775234681645\\
0.929292929292929	-0.517535061324073\\
0.939393939393939	-0.519057864693794\\
0.94949494949495	-0.520350359108969\\
0.95959595959596	-0.521418373492474\\
0.96969696969697	-0.5222668175814\\
0.97979797979798	-0.522899653545586\\
0.98989898989899	-0.523319873136074\\
1	-0.523529480460795\\
};
\end{axis}
\end{tikzpicture}%
	\end{subfigure}
	
	\begin{subfigure}[b]{0.43\textwidth}
%
%
\definecolor{mycolor1}{rgb}{0.00000,0.44700,0.74100}%
\definecolor{mycolor2}{rgb}{0.85000,0.32500,0.09800}%
\begin{tikzpicture}[scale=0.50]

\begin{axis}[%
width=4.521in,
height=3.566in,
at={(0.758in,0.481in)},
scale only axis,
xmin=0,
xmax=1,
ymin=-2,
ymax=2,
axis background/.style={fill=white},
title style={font=\bfseries},
title={\Large$\text{t = t}_\text{0}\text{ = 1.34}$}
]
\addplot [color=mycolor1, forget plot, thick]
  table[row sep=crcr]{%
0	0.618858474410166\\
0.0101010101010101	0.618214898698561\\
0.0202020202020202	0.616928511999963\\
0.0303030303030303	0.61500084249746\\
0.0404040404040404	0.612434179303607\\
0.0505050505050505	0.609231568675091\\
0.0606060606060606	0.605396808980792\\
0.0707070707070707	0.600934444436234\\
0.0808080808080808	0.595849757620524\\
0.0909090909090909	0.590148760795004\\
0.101010101010101	0.583838186045872\\
0.111111111111111	0.576925474275974\\
0.121212121212121	0.569418763073879\\
0.131313131313131	0.561326873491076\\
0.141414141414141	0.552659295760855\\
0.151515151515152	0.543426173994956\\
0.161616161616162	0.533638289896531\\
0.171717171717172	0.523307045530255\\
0.181818181818182	0.512444445192602\\
0.191919191919192	0.501063076427312\\
0.202020202020202	0.48917609023296\\
0.212121212121212	0.476797180511242\\
0.222222222222222	0.463940562806153\\
0.232323232323232	0.450620952385603\\
0.242424242424242	0.43685354171825\\
0.252525252525253	0.422653977399337\\
0.262626262626263	0.408038336580215\\
0.272727272727273	0.39302310295688\\
0.282828282828283	0.377625142373373\\
0.292929292929293	0.361861678096208\\
0.303030303030303	0.345750265816101\\
0.313131313131313	0.329308768433277\\
0.323232323232323	0.312555330682337\\
0.333333333333333	0.29550835365234\\
0.343434343434343	0.278186469257098\\
0.353535353535354	0.260608514709989\\
0.363636363636364	0.242793507056635\\
0.373737373737374	0.224760617817729\\
0.383838383838384	0.206529147793028\\
0.393939393939394	0.188118502076166\\
0.404040404040404	0.169548165328353\\
0.414141414141414	0.15083767735739\\
0.424242424242424	0.132006609046579\\
0.434343434343434	0.113074538676178\\
0.444444444444444	0.0940610286779955\\
0.454545454545455	0.074985602861553\\
0.464646464646465	0.0558677241479736\\
0.474747474747475	0.0367267728454072\\
0.484848484848485	0.0175820254973642\\
0.494949494949495	-0.00154736566717423\\
0.505050505050505	-0.0206423926555648\\
0.515151515151515	-0.0396842109815404\\
0.525252525252525	-0.0586541582624434\\
0.535353535353535	-0.0775337721252989\\
0.545454545454545	-0.0963048074060774\\
0.555555555555556	-0.114949252629201\\
0.565656565656566	-0.133449345757051\\
0.575757575757576	-0.151787589201936\\
0.585858585858586	-0.169946764095629\\
0.595959595959596	-0.18790994381423\\
0.606060606060606	-0.205660506758698\\
0.616161616161616	-0.223182148393913\\
0.626262626262626	-0.240458892551627\\
0.636363636363636	-0.257475102005032\\
0.646464646464646	-0.274215488325007\\
0.656565656565657	-0.290665121030316\\
0.666666666666667	-0.306809436046162\\
0.676767676767677	-0.32263424348753\\
0.686868686868687	-0.338125734785615\\
0.696969696969697	-0.353270489177483\\
0.707070707070707	-0.368055479580689\\
0.717171717171717	-0.382468077876175\\
0.727272727272727	-0.3964960596241\\
0.737373737373737	-0.410127608238532\\
0.747474747474747	-0.423351318648013\\
0.757575757575758	-0.436156200469962\\
0.767676767676768	-0.448531680727675\\
0.777777777777778	-0.460467606139321\\
0.787878787878788	-0.471954245008829\\
0.797979797979798	-0.482982288748867\\
0.808080808080808	-0.49354285306632\\
0.818181818181818	-0.503627478840654\\
0.828282828282828	-0.513228132725454\\
0.838383838383838	-0.522337207503101\\
0.848484848484849	-0.530947522222132\\
0.858585858585859	-0.539052322146232\\
0.868686868686869	-0.546645278543076\\
0.878787878787879	-0.55372048834038\\
0.888888888888889	-0.56027247367549\\
0.898989898989899	-0.566296181363742\\
0.909090909090909	-0.571786982309531\\
0.919191919191919	-0.576740670882691\\
0.929292929292929	-0.581153464281285\\
0.939393939393939	-0.585022001900331\\
0.94949494949495	-0.5883433447243\\
0.95959595959596	-0.591114974759499\\
0.96969696969697	-0.593334794520552\\
0.97979797979798	-0.595001126583347\\
0.98989898989899	-0.596112713214805\\
1	-0.596668716087848\\
};
\addplot [color=mycolor2, forget plot, thick]
  table[row sep=crcr]{%
0	0.749522682931197\\
0.0101010101010101	0.748743224276306\\
0.0202020202020202	0.747185233153151\\
0.0303030303030303	0.744850560401521\\
0.0404040404040404	0.74174197845139\\
0.0505050505050505	0.737863176738354\\
0.0606060606060606	0.733218755609262\\
0.0707070707070707	0.72781421873378\\
0.0808080808080808	0.721655964041367\\
0.0909090909090909	0.714751273206975\\
0.101010101010101	0.707108299712404\\
0.111111111111111	0.698736055513854\\
0.121212121212121	0.689644396349703\\
0.131313131313131	0.679844005725876\\
0.141414141414141	0.669346377619447\\
0.151515151515152	0.658163797944178\\
0.161616161616162	0.646309324824676\\
0.171717171717172	0.63379676772863\\
0.181818181818182	0.620640665509222\\
0.191919191919192	0.606856263412248\\
0.202020202020202	0.592459489104764\\
0.212121212121212	0.577466927784136\\
0.222222222222222	0.561895796428257\\
0.232323232323232	0.54576391724938\\
0.242424242424242	0.529089690415458\\
0.252525252525253	0.511892066104166\\
0.262626262626263	0.49419051595581\\
0.272727272727273	0.476005003992145\\
0.282828282828283	0.457355957068746\\
0.292929292929293	0.438264234928926\\
0.303030303030303	0.418751099927414\\
0.313131313131313	0.398838186491872\\
0.323232323232323	0.378547470390145\\
0.333333333333333	0.357901237870557\\
0.343434343434343	0.336922054741942\\
0.353535353535354	0.31563273545913\\
0.363636363636364	0.294056312278518\\
0.373737373737374	0.272216004547049\\
0.383838383838384	0.250135188186382\\
0.393939393939394	0.227837365432399\\
0.404040404040404	0.205346134888254\\
0.414141414141414	0.182685161947205\\
0.424242424242424	0.15987814963921\\
0.434343434343434	0.13694880995296\\
0.444444444444444	0.113920835682495\\
0.454545454545455	0.0908178728449548\\
0.464646464646465	0.0676634937132602\\
0.474747474747475	0.0444811705046646\\
0.484848484848485	0.0212942497631794\\
0.494949494949495	-0.00187407252916792\\
0.505050505050505	-0.0250007750803819\\
0.515151515151515	-0.0480630346271652\\
0.525252525252525	-0.0710382484587237\\
0.535353535353535	-0.0939040561034835\\
0.545454545454545	-0.116638360159762\\
0.555555555555556	-0.139219346254713\\
0.565656565656566	-0.161625502119153\\
0.575757575757576	-0.183835635769105\\
0.585858585858586	-0.205828892788184\\
0.595959595959596	-0.227584772708053\\
0.606060606060606	-0.249083144487412\\
0.616161616161616	-0.270304261092952\\
0.626262626262626	-0.291228773188787\\
0.636363636363636	-0.311837741943706\\
0.646464646464646	-0.332112650968441\\
0.656565656565657	-0.352035417397819\\
0.666666666666667	-0.371588402135237\\
0.676767676767677	-0.390754419279347\\
0.686868686868687	-0.409516744755157\\
0.696969696969697	-0.427859124173885\\
0.707070707070707	-0.445765779947945\\
0.717171717171717	-0.463221417689252\\
0.727272727272727	-0.480211231920746\\
0.737373737373737	-0.496720911132514\\
0.747474747474747	-0.512736642215241\\
0.757575757575758	-0.528245114304844\\
0.767676767676768	-0.543233522073136\\
0.777777777777778	-0.557689568500116\\
0.787878787878788	-0.571601467164089\\
0.797979797979798	-0.584957944086207\\
0.808080808080808	-0.597748239166239\\
0.818181818181818	-0.6099621072464\\
0.828282828282828	-0.621589818839891\\
0.838383838383838	-0.632622160560461\\
0.848484848484849	-0.643050435288768\\
0.858585858585859	-0.652866462110605\\
0.868686868686869	-0.662062576061159\\
0.878787878787879	-0.670631627708437\\
0.888888888888889	-0.678566982607759\\
0.898989898989899	-0.685862520657869\\
0.909090909090909	-0.692512635387656\\
0.919191919191919	-0.69851223320087\\
0.929292929292929	-0.70385673260437\\
0.939393939393939	-0.708542063443557\\
0.94949494949495	-0.712564666166625\\
0.95959595959596	-0.715921491137075\\
0.96969696969697	-0.718609998011792\\
0.97979797979798	-0.720628155199589\\
0.98989898989899	-0.721974439412818\\
1	-0.722647835322145\\
};
\end{axis}
\end{tikzpicture}%
	\end{subfigure}
	\qquad \qquad
	\begin{subfigure}[b]{0.43\textwidth}
%
%
\definecolor{mycolor1}{rgb}{0.00000,0.44700,0.74100}%
\definecolor{mycolor2}{rgb}{0.85000,0.32500,0.09800}%
\begin{tikzpicture}[scale=0.50]

\begin{axis}[%
width=4.521in,
height=3.566in,
at={(0.758in,0.481in)},
scale only axis,
xmin=0,
xmax=1,
ymin=-2,
ymax=2,
axis background/.style={fill=white},
title style={font=\bfseries},
title={\Large t = 3.00}
]
\addplot [color=mycolor1, forget plot, thick]
  table[row sep=crcr]{%
0	0.474599586407972\\
0.0101010101010101	0.474131039845685\\
0.0202020202020202	0.473194409806976\\
0.0303030303030303	0.471790622003846\\
0.0404040404040404	0.469921063855461\\
0.0505050505050505	0.467587583110844\\
0.0606060606060606	0.464792486014633\\
0.0707070707070707	0.461538535017751\\
0.0808080808080808	0.457828946035309\\
0.0909090909090909	0.453667385254496\\
0.101010101010101	0.449057965495702\\
0.111111111111111	0.444005242130516\\
0.121212121212121	0.438514208560752\\
0.131313131313131	0.432590291263039\\
0.141414141414141	0.426239344403983\\
0.151515151515152	0.419467644031342\\
0.161616161616162	0.41228188184705\\
0.171717171717172	0.404689158568389\\
0.181818181818182	0.396696976883974\\
0.191919191919192	0.388313234011675\\
0.202020202020202	0.379546213865947\\
0.212121212121212	0.370404578842456\\
0.222222222222222	0.36089736122828\\
0.232323232323232	0.351033954246309\\
0.242424242424242	0.340824102742856\\
0.252525252525253	0.33027789352783\\
0.262626262626263	0.319405745377182\\
0.272727272727273	0.308218398707653\\
0.282828282828283	0.296726904934186\\
0.292929292929293	0.284942615520675\\
0.303030303030303	0.27287717073503\\
0.313131313131313	0.260542488119814\\
0.323232323232323	0.247950750690003\\
0.333333333333333	0.235114394869678\\
0.343434343434343	0.222046098179702\\
0.353535353535354	0.208758766688706\\
0.363636363636364	0.195265522239896\\
0.373737373737374	0.181579689466452\\
0.383838383838384	0.167714782608446\\
0.393939393939394	0.153684492144454\\
0.404040404040404	0.13950267125114\\
0.414141414141414	0.125183322104325\\
0.424242424242424	0.110740582035148\\
0.434343434343434	0.0961887095550958\\
0.444444444444444	0.0815420702637663\\
0.454545454545455	0.0668151226533761\\
0.464646464646465	0.0520224038240776\\
0.474747474747475	0.0371785151242539\\
0.484848484848485	0.0222981077300132\\
0.494949494949495	0.00739586817815519\\
0.505050505050505	-0.00751349613307825\\
0.515151515151515	-0.0224152714611414\\
0.525252525252525	-0.0372947522398717\\
0.535353535353535	-0.0521372555893973\\
0.545454545454545	-0.0669281358009549\\
0.555555555555556	-0.0816527987823457\\
0.565656565656566	-0.0962967164498041\\
0.575757575757576	-0.11084544105212\\
0.585858585858586	-0.125284619412929\\
0.595959595959596	-0.139600007077183\\
0.606060606060606	-0.153777482347903\\
0.616161616161616	-0.167803060199455\\
0.626262626262626	-0.181662906053702\\
0.636363636363636	-0.195343349405542\\
0.646464646464646	-0.208830897284483\\
0.656565656565657	-0.222112247539107\\
0.666666666666667	-0.235174301931414\\
0.676767676767677	-0.248004179028271\\
0.686868686868687	-0.260589226877375\\
0.696969696969697	-0.272917035455366\\
0.707070707070707	-0.284975448875952\\
0.717171717171717	-0.296752577346153\\
0.727272727272727	-0.308236808859034\\
0.737373737373737	-0.319416820611524\\
0.747474747474747	-0.330281590136248\\
0.757575757575758	-0.340820406136529\\
0.767676767676768	-0.351022879014058\\
0.777777777777778	-0.36087895107899\\
0.787878787878788	-0.37037890643258\\
0.797979797979798	-0.379513380512763\\
0.808080808080808	-0.388273369293432\\
0.818181818181818	-0.396650238128505\\
0.828282828282828	-0.404635730232215\\
0.838383838383838	-0.41222197478741\\
0.848484848484849	-0.419401494674033\\
0.858585858585859	-0.426167213810303\\
0.868686868686869	-0.432512464099492\\
0.878787878787879	-0.438430991975602\\
0.888888888888889	-0.443916964541609\\
0.898989898989899	-0.448964975294356\\
0.909090909090909	-0.453570049430558\\
0.919191919191919	-0.457727648728811\\
0.929292929292929	-0.461433676002888\\
0.939393939393939	-0.464684479122035\\
0.94949494949495	-0.467476854594377\\
0.95959595959596	-0.469808050709995\\
0.96969696969697	-0.471675770240641\\
0.97979797979798	-0.473078172693474\\
0.98989898989899	-0.474013876116674\\
1	-0.474481958455166\\
};
\addplot [color=mycolor2, forget plot, thick]
  table[row sep=crcr]{%
0	1.03785605891633\\
0.0101010101010101	1.03683143963203\\
0.0202020202020202	1.03478321374124\\
0.0303030303030303	1.03171340559426\\
0.0404040404040404	1.02762504920422\\
0.0505050505050505	1.02252218523526\\
0.0606060606060606	1.01640985699133\\
0.0707070707070707	1.00929410540989\\
0.0808080808080808	1.00118196306553\\
0.0909090909090909	0.992081447189391\\
0.101010101010101	0.982001551711687\\
0.111111111111111	0.970952238335172\\
0.121212121212121	0.958944426648642\\
0.131313131313131	0.945989983290434\\
0.141414141414141	0.93210171017285\\
0.151515151515152	0.917293331779384\\
0.161616161616162	0.901579481547565\\
0.171717171717172	0.884975687351124\\
0.181818181818182	0.867498356096141\\
0.191919191919192	0.849164757446665\\
0.202020202020202	0.829993006696204\\
0.212121212121212	0.810002046802328\\
0.222222222222222	0.789211629602457\\
0.232323232323232	0.767642296229721\\
0.242424242424242	0.745315356748602\\
0.252525252525253	0.722252869030787\\
0.262626262626263	0.69847761689248\\
0.272727272727273	0.674013087515106\\
0.282828282828283	0.648883448172059\\
0.292929292929293	0.623113522284846\\
0.303030303030303	0.596728764832618\\
0.313131313131313	0.569755237139723\\
0.323232323232323	0.542219581066524\\
0.333333333333333	0.51414899262931\\
0.343434343434343	0.485571195075672\\
0.353535353535354	0.456514411442267\\
0.363636363636364	0.427007336622362\\
0.373737373737374	0.397079108971053\\
0.383838383838384	0.366759281476468\\
0.393939393939394	0.336077792525712\\
0.404040404040404	0.305064936294649\\
0.414141414141414	0.273751332791031\\
0.424242424242424	0.242167897580738\\
0.434343434343434	0.210345811227259\\
0.444444444444444	0.17831648847475\\
0.454545454545455	0.146111547205273\\
0.464646464646465	0.113762777201004\\
0.474747474747475	0.0813021087423757\\
0.484848484848485	0.0487615810732676\\
0.494949494949495	0.0161733107644457\\
0.505050505050505	-0.0164305399934494\\
0.515151515151515	-0.0490177951360301\\
0.525252525252525	-0.0815562964790533\\
0.535353535353535	-0.114013935448736\\
0.545454545454545	-0.146358684757209\\
0.555555555555556	-0.178558629991886\\
0.565656565656566	-0.21058200108766\\
0.575757575757576	-0.242397203650939\\
0.585858585858586	-0.273972850104742\\
0.595959595959596	-0.305277790624243\\
0.606060606060606	-0.336281143832403\\
0.616161616161616	-0.36695232722556\\
0.626262626262626	-0.397261087299177\\
0.636363636363636	-0.427177529344204\\
0.646464646464646	-0.456672146884916\\
0.656565656565657	-0.485715850729407\\
0.666666666666667	-0.514279997604362\\
0.676767676767677	-0.542336418346117\\
0.686868686868687	-0.5698574456205\\
0.696969696969697	-0.596815941144408\\
0.707070707070707	-0.623185322382584\\
0.717171717171717	-0.64893958869357\\
0.727272727272727	-0.674053346899409\\
0.737373737373737	-0.698501836254187\\
0.747474747474747	-0.722260952787151\\
0.757575757575758	-0.745307272996739\\
0.767676767676768	-0.767618076872515\\
0.777777777777778	-0.789171370222655\\
0.787878787878788	-0.809945906285319\\
0.797979797979798	-0.829921206602968\\
0.808080808080808	-0.849077581139378\\
0.818181818181818	-0.86739614761987\\
0.828282828282828	-0.884858850076039\\
0.838383838383838	-0.901448476577023\\
0.848484848484849	-0.917148676130162\\
0.858585858585859	-0.931943974734716\\
0.868686868686869	-0.94581979057311\\
0.878787878787879	-0.958762448325039\\
0.888888888888889	-0.970759192590605\\
0.898989898989899	-0.981798200409524\\
0.909090909090909	-0.991868592864328\\
0.919191919191919	-1.00096044575635\\
0.929292929292929	-1.00906479934423\\
0.939393939393939	-1.01617366713547\\
0.94949494949495	-1.02228004372267\\
0.95959595959596	-1.02737791165684\\
0.96969696969697	-1.03146224735108\\
0.97979797979798	-1.03452902600912\\
0.98989898989899	-1.03657522557382\\
1	-1.03759882969188\\
};
\end{axis}
\end{tikzpicture}%
	\end{subfigure}
	
	\begin{subfigure}[b]{0.43\textwidth}
%
%
\definecolor{mycolor1}{rgb}{0.00000,0.44700,0.74100}%
\definecolor{mycolor2}{rgb}{0.85000,0.32500,0.09800}%
\begin{tikzpicture}[scale=0.50]

\begin{axis}[%
width=4.521in,
height=3.566in,
at={(0.758in,0.481in)},
scale only axis,
xmin=0,
xmax=1,
ymin=-2,
ymax=2,
axis background/.style={fill=white},
title style={font=\bfseries},
title={\Large$\text{t = t}_\text{-}\text{ = 5.31}$}
]
\addplot [color=mycolor1, forget plot, thick]
  table[row sep=crcr]{%
0	0.236700664858286\\
0.0101010101010101	0.236467069794065\\
0.0202020202020202	0.236000110196086\\
0.0303030303030303	0.235300246897768\\
0.0404040404040404	0.234368170580694\\
0.0505050505050505	0.233204801092989\\
0.0606060606060606	0.231811286541538\\
0.0707070707070707	0.230189002158934\\
0.0808080808080808	0.228339548946282\\
0.0909090909090909	0.226264752093196\\
0.101010101010101	0.223966659176548\\
0.111111111111111	0.221447538139742\\
0.121212121212121	0.218709875054519\\
0.131313131313131	0.215756371667487\\
0.141414141414141	0.212589942733815\\
0.151515151515152	0.209213713140698\\
0.161616161616162	0.205631014823463\\
0.171717171717172	0.201845383477325\\
0.181818181818182	0.197860555068072\\
0.191919191919192	0.193680462145098\\
0.202020202020202	0.189309229960427\\
0.212121212121212	0.184751172397575\\
0.222222222222222	0.180010787714245\\
0.232323232323232	0.175092754103065\\
0.242424242424242	0.170001925074764\\
0.252525252525253	0.164743324668314\\
0.262626262626263	0.159322142492794\\
0.272727272727273	0.153743728605851\\
0.282828282828283	0.148013588233814\\
0.292929292929293	0.142137376338684\\
0.303030303030303	0.136120892037346\\
0.313131313131313	0.12997007287852\\
0.323232323232323	0.123690988983099\\
0.333333333333333	0.117289837053651\\
0.343434343434343	0.110772934259004\\
0.353535353535354	0.104146711999944\\
0.363636363636364	0.0974177095621799\\
0.373737373737374	0.0905925676628425\\
0.383838383838384	0.0836780218968819\\
0.393939393939394	0.0766808960898325\\
0.404040404040404	0.0696080955635068\\
0.414141414141414	0.062466600321263\\
0.424242424242424	0.0552634581595723\\
0.434343434343434	0.0480057777126825\\
0.444444444444444	0.0407007214372452\\
0.454545454545455	0.0333554985438258\\
0.464646464646465	0.0259773578822757\\
0.474747474747475	0.0185735807879863\\
0.484848484848485	0.0111514738960845\\
0.494949494949495	0.00371836193066197\\
0.505050505050505	-0.00371841952384543\\
0.515151515151515	-0.0111515312619957\\
0.525252525252525	-0.01857363770025\\
0.535353535353535	-0.0259774141163068\\
0.545454545454545	-0.0333555538777161\\
0.555555555555556	-0.0407007756526388\\
0.565656565656566	-0.0480058305956379\\
0.575757575757576	-0.0552635095014062\\
0.585858585858586	-0.0624666499193745\\
0.595959595959596	-0.0696081432221764\\
0.606060606060606	-0.0766809416209951\\
0.616161616161616	-0.0836780651208685\\
0.626262626262626	-0.0905926084090895\\
0.636363636363636	-0.0974177476699022\\
0.646464646464646	-0.10414674731877\\
0.656565656565657	-0.110772966649568\\
0.666666666666667	-0.117289866388143\\
0.676767676767677	-0.123691015145772\\
0.686868686868687	-0.129970095766143\\
0.696969696969697	-0.136120911559614\\
0.707070707070707	-0.142137392418573\\
0.717171717171717	-0.148013600807886\\
0.727272727272727	-0.153743737624504\\
0.737373737373737	-0.159322147920458\\
0.747474747474747	-0.164743326483589\\
0.757575757575758	-0.170001923270508\\
0.767676767676768	-0.17509274868642\\
0.777777777777778	-0.18001077870661\\
0.787878787878788	-0.184751159834521\\
0.797979797979798	-0.189309213891556\\
0.808080808080808	-0.193680442633849\\
0.818181818181818	-0.197860532191469\\
0.828282828282828	-0.201845357325672\\
0.838383838383838	-0.205630985499991\\
0.848484848484849	-0.209213680761156\\
0.858585858585859	-0.212589907426011\\
0.868686868686869	-0.215756333570787\\
0.878787878787879	-0.218709834319295\\
0.888888888888889	-0.221447494926779\\
0.898989898989899	-0.22396661365641\\
0.909090909090909	-0.226264704445552\\
0.919191919191919	-0.228339499359197\\
0.929292929292929	-0.230188950828127\\
0.939393939393939	-0.23181123366961\\
0.94949494949495	-0.233204746888624\\
0.95959595959596	-0.234368115257833\\
0.96969696969697	-0.235300190674767\\
0.97979797979798	-0.236000053294854\\
0.98989898989899	-0.236467012439185\\
1	-0.236700607276133\\
};
\addplot [color=mycolor2, forget plot, thick]
  table[row sep=crcr]{%
0	1.67116516136456\\
0.0101010101010101	1.66951592251081\\
0.0202020202020202	1.66621907240533\\
0.0303030303030303	1.66127786464587\\
0.0404040404040404	1.65469717561502\\
0.0505050505050505	1.64648349966776\\
0.0606060606060606	1.63664494272229\\
0.0707070707070707	1.62519121426041\\
0.0808080808080808	1.61213361774539\\
0.0909090909090909	1.59748503946674\\
0.101010101010101	1.58125993582295\\
0.111111111111111	1.56347431905465\\
0.121212121212121	1.54414574144242\\
0.131313131313131	1.52329327798469\\
0.141414141414141	1.50093750757291\\
0.151515151515152	1.47710049268255\\
0.161616161616162	1.45180575760005\\
0.171717171717172	1.42507826520698\\
0.181818181818182	1.39694439234463\\
0.191919191919192	1.36743190378308\\
0.202020202020202	1.33656992482062\\
0.212121212121212	1.30438891254045\\
0.222222222222222	1.27092062575304\\
0.232323232323232	1.23619809365393\\
0.242424242424242	1.20025558322773\\
0.252525252525253	1.16312856543056\\
0.262626262626263	1.12485368018441\\
0.272727272727273	1.08546870021778\\
0.282828282828283	1.04501249378842\\
0.292929292929293	1.00352498632487\\
0.303030303030303	0.96104712102478\\
0.313131313131313	0.917620818448757\\
0.323232323232323	0.873288935149682\\
0.333333333333333	0.828095221378336\\
0.343434343434343	0.782084277907046\\
0.353535353535354	0.735301512013981\\
0.363636363636364	0.68779309267152\\
0.373737373737374	0.63960590498293\\
0.383838383838384	0.590787503912318\\
0.393939393939394	0.541386067353511\\
0.404040404040404	0.491450348584191\\
0.414141414141414	0.441029628152197\\
0.424242424242424	0.390173665241494\\
0.434343434343434	0.338932648565776\\
0.444444444444444	0.287357146838203\\
0.454545454545455	0.235498058866113\\
0.464646464646465	0.183406563319999\\
0.474747474747475	0.131134068226286\\
0.484848484848485	0.0787321602337828\\
0.494949494949495	0.0262525537038597\\
0.505050505050505	-0.0262529603254041\\
0.515151515151515	-0.0787325652507291\\
0.525252525252525	-0.131134470040368\\
0.535353535353535	-0.183406960345592\\
0.545454545454545	-0.23549844953649\\
0.555555555555556	-0.287357529611716\\
0.565656565656566	-0.338933021931947\\
0.575757575757576	-0.390174027726966\\
0.585858585858586	-0.441029978326559\\
0.595959595959596	-0.491450685065614\\
0.606060606060606	-0.541386388814209\\
0.616161616161616	-0.590787809083783\\
0.626262626262626	-0.639606192660941\\
0.636363636363636	-0.687793361720895\\
0.646464646464646	-0.735301761373057\\
0.656565656565657	-0.782084506591868\\
0.666666666666667	-0.828095428486539\\
0.676767676767677	-0.873289119864057\\
0.686868686868687	-0.917620980040473\\
0.696969696969697	-0.961047258856259\\
0.707070707070707	-1.0035250998523\\
0.717171717171717	-1.04501258256392\\
0.727272727272727	-1.08546876389115\\
0.737373737373737	-1.1248537185045\\
0.747474747474747	-1.16312857824629\\
0.757575757575758	-1.20025557048867\\
0.767676767676768	-1.23619805541053\\
0.777777777777778	-1.27092056215636\\
0.787878787878788	-1.30438882384163\\
0.797979797979798	-1.33656981136987\\
0.808080808080808	-1.36743176602829\\
0.818181818181818	-1.3969442308296\\
0.828282828282828	-1.4250780805693\\
0.838383838383838	-1.45180555056854\\
0.848484848484849	-1.47710026407443\\
0.858585858585859	-1.50093725829053\\
0.868686868686869	-1.52329300901202\\
0.878787878787879	-1.54414545384113\\
0.888888888888889	-1.5634740139599\\
0.898989898989899	-1.58125961443897\\
0.909090909090909	-1.59748470306204\\
0.919191919191919	-1.61213326764776\\
0.929292929292929	-1.62519085185168\\
0.939393939393939	-1.63664456943286\\
0.94949494949495	-1.646483116971\\
0.95959595959596	-1.6546967850214\\
0.96969696969697	-1.66127746769704\\
0.97979797979798	-1.66621867066801\\
0.98989898989899	-1.66951551757063\\
1	-1.67116475481978\\
};
\end{axis}
\end{tikzpicture}%
	\end{subfigure}
	\qquad \qquad
	\begin{subfigure}[b]{0.43 \textwidth}
%
%
\definecolor{mycolor1}{rgb}{0.00000,0.44700,0.74100}%
\definecolor{mycolor2}{rgb}{0.85000,0.32500,0.09800}%
\begin{tikzpicture}[scale=0.50]

\begin{axis}[%
width=4.521in,
height=3.566in,
at={(0.758in,0.481in)},
scale only axis,
xmin=0,
xmax=1,
ymin=-2,
ymax=2,
axis background/.style={fill=white},
title style={font=\bfseries},
title={\Large$\text{t = 6.00}$}
]
\addplot [color=mycolor1, forget plot, thick]
  table[row sep=crcr]{%
0	0.174258155861767\\
0.0101010101010101	0.174086184092444\\
0.0202020202020202	0.173742410269204\\
0.0303030303030303	0.173227173655368\\
0.0404040404040404	0.172540982727362\\
0.0505050505050505	0.171684514672912\\
0.0606060606060606	0.170658614722738\\
0.0707070707070707	0.169464295316416\\
0.0808080808080808	0.16810273510322\\
0.0909090909090909	0.166575277778937\\
0.101010101010101	0.164883430759801\\
0.111111111111111	0.163028863694855\\
0.121212121212121	0.161013406818203\\
0.131313131313131	0.158839049142794\\
0.141414141414141	0.156507936497505\\
0.151515151515152	0.154022369409461\\
0.161616161616162	0.151384800833696\\
0.171717171717172	0.148597833732375\\
0.181818181818182	0.145664218505991\\
0.191919191919192	0.142586850279043\\
0.202020202020202	0.139368766042899\\
0.212121212121212	0.136013141658649\\
0.222222222222222	0.132523288722913\\
0.232323232323232	0.128902651299695\\
0.242424242424242	0.125154802521499\\
0.252525252525253	0.121283441063084\\
0.262626262626263	0.117292387491311\\
0.272727272727273	0.113185580494709\\
0.282828282828283	0.108967072996455\\
0.292929292929293	0.104641028154634\\
0.303030303030303	0.100211715253697\\
0.313131313131313	0.095683505491193\\
0.323232323232323	0.0910608676639241\\
0.333333333333333	0.0863483637577783\\
0.343434343434343	0.0815506444456002\\
0.353535353535354	0.0766724444975356\\
0.363636363636364	0.0717185781083842\\
0.373737373737374	0.0666939341465683\\
0.383838383838384	0.0616034713294087\\
0.393939393939394	0.0564522133294679\\
0.404040404040404	0.0512452438167905\\
0.414141414141414	0.0459877014419331\\
0.424242424242424	0.0406847747647359\\
0.434343434343434	0.0353416971338393\\
0.444444444444444	0.0299637415219993\\
0.454545454545455	0.0245562153222997\\
0.464646464646465	0.0191244551103951\\
0.474747474747475	0.0136738213779542\\
0.484848484848485	0.00820969324250114\\
0.494949494949495	0.00273746313887594\\
0.505050505050505	-0.00273746850244822\\
0.515151515151515	-0.00820969858493784\\
0.525252525252525	-0.0136738266782031\\
0.535353535353535	-0.0191244603475707\\
0.545454545454545	-0.0245562204757651\\
0.555555555555556	-0.0299637465714481\\
0.565656565656566	-0.0353417020593757\\
0.575757575757576	-0.0406847795469531\\
0.585858585858586	-0.0459877060619898\\
0.595959595959596	-0.0512452482564854\\
0.606060606060606	-0.0564522175713116\\
0.616161616161616	-0.0616034753566925\\
0.626262626262626	-0.0666939379434304\\
0.636363636363636	-0.0717185816598722\\
0.646464646464646	-0.0766724477896653\\
0.656565656565657	-0.081550647465411\\
0.666666666666667	-0.0863483664933845\\
0.676767676767677	-0.0910608701045615\\
0.686868686868687	-0.0956835076272616\\
0.696969696969697	-0.100211717076798\\
0.707070707070707	-0.104641029657606\\
0.717171717171717	-0.108967074173398\\
0.727272727272727	-0.113185581341009\\
0.737373737373737	-0.117292388003661\\
0.747474747474747	-0.121283441239493\\
0.757575757575758	-0.125154802361304\\
0.767676767676768	-0.12890265080356\\
0.777777777777778	-0.132523287892828\\
0.787878787878788	-0.136013140497921\\
0.797979797979798	-0.139368764556142\\
0.808080808080808	-0.142586848472157\\
0.818181818181818	-0.145664216386138\\
0.828282828282828	-0.148597831307954\\
0.838383838383838	-0.151384798114306\\
0.848484848484849	-0.154022366405867\\
0.858585858585859	-0.156507933221593\\
0.868686868686869	-0.158839045607524\\
0.878787878787879	-0.161013403037558\\
0.888888888888889	-0.163028859683789\\
0.898989898989899	-0.164883426534177\\
0.909090909090909	-0.166575273355462\\
0.919191919191919	-0.168102730499383\\
0.929292929292929	-0.16946429055042\\
0.939393939393939	-0.170658609813423\\
0.94949494949495	-0.171684509639685\\
0.95959595959596	-0.17254097759012\\
0.96969696969697	-0.173227168434416\\
0.97979797979798	-0.173742404985178\\
0.98989898989899	-0.174086178766231\\
1	-0.174258150514419\\
};
\addplot [color=mycolor2, forget plot, thick]
  table[row sep=crcr]{%
0	1.99973097444671\\
0.0101010101010101	1.99775748131438\\
0.0202020202020202	1.99381244264955\\
0.0303030303030303	1.9878997517298\\
0.0404040404040404	1.9800252436683\\
0.0505050505050505	1.97019668965521\\
0.0606060606060606	1.95842378928851\\
0.0707070707070707	1.9447181610016\\
0.0808080808080808	1.92909333059732\\
0.0909090909090909	1.91156471789964\\
0.101010101010101	1.8921496215361\\
0.111111111111111	1.87086720186615\\
0.121212121212121	1.84773846207219\\
0.131313131313131	1.82278622743195\\
0.141414141414141	1.79603512279267\\
0.151515151515152	1.76751154826932\\
0.161616161616162	1.73724365319089\\
0.171717171717172	1.70526130832036\\
0.181818181818182	1.67159607637587\\
0.191919191919192	1.63628118088204\\
0.202020202020202	1.59935147338244\\
0.212121212121212	1.56084339904523\\
0.222222222222222	1.52079496069617\\
0.232323232323232	1.47924568131442\\
0.242424242424242	1.4362365650281\\
0.252525252525253	1.39181005664814\\
0.262626262626263	1.34600999978034\\
0.272727272727273	1.29888159355703\\
0.282828282828283	1.25047134803095\\
0.292929292929293	1.2008270382754\\
0.303030303030303	1.14999765723602\\
0.313131313131313	1.09803336738062\\
0.323232323232323	1.04498545119489\\
0.333333333333333	0.990906260572709\\
0.343434343434343	0.935849165151182\\
0.353535353535354	0.87986849964119\\
0.363636363636364	0.823019510205586\\
0.373737373737374	0.765358299937889\\
0.383838383838384	0.706941773495276\\
0.393939393939394	0.647827580940546\\
0.404040404040404	0.588074060848453\\
0.414141414141414	0.527740182732555\\
0.424242424242424	0.466885488849421\\
0.434343434343434	0.405570035437589\\
0.444444444444444	0.343854333449307\\
0.454545454545455	0.281799288833515\\
0.464646464646465	0.219466142429018\\
0.474747474747475	0.156916409527167\\
0.484848484848485	0.0942118191636868\\
0.494949494949495	0.0314142531995658\\
0.505050505050505	-0.0314143147488727\\
0.515151515151515	-0.0942118804704485\\
0.525252525252525	-0.156916470349795\\
0.535353535353535	-0.219466202527836\\
0.545454545454545	-0.281799347971701\\
0.555555555555556	-0.343854391393833\\
0.565656565656566	-0.405570091960135\\
0.575757575757576	-0.466885543727281\\
0.585858585858586	-0.527740235749513\\
0.595959595959596	-0.588074111795636\\
0.606060606060606	-0.647827629617253\\
0.616161616161616	-0.706941819709762\\
0.626262626262626	-0.765358343508129\\
0.636363636363636	-0.823019550959991\\
0.646464646464646	-0.879868537419281\\
0.656565656565657	-0.93584919980423\\
0.666666666666667	-0.990906291964315\\
0.676767676767677	-1.04498547920153\\
0.686868686868687	-1.09803339189212\\
0.696969696969697	-1.14999767815601\\
0.707070707070707	-1.20082705552169\\
0.717171717171717	-1.25047136153583\\
0.727272727272727	-1.29888160326756\\
0.737373737373737	-1.34601000565856\\
0.747474747474747	-1.39181005867122\\
0.757575757575758	-1.43623656318841\\
0.767676767676768	-1.47924567561959\\
0.777777777777778	-1.52079495116903\\
0.787878787878788	-1.56084338572375\\
0.797979797979798	-1.59935145631956\\
0.808080808080808	-1.63628116014545\\
0.818181818181818	-1.67159605204777\\
0.828282828282828	-1.70526128049713\\
0.838383838383838	-1.73724362198269\\
0.848484848484849	-1.76751151379968\\
0.858585858585859	-1.79603508519799\\
0.868686868686869	-1.82278618686097\\
0.878787878787879	-1.84773841868538\\
0.888888888888889	-1.8708671558351\\
0.898989898989899	-1.89214957304284\\
0.909090909090909	-1.91156466713591\\
0.919191919191919	-1.92909327776382\\
0.929292929292929	-1.9447181063072\\
0.939393939393939	-1.95842373294943\\
0.94949494949495	-1.97019663189416\\
0.95959595959596	-1.9800251847136\\
0.96969696969697	-1.98789969181447\\
0.97979797979798	-1.99381238201042\\
0.98989898989899	-1.99775742019112\\
1	-1.9997309130809\\
};
\end{axis}
\end{tikzpicture}%
	\end{subfigure}
	
	\caption{The solutions $y_{cave}(\cdot,t)$ (blue line) and $y_{vex}(\cdot,t)$ (brown line) of \eqref{Eq:Reac_Diff_Syst} subject to $\mathcal{U}_{cave}$ resp. $\mathcal{U}_{vex}$ at times $t=0.15, 0.6, 1.34, 3, 5.41, 6$. Being identical until $t_+=0.15$, $y_{cave}$ converges to zero, while $y_{vex}$ starts to grow at $t_0= 1.34$ and remains growing afterwards.}\label{Fig:solution}
\end{figure}

Figure~\ref{Fig:solution} shows a plot of the corresponding solution $y_{cave}(\cdot,t)$ and $y_{vex}(\cdot,t)$ at selected times $t=0.15, 0.6, 1.34, 3, 5.41, 6$.
Until $t_+=0.15$, both solutions are equal. Afterwards, $y_{cave}(\cdot,t)$ converges to zero, while $y_{vex}(\cdot,t)$ starts to grow.

\subsubsection*{Discussion and Fourier Analysis}
In order to analyse how the geometry of the scalar generalised play operator \eqref{Eq:Hyst_op} governs the behaviour  
of the (nonlinear) reaction-diffusion equation 
\eqref{Eq:Reac_Diff_Syst}, we  
expand $y$ as Fourier series in terms of the orthonormal basis $\{\phi_k\}_{k\ge0}$ of $\mathrm{L}^2((0,1))$, i.e. 
\begin{equation*}
y(x,t) = \sum_{k=0}^{\infty} y^{(k)}(t) \phi_k(x), \qquad \text{with}
\quad y^{(k)}(t) = \langle y(\cdot,t) , \phi_k(\cdot) \rangle.
\end{equation*}
Inserting 
into \eqref{Eq:Reac_Diff_Syst} 
yields
\begin{equation}\label{Eq:Derivative_components_y}
\begin{aligned}
\dot{y}^{(k)}(t) &=  \mu_{k}(t)\,y^{(k)}(t), 
\qquad \text{with}\quad \mu_{k}(t):=R(t)-D\lambda_k
\qquad\text{for a.e. } t>0,\\
y^{(k)}(0)	  &= \langle y_0 , \phi_k \rangle.
\end{aligned}
\end{equation}

In the following we demonstrate that the difference between
$y_{cave}$ and $y_{vex}$ stems from 
a change of monotonicity of Fourier coefficients
$y^{(k)}$, which is a consequence of the different decay of $R(t)$ and hence of a different sign of some $\mu_{k}(t)$. This difference appears as the solutions slide along $\mathcal{U}_{cave}$ resp. 
$\mathcal{U}_{vex}$.

The initial data set the  
monotonicity of the Fourier coefficients
$y^{(k)}$ according to \eqref{Eq:Derivative_components_y}: With 
\begin{equation*}
\mu_k(0):=\min \{\max \{\mathcal{L}(\mathrm{T}y_0) , R_0\} , \mathcal{U}(\mathrm{T}y_0) \} - D\lambda_k,
\end{equation*}
we assume initial data $R_0$, $y_0$ and weights $k_i$, $i=m,\ldots, M$  with $m$ and $M$ from the definition of the functional $\mathrm{T}$ in \eqref{Eq:W} such that  (as in the example in Figures~\ref{Fig:phase_diag} and \ref{Fig:solution}) 
\begin{equation}\label{Tyin}
\frac{d}{dt}(\mathrm{T}y)(0) = \sum_{i=m}^{M} k_i \mu_i(0) y^{(i)}_0 < 0
\qquad 
\text{and}
\qquad 
m< I_0\le M, 
\end{equation}
where
\begin{align*}
\mu_k(0) &>0, \qquad \text{ for }\quad 0\leq k < I_0,    \\
\mu_k(0)&<0,\qquad \text{ for }\quad  I_0\leq k \le M. 
\end{align*}
Note that the $\mu_k$ are monotone decreasing in $k$ 
in the same way as the eigenvalues $\lambda_k$ are monotone increasing.  
Hence, we introduce the \emph{monotonicity index} 
\begin{equation}\label{MI}
\begin{split}
I(0) &:=I_0,\\
I(t) &:= 
\begin{cases}
\min\{k\in\mathbb{N} : \ 
\mu_k(t)\le 0\},&\qquad \text{if}\quad \mu_0(t)\ge 0,\\
-\infty,	&\qquad \text{if}\quad \mu_0(t)<0,\\
\end{cases}
\end{split}
\end{equation}
which points to the lowest non-increasing Fourier mode $y^{(k)}$ for $k=0,\ldots, \infty$.


\begin{remark}
Note that $m\ge 1$  in \eqref{Eq:W} (such as $m=1$ in Figures~\ref{Fig:phase_diag} and \ref{Fig:solution}) implies that $\mathrm{T}$ focuses on the Fourier modes 
orthogonal to the lowest Fourier mode $y^{(0)}(t) = \langle y(\cdot,t) , \phi_0 \rangle$, where $ \phi_0$ is a positive constant. Hence, while the zero order Fourier mode  $y^{(0)}(t)  \sim \int_{\Omega} y(x,t)dx$ represents the 
total population,
the higher order Fourier modes $y^{(k)}(t)$ for $k\ge m$
determine whether the solution $y$ converges to a space-homogenous large-time behaviour. 
The example of Figures~\ref{Fig:phase_diag} and \ref{Fig:solution} shows that despite being spatially homogeneous, the hysteresis operator \eqref{Eq:Hyst_op}
may not only prevent spatial homogenisation but yield 
grow-up of higher Fourier modes, i.e. Fourier modes  become unbounded in infinite time. 
We emphasise that \eqref{Eq:Reac_Diff_Syst} is a scalar PDE and that the observed mechanism of  
conditional spatial homogenisation versus inhomogeneous grow-up 
is quite different to e.g. Turing instability. 
\end{remark}

In Figures~\ref{Fig:phase_diag} and \ref{Fig:solution},
since $y_{0}^{(m)},y_{0}^{(M)}>0$ and \eqref{Eq:Derivative_components_y}, \eqref{Tyin} are satisfied, the Fourier coefficient $y^{(m)}(t)$ is strictly increasing while 
$y^{(M)}(t)$ strictly decreasing on some sufficiently 
small time interval. 
It also holds that 
\begin{equation}\label{Eq:turn_or_saddle_s}
\frac{d}{dt}(\mathrm{T}y)(t) = \sum_{i=m}^{M} k_i  \mu_k(t) y^{(i)}(t)<0, \qquad \text{for $t>0$ sufficiently small}.
\end{equation}
The evolution \eqref{Eq:turn_or_saddle_s} is 
determined by the values/signs of $\mu_k(t)=R(t)-D\lambda_k$ and, hence, by the monotonicity index $I(t)$ defined in \eqref{MI}.

In return, the evolution of the $\mu_k(t)$ is governed  
by the evolution of $R(t)$ given by \eqref{Eq:Hyst_op}, i.e. $R(t)$ is constant if $R \in (\mathcal{L}(\mathrm{T}y(t)), \mathcal{U}(\mathrm{T}y(t)))$, strictly decreasing if $R(t) = \mathcal{L}(\mathrm{T}y(t))$ and $\frac{d}{dt}(\mathrm{T}y)(t)<0$, and strictly increasing if $R(t) = \mathcal{U}(\mathrm{T}y(t))$ and $\frac{d}{dt}(\mathrm{T}y)(t)>0$.

%
%
%

\begin{remark}\label{Rem:regularity}
Before analysing the dichotomy of Figures~\ref{Fig:phase_diag} and \ref{Fig:solution}, we note first that $\mathrm{T}y\in \mathrm{C}([0,T])\cap\mathrm{C}^1(0,T)$ for any $T>0$. Moreover, $R\in \mathrm{W}^{1,\infty}(0,T)$ and $\mathrm{T}y\in \mathrm{W}^{2,\infty}(0,T)$. 
\end{remark}

The following proposition provides a largely 
explicit analysis of the nonlinear behaviour depicted in Figures~\ref{Fig:phase_diag} and \ref{Fig:solution}. 

\begin{proposition}[Spatial homogenisation versus grow-up]\label{dicho}\hfill\\
\noindent\underline{Case I, Spatial homogenisation:} Assume that the monotonicity index $I(t)$ \eqref{MI} drops
to the value $m$ at some positive time $t_m>0$.

Then, $\frac{d}{dt}(\mathrm{T}y)$ is negative for all $t>t_m$ (and hence for all $t>0$). As a consequence, $R(t)$ remains non-increasing for all $t>0$ and $y(x,t)$ converges to zero (at least) exponentially fast. 
	
\noindent\underline{Case II, Grow-up:} If Case I does not occur
and the monotonicity index $I(t)$ remains larger than $m$, then the ordering of the $\mu_k(t)$ implies that  
$\mathrm{T}y$ stops decreasing at some time $t_0>0$, i.e. $\frac{d}{dt}(\mathrm{T}y)(t_0) = 0$.

	Consequentially, for $t>t_0$, $\mathrm{T}y(t)$ increases and $R(t)$ is constant at first and increasing later as soon as $R(t) = \mathcal{L}(\mathrm{T}y(t))$. The growth of $\mathrm{T}y$ does not stop and so continues the consequential grow of $R(t)= \mathcal{L}(\mathrm{T}y(t))$.
Hence, all $\mu_k(t)$, $k\geq m$ become positive after finite time and the corresponding Fourier modes $y^{(k)}(t)$ grow exponentially fast.
The leading order contribution, however, is always given by $y^{(m)}(t) \phi_m(x)$.
%
\end{proposition}
\begin{proof}
Case I: If the monotonicity index $I({t_m})=m$ for some $t_m>0$, then $\mu_m(t_m) \le 0$ and $\mu_k(t_m) < 0$ for $m< k <\infty$. Hence, \eqref{Eq:turn_or_saddle_s} yields $\frac{d}{dt}(\mathrm{T}y)(t_m)<0$ and $\mathrm{T}y\in\mathrm{C}^1(0,T)$ (recall Remark~\ref{Rem:regularity}) implies $\frac{d}{dt}(\mathrm{T}y)(t)<0$
on a sufficiently small interval $t\in [t_m,t_m+\varepsilon)$.
Therefore, $\dot{R}(t_m)\le 0$ a.e. on $[t_m,t_m+\varepsilon)$ since $R$ can only increase if $R = \mathcal{L}(\mathrm{T}y)$ and $\frac{d}{dt}(\mathrm{T}y)>0$. Consequentially, $\mu_m(t) \le 0$ and $I({t})\le m$ on $[t_m,t_m+\varepsilon)$ and we can iterate this argument 
to obtain $\frac{d}{dt}(\mathrm{T}y)(t)<0$
for all $t\ge t_m$. Moreover, the decay of $\mathrm{T}y$ implies $R = \mathcal{U}(\mathrm{T}y)$ 
after some finite time and thus the strict monotone 
decay of $R(t)$ and $\mu_k(t) < 0$ for $m\leq k <\infty$ after some finite time.
As a consequence, the solution $y(t,x)$ decays (at least) exponentially to zero.

Case II:	If Case I does not apply then $I(t)>m$ holds and the exponential grow of the Fourier coefficient $y^{(m)}(t)$ 
versus the exponential decay of $y^{(k)}(t)$ for $k\ge I(t)$ 
implies the existence of a time $t_0>0$ such that $\frac{d}{dt}(\mathrm{T}y)(t_0)=0$ as well as $\frac{d}{dt}(\mathrm{T}y)(t)>0$ 
and $\dot{R}(t)\geq 0$ a.e. on some time interval $t\in (t_0,t_0+\varepsilon)$. The latter implies that the 
$\mu_k(t)$ and thus the monotonicity index $I(t)$ are non-decreasing in time. Hence we can iterate this argument 
and obtain 
$\frac{d}{dt}(\mathrm{T}y)(t)>0$ for all $t>t_0$
and $\dot{R}(t)\geq 0$ for a.e. $t>t_0$.
In particular, $R(t)$ remains constant equal to $R(t)=R(t_0)=\mathcal{U}(\mathrm{T}y(t_0))$ for all $t\geq t_0$ 
until a time $t_->t_0$ when $R(t_-)=\mathcal{L}(\mathrm{T}y(t_-))$.
Afterwards, 
for $t\geq t_-$, $R(t)$ increases according to $R(t)=\mathcal{L}(\mathrm{T}y(t))$. 
It follows that $I(t)>M$ after some finite time and all Fourier modes $y^{(k)}(t)$, $m\leq k \le M$ grow exponentially. 
However, the main contribution to the solution $y$ is again given by $y^{(m)}(t) \phi_m(x)$.
%
%
\end{proof}

\subsubsection*{Qualitative analysis of Figures~\ref{Fig:phase_diag} and \ref{Fig:solution}}\hfill\\
In the view of Proposition~\ref{dicho}, Figure~\ref{Fig:phase_diag} can now be interpreted more explicitly:
After the identical initial decay of $(\mathrm{T}y_{vex},R_{vex})$ and $(\mathrm{T}y_{cave},R_{cave})$
until hitting the upper boundary $\mathcal{U}$ at $t_+=0.15$, it is the different decay of 
$R_{cave}=\mathcal{U}_{cave}(\mathrm{T}y_{cave})$ and  $R_{vex}=\mathcal{U}_{vex}(\mathrm{T}y_{vex})$,
which yields that $R_{cave}(t_m)= D\lambda_1 = 1$ at some time $t_m>t_+$ at which the monotonicity index 
satisfies $I(t_m)=m$ and Case I in Proposition~\ref{dicho} applies to $(\mathrm{T}y_{cave},R_{cave})$.

On the other hand, $\mathrm{T}y_{vex}$ has a turning point at $t_0= 1.34$ and starts to increase again after $t_0$ with $R_{vex}(t)=\mathcal{U}	_{vex}(\mathrm{T}y_{vex}(t_0))$ constant. Note that the plot shows $R_{vex}(t)>1$  and thus $I(t)>m$ for all $t\geq 0$. 
Therefore,  $\mathrm{T}y_{vex}$ grows monotone as discussed in Case II in Proposition~\ref{dicho}.
At time $t_-=5.31$,  $(\mathrm{T}y_{vex},R_{vex})$ hits the graph of $\mathcal{L}$ and $\mathrm{T}y_{vex}$ increases further for $t\geq t_-$ with $R_{vex}(t) = \mathcal{L}(\mathrm{T}y_{vex}(t))$.

\subsection{Generalisations}\hfill\\
Under an additional assumption on $\mathcal{U},\mathcal{L}$ we can extend the qualitative analysis of Case I and Case II in Proposition \ref{dicho} to situations where $\mathrm{T}y$ is initially increasing.

\begin{corollary}[Example with initially increasing $\mathrm{T}y$]\label{Rem:growth_first}\hfill\\
Assume the boundary curves of the 
generalised play operator \eqref{Eq:Hyst_op}
to be defined by $\mathcal{U}$ being a strictly monotone increasing function with $\mathcal{U}(0)>0$, 
which is point symmetric at $(0,\mathcal{U}(0))$ and set $\mathcal{L} = \mathcal{U} - 2\mathcal{U}(0)$. 

Then,  the following point symmetry holds:
Replace $k_m,k_M>0$ and $k_{m+1}, \ldots, k_{M-1} \geq 0$ in \eqref{Eq:W} by $k_m,k_M<0$ and $k_{m+1}, \ldots, k_{M-1} \leq 0$ to obtain an operator $\widetilde{\mathrm{T}}$. With $R=R(\mathrm{T}y)$ where $y$ solves \eqref{Eq:Reac_Diff_Syst}, we obtain $R(\mathrm{T}y)=-R(\widetilde{\mathrm{T}}y)=:-\tilde{R}$.
Consider the modified evolution problem
\begin{equation}\label{Eq:Reac_Diff_Syst2}
\begin{aligned}
&\partial_t y -D\Delta y = -\tilde{R} y 	&&\qquad \text{on } \quad \Omega\times (0,\infty),\\ 
&\partial_\nu y = 0				&&\qquad \text{on } \quad \partial\Omega\times (0,\infty),\\ 
&y(0) = y_0&&\qquad \text{on }\quad\Omega.
\end{aligned}
\end{equation}

Then, the new solution $\tilde{y}$ of \eqref{Eq:Reac_Diff_Syst2} equals the solution $y$ of \eqref{Eq:Reac_Diff_Syst} with $k_m,k_M>0$ and $k_{m+1}, \ldots, k_{M-1} \geq 0$. 
But now, $\widetilde{\mathrm{T}}\tilde{y}<0$ is initially increasing.
	
Moreover, in the phase-space diagram of $\widetilde{\mathrm{T}}\tilde{y}$ and $\tilde{R}$, the new solution $(\widetilde{\mathrm{T}}\tilde{y},\tilde{R})$ is obtained by reflecting the old solution $(\mathrm{T}y,R)$ at the origin. In Case I, $\widetilde{\mathrm{T}}\tilde{y}<0$ is always increasing so that $\tilde{y}$ converges to zero. In Case II, $\widetilde{\mathrm{T}}\tilde{y}$ is increasing until $t_0$, and then decreasing, which leads to a grow-up of $-\tilde{y}=|\tilde{y}|$.
\end{corollary}

\subsubsection{General influence of slope and curvature}\hfill\\

The example in Section~\ref{dichoto}
is constructed in such a way that the difference in the curvature of $\mathcal{U}$
is responsible for spatial homogenisation versus unbounded grow-up of solutions to \eqref{Eq:Reac_Diff_Syst}--\eqref{y0}.
However, analog examples can be constructed with different slopes of $\mathcal{U}$
deciding the large-time behaviour of solutions. 
The corresponding evolution of $y$ depends mainly on the following two properties of $\mathcal{U}$:
\begin{itemize}
	\item Is $\mathcal{U}(z)$ convex, linear or concave near $z = \mathrm{T}y(t_+)$?
	\item What is the slope $\mathcal{U}'(z)$ at $z = \mathrm{T}y(t_+)$ (or near $z = \mathrm{T}y(t_+)$, if $\mathcal{U}'(\mathrm{T}y(t_+))$ is not defined)?
\end{itemize}


\subsubsection*{Discussion} Assume analog to Figures~\ref{Fig:phase_diag} and \ref{Fig:solution} 
that the evolution of various solutions $y$, $\mathrm{T}y$ and $R$ is identical for $t\in [0,t_+]$ independently of 
the corresponding considered curve $\mathcal{U}$.
After $t_+$, the solutions $(\mathrm{T}y,R)$ slide in phase-space diagram along the graphs of $\mathcal{U}$ 
at least for a short distance as long as $\mathrm{T}y$ decreases, no matter if Case I or Case II applies:



\begin{description}
\item[Steep $\mathcal{U}$]
If $\mathcal{U}(z)$ is "steep" near $z = \mathrm{T}y(t_+)$, then $R(t)$ decreases ''fast" compared to  $\mathrm{T}y(t)$ for $t>t_+$.
Hence, as long as $\frac{d}{dt}(\mathrm{T}y)(t) = \sum_{i=m}^{M} k_i \mu_i(t) y^{(i)}(t) <0$ with $R(t)=\mathcal{U}(\mathrm{T}y(t))$, 
also $\mu_k(t)$, $m\leq k \leq M$ decreases "fast" compared to the evolution of Fourier coefficients $y^{(k)}$.
If this decay happens sufficiently fast for given initial data, we will find 
$I({t_m})=m$ for some $t_m>t_+$ and thus a behaviour as in Case I of Proposition~\ref{dicho} and $y$ will converge to zero.
	
\item[Flat $\mathcal{U}$]
On the other hand, if $\mathcal{U}(z)$ is "flat" near $z = \mathrm{T}y(t_+)$, $R(t)$ decreases "slowly" compared to $\mathrm{T}y$ for $t>t_+$.
Hence, as long as $\frac{d}{dt}(\mathrm{T}y)(t) = \sum_{i=m}^{M} k_i \mu_i(t) y^{(i)}(t) <0$ for $t>t_+$ with $R(t)=\mathcal{U}(\mathrm{T}y(t))$ also $\mu_k(t)$, $m\leq k \leq M$ decreases "slowly" compared to the evolution of Fourier coefficients $y^{(k)}$. 
If this decay happens sufficiently slowly for given initial data, it yields $\frac{d}{dt}(\mathrm{T}y)(t_0) = 0$ while $I({t_0})>m$ for some $t_0>t_+$. Thus, we observe a behaviour as in Case II of Proposition~\ref{dicho} and $y$ will grow unboundedly.

\end{description}

These observations concerning steep and flat $\mathcal{U}$ carry readily over 
to $\mathcal{U}$ having different curvature. In particular, when considering either 
$\mathcal{U}_{vex}$ convex or $\mathcal{U}_{cave}$ concave,
even if $\mathcal{U}_{vex}$ and $\mathcal{U}_{cave}$ have the same slope at $z = \mathrm{T}y(t_+)$, 
we find:
\begin{description}
\item[Strongly convex $\mathcal{U}_{vex}$] A sufficiently convex $\mathcal{U}_{vex}$ will yield $\frac{d}{dt}(\mathrm{T}y)(t_0) = 0$ while $I({t_0})>m$ for some $t_0>t_+$
and thus Case II and unbounded growth. 

\item[Strongly concave $\mathcal{U}_{cave}$] A sufficiently concave $\mathcal{U}_{cave}$ will imply $I({t_m})=m$ for some $t_m>t_+$ and, therefore, $y$ to converge to zero as in Case I.
\end{description}

All those examples can be 
adapted to happen at the lower boundary curve $\mathcal{L}$: 
Consider the point symmetric setting from Corollary~\ref{Rem:growth_first} and the evolution problem \eqref{Eq:Reac_Diff_Syst2}. Suppose convex/concave curves $\mathcal{L}$. Chose initial data and $\mathcal{U}$ in such a way that the first contact $\tilde{R}(t_-) = \mathcal{L}(\mathrm{T}\tilde{y}(t_-))$ is identical for all considered curves $\mathcal{L}$. Moreover, assume the evolution of $\tilde{y}$, $\mathrm{T}\tilde{y}$ and $\tilde{R}$ to be equal for $t\in [0,t_-]$ independently of $\mathcal{L}$.
Then, we find analog to above: 
\begin{description}
\item[Steep $\mathcal{L}$] For sufficiently steep $\mathcal{L}$ and for $\frac{d}{dt}(\mathrm{T}\tilde{y})>0$ such that $\tilde{R}$ increases sufficiently fast compared to the evolution of the Fourier coefficient $\tilde{y}^{(k)}$, the solution $\tilde{y}$ converges to zero according to Case I.
\item[Flat $\mathcal{L}$] 		
For sufficiently flat $\mathcal{L}$ and $\frac{d}{dt}(\mathrm{T}\tilde{y})>0$, $\tilde{R}$ increases slowly compared to the evolution of the Fourier coefficient $\tilde{y}^{(k)}$. This yields Case II and unbounded growth of $|\tilde{y}|$.
\item[Strongly convex $\mathcal{L}_{vex}$] 
If $\mathcal{L}_{vex}$ is sufficiently convex, then $\tilde{y}\to 0$ according to Case I.
\item[Strongly concave $\mathcal{L}_{cave}$] 		
If $\mathcal{L}_{cave}$ is sufficiently concave, then Case II and unbounded growth of $|\tilde{y}|$ take place.
\end{description}

\section{Existence and regularity of the solution $(u_\varepsilon,S_\varepsilon)$}\label{existence}
In this section, we consider arbitrary $\varepsilon>0$ fixed and prove existence and uniqueness of global, strong, nonnegative solutions $(u,S)$ to system \eqref{pop_dnymics_pop_evol1}-\eqref{pop_dnymics_stock_evol2}.
\medskip

For any open set $X$, we denote by $\mathrm{C}^1(\overline{X})$ the space of bounded and uniformly continuous functions on $X$, which have bounded and uniformly continuous derivatives.

%
%
%
%
%
%

\begin{proof}[Proof of Theorem~\ref{Thm:pop_dnymics_existence}]
	We consider the Neumann realization in $\mathrm{L}^2(\Omega)$ of the Laplace operator with domain $\mathrm{dom}(-D\Delta) \doteq \lbrace v\in \mathrm{H}^2(\Omega): \partial_\nu v = 0 \text{ on } \partial\Omega \rbrace$, see e.g. \cite[Introduction]{amann1995linear}.

\noindent\underline{(I) Local existence of non-negative solutions via Banach's fixed point theorem:}
For any $\alpha\in (0,1]$ and sufficiently large $\omega >0$, we introduce the notation 
\begin{equation*}
	A^\alpha := (-D\Delta + \omega)^\alpha : X^\alpha\subset \mathrm{L}^2(\Omega)\rightarrow \mathrm{L}^2(\Omega), 
\end{equation*}
where $X^\alpha:=\mathrm{dom}(A^\alpha)$.	
The embeddings $X^\alpha \hookrightarrow \mathrm{H}^{2\alpha}(\Omega)$ are continuous for $0 < \alpha \leq 1$, so that $X^\alpha \hookrightarrow \mathrm{C}^{0,\beta}(\overline{\Omega})$ for $\alpha > \frac{3}{4}$ and for some $\beta\in (0,1)$ if the dimension $d$ of $\Omega$ is less or equal than $3$, see e.g. \cite[Introduction]{amann1995linear} or \cite{henry,pazy}.
	
Let $\alpha\in (3/4,1]$ be arbitrary but fixed. Then, for $u_1,u_2\in X^\alpha$ and 
\begin{equation*}
N_i:= N(u_i)=\int_\Omega u_i(x) dx
\end{equation*}
we estimate for some constant $c>0$
\begin{align}
	|N_1-N_2| &\leq  \int_\Omega |u_1(x)-u_2(x)| dx \leq c\,|\Omega|\|u_1-u_2\|_{X^\alpha}.\label{Est:N_bound}
\end{align}

	Let $S_{in}\in \mathbb{R}^+$ and $u_{in}\in \mathrm{dom}(-D\Delta)$ with $u_{in} \geq 0$ and $N_{in}>0$ be given and consider the closed ball \[B^\alpha_{\delta}:=\overline{B_{\,\mathbb{R} \times X^\alpha}((S_{in},u_{in}),\delta)}\] for some small $\delta >0$ to be chosen.
	Note that $u_{in}\in X^\alpha$ for any $\alpha\in (0,1]$.
	If $\delta$ is small enough, then \eqref{Est:N_bound} implies $N(u)>0$ for all $(S,u)\in B^\alpha_{\delta}$.
	Moreover, since $\lambda$ is bounded on bounded sets and because $F$ and $\lambda$ are (locally) Lipschitz continuous, there exists a time $T>0$ together with constants $C_0,C_\alpha,L_F=L_F(S_{in},u_{in},T)>0$ and $L_\lambda=L_\lambda(S_{in},u_{in},T)>0$ such that for arbitrary $(S_1,u_1),(S_2,u_2)\in B^\alpha_{\delta}$ and for all $0\leq t_1,t_2 \leq T$, we estimate
%
%
	\begin{align*}
	\|\lambda(S_1&,N_1,F(t_1))u_1-\lambda(S_2,N_2,F(t_2))u_2\|_{{\mathrm{L}^2(\Omega)}}\\
	&\leq |\lambda(S_1,N_1,F(t_1))| \|u_1-u_2\|_{{\mathrm{L}^2(\Omega)}} + |\lambda(S_1,N_1,F(t_1)) - \lambda(S_2,N_2,F(t_2))| \|u_2\|_{{\mathrm{L}^2(\Omega)}}\\
	&\leq  C_0C_\alpha\|u_1-u_2\|_{X^{\alpha}} + L_\lambda (|S_1 - S_2| + |N_1-N_2| + L_F|t_1-t_2|) C_\alpha\|u_2\|_{X^{\alpha}}\\
	&\leq C_\alpha\left(\frac{C_0}{c|\Omega|}+ L_\lambda(2\delta + \|u_{in}\|_{X^{\alpha}})\right) (|S_1 - S_2| + c|\Omega|\|u_1-u_2\|_{X^{\alpha}} + L_F|t_1-t_2|)\\
	&\leq C_1 (|S_1 - S_2| + \|u_1-u_2\|_{X^{\alpha}} + |t_1-t_2|).
	\end{align*}
Note that the above estimate holds equally when replacing the left-hand side norm $\mathrm{L}^2(\Omega)$ by $X^{\alpha}$. 
We remark moreover that if $\lambda$ would be space-dependent but sufficiently smooth, an estimate of the same form can be shown 
for $\mathrm{L}^2(\Omega)$ replaced by $\mathrm{dom}(-D\Delta)$ if $\alpha=1$.
Continuing the proof of Theorem~\ref{Thm:pop_dnymics_existence} for spatially homogeneous $\lambda$, the above estimate 
proves that the mapping $(S,u,t)\rightarrow \lambda(S,N(u),F(t))u$ is Lipschitz continuous from $B^{\alpha}_\delta\times [0,T]$ into ${\mathrm{L}^2(\Omega)}$.

Similarly, we obtain 
\begin{align*}
	|F(t_1) - F(t_1)&N_1\, c(S_1,N_1,F(t_1)) - \left(F(t_2) - F(t_2)N_2\, c(S_2,N_2,F(t_2))\right)|\\
	&\leq |F(t_1) - F(t_2)| |1+ N_1\, c(S_1,N_1,F(t_1))|\\
	& \quad+ |F(t_2)| |N_1\,c(S_1,N_1,F(t_1)) - N_2\, c(S_2,N_2,F(t_2))|\\
	&\leq L_F |t_1-t_2|C_2 + F_{\max}|N_1 - N_2| |c(S_1,N_1,F(t_1))| \\
	&\quad+  F_{\max}|N_2| |c(S_1,N_1,F(t_1)) - N_2\, c(S_2,N_2,F(t_2))|\\
	&\leq C_4(|S_1 - S_2| + \|u_1-u_2\|_{X^{\alpha}} + |t_1-t_2|),
\end{align*}
where we used $|1+ N_1\, c(S_1,N_1,F(t_1))|<C_2$ independently of $(S_1,u_1,t_1)\in B_\delta^\alpha\times [0,T]$, as well as \eqref{Est:N_bound}. W.l.o.g. we can choose $C_4=C_1$ from above.
	
In the following, we set $\alpha=1$, denote $A:=-D\Delta$, extend $\lambda$ and $c$ to arbitrary functions on $\mathbb{R}\times\mathbb{R}\times\mathbb{R}^+$ and reformulate system \eqref{pop_dnymics_pop_evol1}--\eqref{pop_dnymics_stock_evol2} in terms of $x = A u$ (see e.g. \cite[Section 6.3]{pazy}) 
and define the mild-formulation mapping
\begin{equation*}
\begin{cases}
\Phi:\mathrm{C}([0,T];\mathbb{R}\times \mathrm{L}^2(\Omega))\rightarrow \mathrm{C}([0,T];\mathbb{R}\times \mathrm{L}^2(\Omega)),\\[2mm]
\Phi(S,x)(t):=
\begin{pmatrix}
	S_{in} + \frac{1}{\varepsilon}\int_0^t \left[F(\tau) - N(A^{-1}x(\tau))c(S(\tau),N(A^{-1}x(\tau)),F(\tau))\right] d\tau\\[2mm]
	e^{t(-A)}Au_{in} + \int_0^t e^{(t-\tau)(-A)}\lambda(S(\tau),N(A^{-1}x(\tau)),F(\tau))x(\tau) d\tau
\end{pmatrix}.
\end{cases}
\end{equation*}

Next, we introduce the closed set $\Sigma$
\begin{align*}
	\Sigma:=&\left\lbrace (S,x)\in \mathrm{C}([0,T];\mathbb{R}\times \mathrm{L}^2(\Omega)):\,(S(0),x(0))=(S_{in},Au_{in}), \, S\geq 0,\right.\\
	&\left.\hspace{5.3cm} \|(S,x)-(S_{in},Au_{in})\|_{\mathrm{C}([0,T];\mathbb{R}\times \mathrm{L}^2(\Omega))} \leq \delta \right\rbrace.
\end{align*}
Then, for $T,\delta>0$ small enough, the Lipschitz continuity of $N:\mathrm{dom}(-D\Delta)\rightarrow \mathbb{R}$, see \eqref{Est:N_bound}, implies that $N(A^{-1}x)>0$ at every time $t\in [0,T]$ and for all $x\in \Sigma$.
Moreover, the Lipschitz continuity of $(S,u,t)\mapsto \lambda(S,N(u),F(t))u \in \mathrm{dom}(-D\Delta)$ and of $(S,u,t)\mapsto F(t) - c(S,N(u),F(t))N(u) \in \mathbb{R}$ in a neighbourhood of $(S_{in},u_{in},0)$ in $\mathbb{R}\times \mathrm{dom}(-D\Delta) \times [0,T]$, yields (for sufficiently small $T$), that $\Phi$ maps the closed set $\Sigma$ into itself and that $\Phi$ is a contraction.
	
Therefore, Banach's fixed point theorem yields a unique fixed point $(S,x)\in \Sigma$ of $\Phi$.
Note that since $S_{in}\ge0$, the quasi-positivity property of 
\begin{equation*}
	(F - N\,c(S,N,F))_{-} = l(S,N,F)\,S(t)
\end{equation*}
and the local Lipschitz continuity of $(S,u,t)\rightarrow F(t) - N(u(t))\,c(S(t),N(u(t)),F(t))$
ensure that $S\geq 0$ holds on $[0,T]$ independently of $T$.

	
	The function $t\mapsto \lambda(S(t),N(A^{-1}x(t)),F(t))A^{-1}x(t)$ is contained in $\mathrm{C}([0,T];\mathrm{L}^2(\Omega))$.
With some additional work (see e.g. \cite{pazy}) one can show that $t\mapsto  x(t) \in \mathrm{L}^2(\Omega)$ is also 
locally Hölder continuous for $t\in (0,T]$. Hence, also $S$ is locally Lipschitz continuous
and so is the function $t\mapsto \lambda(S(t),N(A^{-1}x(t)),F(t))A^{-1}x(t)\in\mathrm{L}^2(\Omega)$ for $t\in (0,T]$.
	This implies that the linear inhomogeneous problem
	\begin{alignat*}{2}
	\dot{v} +A v &= \lambda(S,N(A^{-1}x),F(\tau))A^{-1}x,\\ 
	v(0) &= u_{in},
	\end{alignat*}
	has a unique solution $v\in \mathrm{C}([0,T];\mathrm{L}^2(\Omega))\cap \mathrm{C}^1((0,T];\mathrm{L}^2(\Omega))$, given by 
	\begin{align*}
	v(t)= e^{t(-A)}u_{in} + \int_0^t e^{(t-\tau)(-A)}\lambda(S,N(A^{-1}x),F(\tau))A^{-1}x d\tau.
	\end{align*}
	Applying $A$ to this equation shows that $v=A^{-1}x$, which implies that the function $v\in \mathrm{C}([0,T];\mathrm{dom}(-D\Delta))\cap \mathrm{C}^1((0,T];\mathrm{L}^2(\Omega))$ solves \eqref{pop_dnymics_pop_evol1}-\eqref{pop_dnymics_pop_evol3}. Since $v$ is unique, this proves that
\[
	(S,u):= (S,v)\in \mathrm{C}([0,T];\mathbb{R}\times \mathrm{dom}(-D\Delta))\cap \mathrm{C}^1((0,T];\mathbb{R}\times \mathrm{L}^2(\Omega))
\]
is the unique local solution of \eqref{pop_dnymics_pop_evol1}-\eqref{pop_dnymics_stock_evol2}. Moreover, $S\geq 0$.
		
We are left to prove that the solution $u$ of \eqref{pop_dnymics_pop_evol1}-\eqref{pop_dnymics_pop_evol3} satisfies $u(t,x)\geq 0$ for all $t\in [0,T]$ and all $x\in \overline{\Omega}$.
Note first that
\[
	|\lambda(S(t),N(t),F(t))| \leq C
\]
uniformly in $t\in [0,T]$ for some $C>0$.
%
Let $\mu< -C$ be chosen arbitrary and introduce the auxiliary function $\tilde{u}= ue^{\mu t}$. This function solves the evolution equation
\begin{equation}\label{uu}
\begin{cases}
\partial_t \tilde{u}(t,x) -D\Delta \tilde{u}(t,x)  = \left(\lambda(S,N,F) + \mu\right)\tilde{u}(t,x) \le 0\ & \text{a.e. in } (0,T)\times \Omega,\\
\partial_\nu \tilde{u}(t,x) = 0 & \text{a.e. in } (0,T)\times\partial\Omega,
\end{cases}
\end{equation}
subject to non-negative initial data $\tilde{u}_{in}\ge0$. Moreover, 
$|\lambda(S(t),N(t),F(t))| + \mu\le C+\mu<0$ uniformly in $[0,T]$.
Hence, by using weak parabolic maximum principle arguments (see e.g. \cite{Chi}), we test 
\eqref{uu} with $\tilde{u}_{-} = \min\{0,\tilde{u}\}$ and obtain  
after integration by parts
\begin{align*}
\frac{d}{dt}\int_{\Omega} \frac{(\tilde{u}_{-})^2}{2} \,dx &\le - D \int_{\Omega} |\nabla (\tilde{u}_{-})|^2 \,dx
+ \left( |\lambda|+\mu \right) \int_{\Omega} (\tilde{u}_{-})^2 \,dx 
\le 0.
\end{align*} 
Hence $(\tilde{u}_{in})_{-}=0$ implies that $\tilde{u}_{-}=0$ a.e. on $\Omega$ for all $t>0$. 
Since $u\in \mathrm{C}([0,T]\times \overline{\Omega})$, this yields
\[
	u(t,x)\geq 0, \qquad \text{for all }t\in [0,T] \text{ and } x\in \overline{\Omega}.
\]

\medskip
	
\noindent\underline{(II) Higher regularity and strong solutions:}
	
Since 
$$
(S,u,t)\mapsto \left(\begin{matrix}
\frac{1}{\varepsilon}(F(t) - N(u)c(S,N(u),F(t))\\
\lambda(S,N(u),F(t))u
\end{matrix}\right)
$$ is locally Lipschitz continuous into $\mathbb{R}\times\mathrm{L}^2(\Omega)$ on a neighbourhood of $(S_{in},u_{in})$ in $\mathbb{R}^+\times X^\alpha$ for any $0<\alpha <1$ and for $t>0$, it follows by classical arguments that $t\mapsto \frac{d}{dt} u(t)$ is in fact locally Hölder continuous into $X^\gamma$ for any $0<\gamma<1$ and $t>0$. As a consequence (see e.g. 	\cite[Theorem 3.5.2]{henry}), it follows
\[
	(S,u)\in \mathrm{C}([0,T];\mathbb{R}\times \mathrm{H}^{2}(\Omega)) \cap \mathrm{C}^1((0,T];\mathbb{R}\times \mathrm{C}^{0,\beta}(\overline{\Omega}))
\]
	

Note that the derivative $\dot{S}$ for $t>0$ is given by 
\[
	\dot{S}(t)=\frac{1}{\varepsilon} (F(t) - N(u(t))\,c(S(t),N(u(t)),F(t))),
\]
and that the right hand side is continuous and bounded also for $t\rightarrow 0$.
Hence $S\in \mathrm{C}^1([0,T])$.
\medskip
	
\noindent\underline{(III) Global existence and lower bound for $N(u)$:}

Global existence of solutions follows from the fact that the nonlinear functions $t\mapsto \lambda(S(t),N(u(t)),F(t))u(t)$ and $t\mapsto F(t) - N(u(t))c(S(t),N(u(t),F(t))$ satisfy at most linear growth estimates along solutions $(S,u)$ of \eqref{pop_dnymics_pop_evol1}-\eqref{pop_dnymics_stock_evol2}, see e.g. \cite[Corollary 3.3.5]{henry}.

 Moreover, we show the existence of a constant $\delta(T) >0$ such that $N(u(t))>\delta(T)$ for all $t\in [0,T]$.	
Let $T>0$ be arbitrary and note that	
$N:=N(u)$ and $S$ solve the system
\begin{alignat*}{2}
	\dot{N} &= \lambda(S,N,F) N,\qquad &&N(0) = N_{in}>0,\\
	\varepsilon \dot{S} &= F - Nc(S,N,F),\qquad
	&&S(0)=S_{in}\geq 0.
\end{alignat*}
Since the function $t\mapsto \lambda(S(t),N(t),F(t))=: g(t)$ is continuous, the solution $N$ can be written as 
%
\[
{N}(t)= N_{in} \exp\left(\int_0^t g(s) \,ds \right).
\] 
This allows to estimate
\[
N(t)\geq N_{in} \exp(-T \|g\|_{C([0,T])})=: \delta(T)>0.
\]
\medskip	

\noindent\underline{(IV) Further regularity and classical solutions:}

For $t\in (0,T)$, we calculate
\begin{align*}
\frac{d}{dt} u(t) 
	&= -e^{t(-A)} A u_{in} - \int_0^t e^{(t-\tau)(-A)}\lambda(S,N(u),F)Au\, d\tau + \lambda(S(t),N(u(t)),F(t)u(t),
\end{align*}
and all functions on the right side are contained in $\mathrm{C}([0,T];\mathrm{L}^2(\Omega))$.
Consequently, $\partial_t u$ is uniformly bounded in $\mathrm{L}^2(\Omega)$, i.e.
$u\in \mathrm{C}^1([0,T];\mathrm{L}^2(\Omega))$.
Moreover, we recall that $t\mapsto \lambda(S(t),N(u(t)),F(t))$ is continuous and $\partial_t u(t),u(t)\in \mathrm{C}^{0,\beta}(\overline{\Omega})$ for all $t\in (0,T]$ from Step~II. 
Hence, for any fixed $t>0$, we define $h(x):=-\partial_t u(t,x) + \lambda(S(t),N(t),F(t))\, u(t,x)$, and $u(t,.)$ solves 
 the equation
\begin{equation}\label{Pop_dynamics_Laplace_eq1}
\begin{cases}
-D\Delta z(x) = h(x)&\quad \text{for } x\in \Omega\\
\partial_\nu z = 0 &\quad \text{for } x\in \partial\Omega. 
\end{cases}
\end{equation}	
Moreover, $h\in \mathrm{C}^{0,\beta}(\overline{\Omega})$ satisfies the solvability condition $\int_{\Omega} h dx =0$ since $u$ solves \eqref{pop_dnymics_pop_evol1}-\eqref{pop_dnymics_pop_evol3}. Thus, 
by \cite[Theorem 3.1]{nardi2013schauder}, problem 
\eqref{Pop_dynamics_Laplace_eq1}
has a unique, normalised solution 
\[
	z\in \mathcal{C}:=\left\{ u\in \mathrm{C}^{2,\beta}(\overline{\Omega}): \int_\Omega u(x) dx={N(u)} = 0 \right\}.
\]	
Moreover, \cite[Theorem 4.1]{nardi2013schauder} yields for a constant $C=C(\Omega,\beta,d)>0$
\[
	\|z \|_{\mathrm{C}^{2,\beta}(\overline{\Omega})} \leq C \|h\|_{\mathrm{C}^{0,\beta}(\overline{\Omega})}.
\]
Because $u(t,.)$ solves \eqref{Pop_dynamics_Laplace_eq1}, the uniqueness of the normalised solution $z\in \mathcal{C}$ implies
\[
z= u(t,.) - \frac{1}{|\Omega|} N(u(t)).
\]
Therefore, the function $u(t,.)$ is contained in $\mathrm{C}^{2,\beta}(\overline{\Omega})$ with
\begin{align*}
\biggl\|u(t,.)- &\frac{N(u(t,.))}{|\Omega|}  \biggr\|_{\mathrm{C}^{2,\beta}(\overline{\Omega})}
\leq C \|h\|_{\mathrm{C}^{0,\beta}(\overline{\Omega})}\\	&= C \left\|-\partial_t u(t,.) + \lambda(S(t),N(t),F(t)) u(t,.)\right\|_{\mathrm{C}^{0,\beta}(\overline{\Omega})}\\
	&\leq C\|-\partial_t u(t,.) \|_{\mathrm{C}^{0,\beta}(\overline{\Omega})} + C\left|\lambda(S(t),N(t),F(t))\right| \|u(t,.)\|_{\mathrm{C}^{0,\beta}(\overline{\Omega})}<\infty.
	\end{align*}
Since the right side is uniformly bound for all $0<t_0\le t \le T$, we conclude that $u\in \mathrm{L}^\infty((t_0,T);\mathrm{C}^{2,\beta}(\overline{\Omega}))$ for any $t_0>0$.
Finally, $u\in \mathrm{C}((0,T];\mathrm{C}^{2,\beta}(\overline{\Omega}))$ follows from a similar estimate and the observations
that $t\mapsto \lambda(S(t),N(t),F(t)) u(t)\in \mathrm{C}([0,T];\mathrm{C}^{0,\beta}(\overline{\Omega}))$, 
$t\mapsto\frac{d}{dt}u(t)\in \mathrm{C}((0,T];\mathrm{C}^{0,\beta}(\overline{\Omega}))$ and
$t\mapsto N(u(t))\in \mathrm{C}([0,T])$.
\end{proof}

\noindent{\bf Aknowledgements.}\hfill

This work is supported by the International Research Training Group IGDK 1754 and NAWI Graz.

\newcommand{\etalchar}[1]{$^{#1}$}

\end{document}